\documentclass[english]{article}
\usepackage[T1]{fontenc}
\usepackage[latin9]{inputenc}
\usepackage{amsmath}
\usepackage{amssymb}
\usepackage{esint}
\usepackage{pgf, tikz}
\usepackage{color}\makeatletter

\providecommand{\tabularnewline}{\\}

\usepackage{pgf}\usepackage{tikz}\usepackage{color}

\usepackage[margin=1.5in]{geometry}\providecommand{\tabularnewline}{\\}

\usepackage{amsfonts}\usepackage{amsthm}\usepackage{color}\usepackage{enumitem}\usepackage{bm}\usepackage{cancel}\usepackage{thmtools}\providecommand{\tabularnewline}{\\}

\usepackage{pgf}\usepackage{tikz}\usepackage{color}

\providecommand{\tabularnewline}{\\}

\date{}
\def\theenumi{\arabic{enumi}}

\def\theenumii{\alph{enumii}}
\def\p@enumii{\theenumi.}

\def\theenumiii{\arabic{enumiii}}
\def\p@enumiii{(\theenumi)(\theenumii)}

\def\p@enumiv{\p@enumiii.\theenumiii}

\newtheorem{theorem}{Theorem}[section]

\newtheorem{assumption}[theorem]{Assumption}

\newtheorem{corollary}[theorem]{Corollary}

\newtheorem{fact}[theorem]{Fact}
\newtheorem{lemma}[theorem]{Lemma}
\newtheorem{problem}[theorem]{Problem}
\newtheorem{proposition}[theorem]{Proposition}

\theoremstyle{definition}
\newtheorem{definition}[theorem]{Definition}
\newtheorem{notation}[theorem]{Notation}
\newtheorem{remark}[theorem]{Remark}
\newtheorem{example}[theorem]{Example}
\newtheorem{examples}[theorem]{Examples}

\newcommand{\e}{\varepsilon}
\newcommand{\G}{\Gamma}

\newcommand{\E}{{\mathbf E}}
\newcommand{\Z}{{\mathbb Z}}
\newcommand{\F}{{\mathbb F}}

\usepackage{babel}

\usepackage{babel}

\usepackage{babel}

\makeatother

\usepackage{babel}
\begin{document}

\title{Speed of random walks, isoperimetry and compression of finitely generated
groups}

\author{Jérémie Brieussel and Tianyi Zheng}
\maketitle
\begin{abstract}
We give a solution to the inverse problem (given a function, find a corresponding group) for large classes of speed, entropy, isoperimetric profile, return probability and $L_p$-compression functions of finitely generated groups of exponential volume growth. For smaller classes, we give solutions among solvable groups. As corollaries, we prove a recent conjecture of Amir on joint evaluation of speed and entropy exponents and we obtain a new proof of the existence of uncountably many pairwise non-quasi-isometric solvable groups, originally due to Cornulier and Tessera. We also obtain a formula relating the $L_p$-compression exponent of a group and its wreath product with the cyclic group for $p$ in $[1,2]$.
\end{abstract}

\section{Introduction}

An important topic in group theory is the description of asymptotic
behaviors of geometric and probabilistic quantities, such as volume
growth, isoperimetric profile, Hilbert and Banach space compression
on the geometric side, and speed, entropy and return probability of
random walks on the probabilistic side. The study of these quantities
falls into three types of questions. First given a group, compute
the associated functions. Secondly the inverse problem : given a function,
find a group with this asymptotic behavior. Thirdly understand the
relationship between these quantities and their interactions with
other topics in group theory, such as amenability, Poisson boundaries,
classification up to quasi-isometry, etc. This paper contributes to
solve the second question for finitely generated groups of exponential
volume growth.

The first solution to an inverse problem for a large class of functions concerned compression gaps for
non-amenable groups. Arzhantseva, Drutu and
Sapir proved that essentially any sublinear function is the upper
bound of a Hilbert compression gap of width $\log^{1+\epsilon}(x)$
of some non-amenable group \cite{Arzhantseva2009}. Their construction does not provide a
solution to the other inverse problems, because non-amenability forces
volume growth to be exponential, speed and entropy growth to be linear and return
probability to decay exponentially.

In the amenable setting, a partial solution to the inverse problem
is known for volume growth, entropy and speed. Bartholdi and Erschler
have proved \cite{Bartholdi2014} that for any regular
function $f(n)$ between $n^{0.77}$ and $n$ there is a group with volume growth
$e^{f(n)}$ up to multiplicative constant in front of the argument. This statement at the level of exponents was first obtained in \cite{Brieussel2014}. For any function
between $\sqrt{n}$ and $n^{\gamma}$ for $\gamma<1$, there is a
group and a finitely supported measure with this entropy up
to multiplicative constant by Amir-Virag \cite{AmirVirag2012}, see also \cite{Brieussel2013} for a statement
with precision $n^{o(1)}$. Amir and Virag also showed that for any
function between $n^{\frac{3}{4}}$ and $n^{\gamma}$, $\gamma<1$,
there is a group and a finitely supported measure with this speed
up to multiplicative constant. These examples are all permutational
wreath products, based on groups acting on rooted trees.

This paper develops the construction of new families of groups, not
related to rooted trees nor permutational wreath products, for which
speed, entropy, return probability, isoperimetric profiles, Hilbert
and some Banach space compression can all be computed.

Before stating our result, let us recall the usual definitions. Let
$\Delta$ be a finitely generated group equipped with a generating
set $T$, $\mu$ be a probability measure on $\Delta$. Let $X_{1},X_{2},\ldots$
be a sequence of i.i.d. random variables with distribution $\mu$,
then $W_{n}=X_{1}\ldots X_{n}$ is the random walk on $\Delta$ with
step distribution $\mu$. Its law is the $n$-fold convolution power
$\mu^{\ast n}$. Its \textbf{speed} (rate of escape) is the expectation
\[
L_{\mu}(n)=\E|W_{n}|_{\Delta}=\sum_{g\in\Delta}\left|g\right|_{\Delta}\mu^{\ast n}(g)
\]
where $|\cdot|_{\Delta}$ is the word distance on the Cayley graph
$(\Delta,T)$. Its \textbf{Shannon entropy} is the quantity 
\[
H_{\mu}(n)=H(W_{n})=-\sum_{x\in\Delta}\mu^{\ast n}(x)\log\mu^{\ast n}(x).
\]
The pair $\left(\Delta,\mu\right)$ has Liouville property if the Avez
asymptotic entropy $h_{\mu}=\lim_{n\to\infty}\frac{H_{\mu}(n)}{n}$
is $0$. By classical work of Avez \cite{Avez1976}, Derrienic \cite{Derriennic1976},
Kaimanovich-Vershik \cite{KV}, this is equivalent to the fact that
all bounded $\mu$-harmonic functions are constant.

The \textbf{return probability} is 
\[
\mathbf{P}\left[W_{2n}=e\right]=\mu^{\ast(2n)}(e)
\]
where $e$ is the neutral element in $\Delta$. The $\ell^{p}$-\textbf{isoperimetric
profile} is defined as 
\[
\Lambda_{p,\Delta,\mu}(v)=\inf\left\{ \frac{\sum_{x,y\in\Delta}\left|f(xy)-f(x)\right|^{p}\mu(y)}{2\sum_{x\in\Delta}\left|f(x)\right|^{p}}:\ f\in\ell^{p}(G),\ 1\le\left|\mbox{supp}(f)\right|\le v\right\} .
\]

The \textbf{compression} of an embedding $\Psi$ of $\Delta$ into
a Banach space $\mathfrak{X}$ is the function 
\[
\rho_{\Psi}(t)=\inf\left\{ \left\Vert \Psi(x)-\Psi(y)\right\Vert _{\mathfrak{X}}:\ \left|x^{-1}y\right|_{\Delta}\ge t\right\} .
\]
The embedding is said to be uniform if $\rho_{\Psi}(t)>0$ for all $t\ge 1$
and equivariant if $\Psi$ is a $1$-cocycle, see Section \ref{sec:Compression-expander}.
The couple of functions $(g_{1},g_{2})$ is an $\mathfrak{X}$\textendash{}\textbf{compression
gap} of $\Delta$ if any $1$-Lipschitz embedding $\varphi:\Delta\to\mathfrak{X}$
satisfies $\rho_{\varphi}(t)\le g_{2}(t)$ for all $t\ge1$ and there
exists a $1$-Lipschitz embedding $\Psi:\Delta\to\mathfrak{X}$ such that
$\rho_{\Psi}(t)\ge g_{1}(t)$ for all $t\ge1$. The $\mathfrak{X}$\textendash{}equivariant
compression gap is defined in the same manner, restricting to equivariant
embeddings. Let $L_{p}=L_{p}([0,1],m)$ be the classical Lebesgue
space. By \cite[Theorem 9.1]{Naor2011}, when $Delta$ is amenable, 
for all $p\ge1$, $(g_{1},g_{2})$ is an $L_{p}$-compression gap of $\Delta$
if and only if it is an equivariant $L_{p}$-compression gap of~$\Delta$.

Among these quantities, the compression gap is obviously independent
of the choice of the measure $\mu$ and up to multiplicative constants,
it is invariant under quasi-isometry. The $\ell^{p}$-isoperimetric
profiles and return probability associated with symmetric probability
measures of finite generating support are also known to be stable
under quasi-isometry, see Pittet-Saloff-Coste \cite{PittetSaloffCoste2000},
but it is an open question regarding speed and entropy.

Denote $f(x)\simeq_{C}g(x)$ if $\frac{1}{C}g\left(x\right)\leq f(x)\leq Cg(x)$
for all $x\geq1$. We write $f(x)\simeq g(x)$ and call $f$ and $g$
equivalent if there exists $C$ with $f(x)\simeq_{C}g(x)$. Write
$\log_{\ast}(x)=\log(x+1)$. 

The groups we construct are diagonal products of lamplighter groups.
Such a group is determined once are given a family of groups $\left\{ \Gamma_{s}\right\} $,
usually finite and satisfying some conditions, and a sequence $(k_{s})$
of \textquotedbl{}scaling factors\textquotedbl{} -see Section \ref{section-metric}.
The case where the groups $\Gamma_{s}$ are expanders permits to show
our main result.

\begin{theorem}\label{main-expanders} There exists a universal constant
$C>1$. For any non-decreasing function $f:[1,\infty)\to[1,\infty)$
such that $f(1)=1$ and $x/f(x)$ is non-decreasing, there exists
a group $\Delta$ of exponential volume growth equipped with a finite
generating set $T$ and a finitely-supported symmetric probability
measure $\mu$ such that 
\begin{itemize}
\item the speed and entropy are $L_{\mu}(n)\simeq_{C}H_{\mu}(n)\simeq_{C}\sqrt{n}f(\sqrt{n})$, 
\item the $\ell^{p}$-isoperimetric profile $\Lambda_{p,\Delta,\mu}(v)\simeq_{C}\left(\frac{f\left(\log(v)\right)}{\log v}\right)^{p}$
for any $p\in[1,2]$, 
\item the return probability $-\log\left(\mu^{\ast(2n)}(e)\right)\simeq_{C}w(n)$,
where $w(n)$ is defined implicitly by $n=\int_{1}^{w(n)}\left(\frac{s}{f(s)}\right)^{2}ds$, 
\item $\left(\frac{1}{C\epsilon}\frac{n/f(n)}{\log_{\ast}^{1+\epsilon}(n/f(n))},C\cdot2^{p}\frac{n}{f(n)}\right)$
is an equivariant $L_{p}$-compression gap for $\Delta$ for any $p>1$,
$\epsilon>0$. 
\end{itemize}
When the function $f$ is not asymptotically linear, i.e. $\lim_{x\to\infty}f(x)/x=0$,
the group \ensuremath{\Delta} can be chosen elementary amenable with asymptotic dimension one,
$(\Delta,\mu)$ has the Liouville property, and the equivariant $L_p$-compression gap
is also valid for $p=1$.

\end{theorem}

Since the constant $C$ is universal, this result is new even when
$f$ is asymptotically linear.

The first statement asserts that any regular function between diffusive
$\sqrt{n}$ and linear $n$ is the speed and entropy function of a
random walk on a group. It improves on Amir-Virag \cite{AmirVirag2012}
by the range between diffusive and $n^{\frac{3}{4}}$ for speed, and
by the range close to linear for speed and entropy. The constant in
Theorem \ref{main-expanders} is universal, whereas the constants
in \cite{AmirVirag2012} diverge when approaching linear behavior.
Moreover, concerning speed and entropy we can find such a group $\Delta$
in the class of $4$-step solvable groups. This is the case when the
groups $\Gamma_{s}$ are usual lamplighters over finite $d$-dimensionnal
lattices with $d\geq3$, see Theorem \ref{possible-speed-linear}
and Example \ref{example-lamplighters}.

Our result regarding possible speed functions should be compared with known constraints on speed functions.
From subadditivity, $L_{\mu}(n+m)\le L_{\mu}(n)+L_{\mu}(n)$ for any convolution
walk on a group $\Delta$. By Lee-Peres \cite{LeePeres2013}, there
is a universal constant $c>0$ such that for any amenable group $\Delta$
equipped with a finite generating set $T$, for any symmetric probability
measure $\mu$ on $G$, $L_{\mu}(n)\ge c\sqrt{p_{\ast}n}$, where
$p_{\ast}=\min_{g\in T}\mu(g)$. On the other hand, we obtain that
any function $g(n)$ such that $\frac{g(n)}{\sqrt{n}}$ and $\frac{n}{g(n)}$
are non-decreasing is equivalent to a speed function.

The third statement can be derived from the $\ell^{2}$-isoperimetric
profile estimate in the second statement via the Coulhon-Grigor'yan
theory \cite{Coulhon1997}, see section \ref{return}. For $p\in[1,2]$
any regular function between constant and $n^{-p}$ is equivalent to $\Lambda_{p,\Delta,\mu}\circ\exp$
for some group $\Delta$, and any regular function between $n^{\frac{1}{3}}$
and linear $n$ is equivalent to $-\log\mu^{\ast n}(e)$. Again, this should be compared with known constraints for isoperimetric profile and return probability
for groups with exponential volume growth. By Coulhon
and Saloff-Coste \cite{Coulhon1993}, for any symmetric probability
measure $\mu$ on $\Delta$, $-\log\mu^{\ast n}(e)\geq cn^{\frac{1}{3}}$
and $\Lambda_{p,\Delta,\mu}\circ\exp(x)\geq c'x^{-p}$, $p\in[1,2]$,
where the constants $c,c'$ depend on the volume growth rate of $\left(\Delta,T\right)$
and $p_{\ast}=\min_{g\in T}\mu(g)$.

From the result on $\ell^{1}$-isoperimetric profile we derive Corollary
\ref{folner} that any sufficiently regular function above exponential
is equivalent to a Følner functions. The result extends \cite[Corollary 1.5]{Olshanskii2013}.
It also answers \cite[Question 5]{Erschler2010} positively that there
exists elementary amenable groups with arbitrarily fast growing Følner
function, while simple random walk on it has the Liouville property. Groups
of subexponential volume growth and arbitrarily large Følner function
were first constructed by Erschler \cite{Erschler2006}.

When $f$ is not asymptotically linear, the fourth statement asserts
that any unbounded non-decreasing sublinaer function $h(n)$  is
equivalent to the upper bound of an equivariant $L_{p}$-compression gap of width
$\log^{1+\epsilon}h(n)$ of an amenable group. Recall that equivariant
and non-equivariant compression are equivalent for amenable groups
\cite{Naor2011}. Regarding non-equivariant compression, it is an amenable analogue to Arzhantseva-Drutu-Sapir
\cite{Arzhantseva2009} and slightly improves on it as the width depends
on the upper bound. It also provides other examples of amenable groups with poor compression, after \cite{Austin2011}, \cite{Olshanskii2013} and \cite{BEcompression}. It also follows easily that any function below $\sqrt{x}$ is the upper bound of an equivariant $L_p$-compression gap of a non-amenable group, simply considering the direct product of an amenable solution with the free group on two generators. The equivariant compression of a non-amenable group is at most $\sqrt{x}$ and this bound is attained for free groups by \cite{Guentner2004}.

In order to obtain the equivariant $L_p$-compression gap with upper bound $x/f(x)$ bounded,
 we actually need to choose the family $\{\Gamma_{s}\}$
among quotients of a Lafforgue lattice with strong Property (T) \cite{Lafforgue2008}.
In this case, we also obtain an upper bound on the compression exponent
of $\Delta$ into any uniformly convex normed space, see Corollary
\ref{compression-expander-upper}.

The relationship between these five quantities is more easily understood
at the level of exponents. The \textbf{exponent} of a function $f$
is $\lim\frac{\log f(n)}{\log n}$ when the limit exists. For a compression
gap of width less than any power, the lower exponent of the upper
bound coincides with the definition of the $\mathfrak{X}$-compression
exponent introduced in Guentner-Kaminker \cite{Guentner2004}, 
\[
\alpha_{\mathfrak{X}}^{\ast}(G)=\sup\left\{ \alpha_{\mathfrak{X}}(\Psi):\ \Psi\mbox{ is a Lipschitz map }G\to\mathfrak{X}\right\} ,
\]
where the compression exponent $\alpha_{\mathfrak{X}}(\Psi)$ of the
map $\Psi:\Delta\rightarrow\mathfrak{X}$ is defined as 
\[
\alpha_{\mathfrak{X}}(\Psi)=\sup\left\{ \alpha\ge0:\ \exists c>0\ \mbox{s.t. }\rho_{\Psi}(t)\ge ct^{\alpha}\ \mbox{for all }t\ge1\right\} .
\]
When $\mathfrak{X}$ is the classical Lebesgue space $L_{p}$, we
write $\alpha_{p}^{\ast}(G)$ for the $L_{p}$-compression exponent.
The equivariant compression exponent $\alpha_{\mathfrak{X}}^{\#}(G)$
is defined similarly, restricting to $G$-equivariant embeddings $\Psi$.
When $G$ is amenable, $\alpha_{p}^{\ast}(G)=\alpha_{p}^{\#}(G)$,
see \cite[Theorem 1.6]{Naor2011}.

The exponents when $\left\{ \Gamma_{s}\right\} $ are expanders or
dihedral are given in Figure \ref{table}.

\begin{figure}
\begin{centering}
\begin{tabular}{|c|c|c|c|c|c|}
\hline 
Exponent of  & $\E\left|W_{n}\right|_{\Delta}$  & $H(W_{n})$  & $-\log\mu^{\ast2n}(e)$  & $\Lambda_{p,\Delta,\mu}\circ\exp$  & $\alpha_{p}^{\#}(\Delta)$\tabularnewline
\hline 
\hline 
$\left\{ \Gamma_{s}\right\} $ \mbox{expanders}  & $\frac{1+\theta}{2+\theta}$  & $\frac{1+\theta}{2+\theta}$  & $\frac{1+\theta}{3+\theta}$  & $\frac{-p}{1+\theta}$  & $\frac{1}{1+\theta}$\tabularnewline
\hline 
$\left\{ \Gamma_{s}\right\} $ \mbox{dihedral}  & $\frac{1+3\theta}{2+4\theta}$  & $\frac{1}{2}$  & $\frac{1}{3}$  & $-p$  & $\max\left\{ \frac{1}{1+\theta},\frac{2}{3}\right\} $\tabularnewline
\hline 
\end{tabular}
\par\end{centering}

\caption{\label{table}Exponents for sequences of parameters $k_{s}=2^{2s}$
and $l_{s}=\mbox{diam}(\Gamma_{s})\simeq2^{2\theta s}$, where $\theta\in(0,\infty)$. The isoperimetric profile and compression exponent are all valid for $p\in[1,2]$. The compression for expanders is valid for $p \in [1,\infty)$.}
\end{figure}

When the quotient groups $\Gamma_{s}$ are dihedral, the resulting
diagonal product $\Delta$ is $3$-step solvable.

\begin{theorem}\label{main-proba-dihedral} Let $\epsilon>0$.There
exists a constant $C>0$. For any function $f$ such that $\frac{f(n)}{\log^{1+\epsilon}(n)\sqrt{n}}$
and $\frac{n^{\frac{3}{4}}}{f(n)}$ are non-decreasing, there is a
$3$-step solvable group $\Delta$ and a finitely-supported symmetric probability measure such that the speed function is
\[
\E\left|W_{n}\right|_{\Delta}\simeq_{C}f(n)
\]
and the entropy and return probability satisfy 
\[
\frac{1}{C}\sqrt{n}\leq H(W_{n}^{\Delta})\leq C\sqrt{n}\log^{2}n\mbox{ and }\frac{1}{C}n^{\frac{1}{3}}\leq-\log\mu^{\ast n}(e)\leq Cn^{\frac{1}{3}}\log^{\frac{4}{3}}n.
\]
\end{theorem}

The factor $\log^{1+\epsilon}n$ is only technical. It follows from
Proposition \ref{dihedral-speed} that there is no gap isolating the
diffusive behavior $\sqrt{n}$, but the analysis is simplified by
this mild hypothesis.

Theorems \ref{main-expanders} and \ref{main-proba-dihedral} permit
to solve a recent conjecture of Amir \cite{Amir2015}.

\begin{corollary}\label{entropy-speed-exponent} For any $\theta\in\left[\frac{1}{2},1\right]$
and $\gamma\in\left[\frac{1}{2},1\right]$ satisfying 
\[
\theta\le\gamma\le\frac{1}{2}(\theta+1),
\]
there exists a finitely generated group $G$ and a symmetric probability
measure $\mu$ of finite support on $G$, such that the random walk
on $G$ with step distribution $\mu$ has entropy exponent $\theta$
and speed exponent $\gamma$. \end{corollary} The case where both
exponents belong to $\left[\frac{3}{4},1\right]$ was treated by Amir
\cite{Amir2015}. Proposition \ref{entropy-speed-function} gives
a precise statement regarding functions rather than exponents.

For the group $\Delta$ constructed with finite dihedral groups that
appears in Theorem \ref{main-proba-dihedral}, we estimate its $L_{p}$-compression
exponent for $p\in[1,\infty)$, see Theorem \ref{cotype2}. Explicit
evaluation of compression exponents yields the following result. It
answers \cite[Question 7.6]{Naor2008} positively within the class
of finitely generated $3$-step solvable groups. With certain choices
of parameters, such groups also provide the first examples of amenable
groups where the $L_{p}$-compression exponent, $p>2$, is strictly
larger than the Hilbert compression exponent.

\begin{theorem} \label{dihedral-exponent}

For any $\frac{2}{3}\leq\alpha\leq1$, there exists a $3$-step solvable
group $\Delta$ such that for any $p\in[1,2]$, 
\[
\alpha_{p}^{\ast}(\Delta)=\alpha_{p}^{\#}(\Delta)=\alpha.
\]
Further, there exists a $3$-step solvable group $\Delta_{1}$ such
that for all $q\in(2,\infty)$, 
\[
\alpha_{q}^{\ast}(\Delta_{1})\ge\frac{3q-4}{4q-5}>\alpha_{2}^{\ast}(\Delta_{1})=\frac{2}{3}.
\]

\end{theorem}

In fact, both speed and compression exponents depend explicitly on
the parameter sequences $(k_{s}),(l_{s})$, see Figure \ref{table}
for the usual choice of parameters.

As a corollary, we obtain a new proof of the following result of Cornulier-Tessera
\cite{Cornulier2013}.

\begin{corollary}[Cornulier-Tessera
\cite{Cornulier2013}]\label{CT} There exists uncountably many pairwise non quasi-isometric
finitely generated $3$-step solvable groups. \end{corollary}

The original proof used asymptotic cones. Our method is completely
different, using compression as quasi-isometry invariant. Theorem \ref{dihedral-exponent} and Corollary \ref{CT} do not hold for $2$-step solvable groups. By Baumslag \cite{Baumslag1972},
any finitely generated metabelian group embeds into a finitely presented
metabelian group, so there are countably many classes of isomorphism
of metabelian groups.

The method we apply to estimate compression exponent of $\Delta$
in the dihedral case is very different from the case with expanders.
To better understand the compression of $\Delta$, which is a diagonal
product of groups $\Gamma_{s}\wr\mathbb{Z}$, a first step is
to evaluate compression exponent of general wreath product $H\wr\mathbb{Z}$. This has been an  object of intensive study \cite{AGS2006},  \cite{Naor2008}, \cite{Austin2009}, \cite{Tessera2011}, \cite{Cornulier2012}. We develop a novel approach and derive the following explicit formula.

\begin{theorem}\label{wreath}

Let $p\in[1,2]$, $H$ be a finitely generated infinite group. Then
the equivariant $L_{p}$-compression exponent of $H\wr\mathbb{Z}$
is 
\[
\alpha_{p}^{\#}(H\wr\mathbb{Z})=\min\left\{ \frac{\alpha_{p}^{\#}(H)}{\alpha_{p}^{\#}(H)+\left(1-\frac{1}{p}\right)},\alpha_{p}^{\#}(H)\right\} .
\]
\end{theorem}

The groups we consider to prove Theorems \ref{main-expanders}, \ref{main-proba-dihedral}
and \ref{dihedral-exponent} are diagonal products of lamplighter
groups. Given a family of groups $\left\{ \Gamma_{s}\right\} $ all
generated by the union of two finite groups $A$ and $B$, a factor
of the diagonal product is the lamplighter group $\Gamma_{s}\wr\Z$
endowed with generating set consisting of the shift, a copy of the
lamp subgroup $A$ at position $0$ and a copy of the lamp subgroup
$B$ at position $k_{s}\in\Z$. The diagonal product is the subgroup
of the direct product generated by the diagonal generating set, see Section \ref{section-metric}. It
is parametrized by the sequence of groups $\{\Gamma_{s}\}$ and the
sequence of \textquotedbl{}scaling factors\textquotedbl{} $(k_{s})$.
When the groups $\{\Gamma_{s}\}$ are chosen among quotients of a
group $\Gamma$, the choice of parameters heuristically permits to
interpolate between $\Gamma\wr\Z$ and the usual wreath product $(A\times B)\wr\Z$.

Metrically, placing the two types of generators apart essentially
has the effect of rescaling the copies of $\Gamma_{s}$ by a factor
$k_{s}$. Moreover, under suitable assumptions, the metric in the
diagonal product is essentially the same as in a direct product.

The construction of diagonal products is reminiscent of the piecewise
automata group of Erschler \cite{Erschler2006} and the groups of
Kassabov-Pak \cite{KassabovPak2011}, which permit to obtain oscillating
or \textquotedbl{}close to non-amenable\textquotedbl{} behaviors,
but where more precise estimates are not known. The diagonal product
with $\{\Gamma_{s}\}$ dihedral groups was introduced in \cite{Brieussel2015}
to obtain the first examples of speed functions between $\sqrt{n}$
and $n^{\frac{3}{4}}$.

\textbf{Organization of the paper:}

The detailed construction of diagonal products is given in Section
\ref{section-metric}. A technical assumption on the family $\{\Gamma_{s}\}$
and a list of examples satisfying it appears in \ref{sub: construction}.
The essential estimate relating the metric of a diagonal product to
that of a direct product is established in \ref{sub: metric} under
the assumption that the sequence $(k_{s})$ is strictly doubling.
In \ref{sub: embedded-metric}, we describe some metric spaces naturally
embedded in the diagonal product. It will be used in sections \ref{sec:Compression-expander}
and \ref{sec:dihedral-compression} on compression.

Section \ref{sec: speed} is devoted to the speed and entropy of random
walks. We first treat in \ref{sub:linear-speed} the case, including
expanders, where the groups $\{\Gamma_{s}\}$ have uniform linear
speed up to diameter. Theorem \ref{possible-speed-linear} gives the
first point of Theorem \ref{main-expanders}. The case of dihedral
groups is studied in \ref{sub:dihedral-speed}, proving two thirds
of Theorem \ref{main-proba-dihedral}. Evaluation of speed and entropy
of diagonal products relies on estimations of traverse time of the simple
random walk on $\Z$, which are recalled in Appendix \ref{app: excursions}.
The joint prescription of speed and entropy of Corollary \ref{entropy-speed-exponent}
is obtained in \ref{sub: joint}. Section \ref{sec: speed} is not
used further in the paper, and a reader not interested in speed or
entropy can omit it.

Isoperimetric profiles and return probability are studied in Section
\ref{sec:isoperimetric-profiles}. The second point of Theorem \ref{main-expanders}
is derived as Theorem \ref{possible-profile}. It is proved together
with Corollary \ref{folner} regarding Følner functions in \ref{sub: profile}.
The third point is derived as Theorem \ref{possible-return} in \ref{return}
using Coulhon-Grigor'yan theory. Dihedral groups are treated in \ref{iso-dihedral}
finishing the proof of Theorem \ref{main-proba-dihedral}. A reader
interested mainly in compression can formally omit Section \ref{sec:isoperimetric-profiles},
though the test functions of Proposition \ref{iso-Delta} will be
used in Section \ref{sec:Compression-expander}.

Obstructions for embeddings into Banach spaces are reviewed in Section
\ref{obstructions}. They are based on Poincaré inequalities on finite
metric spaces embedded in the group. The classical spectral version stated
in \ref{sub:Spectral-method} will be used in Sections \ref{sec:Compression-expander}
and \ref{HwrZ}. Markov type inequalities of \ref{sub:Markov-type-method}
will be used in Sections \ref{HwrZ} and \ref{sec:dihedral-compression}
and the Mendel-Naor metric cotype inequalities presented in \ref{sub:metric-cotype-inequalities}
will be used in Section \ref{sec:dihedral-compression}.

In Section \ref{sec:Compression-expander}, we consider diagonal products
where $\{\Gamma_{s}\}$ are quotients of a Lafforgue lattice with
strong Property $(T)$. We first establish in 6.1 an upper bound on
compression exponent valid in any uniformly convex Banach space, and
then in 6.2 derive the proof of the fourth part of Theorem \ref{main-expanders},
in the form of Theorem \ref{compression-expander}. This is done after
three preliminary steps : first provide an upper bound when all quotients
$\{\Gamma_{s}\}$ are finite, secondly an upper bound when one of
them is the whole group $\Gamma$. Thirdly an explicit $1$-cocycle,
related to isoperimetry, is constructed to get a lower bound.

Section \ref{HwrZ} is devoted to Theorem \ref{wreath}. It requires
none of previous Sections except for Poincaré and Markov type inequalities
of Section \ref{obstructions}, but it uses several facts about stable
random walks on lamplighters over a segment, gathered in Appendix
\ref{stable}. It also serves as a warm-up for Section \ref{sec:dihedral-compression}.

The compression of diagonal products with dihedral lamp groups is
studied in Section \ref{sec:dihedral-compression}. Theorem \ref{dihedral-exponent} is proved there,
as well as some explicit bounds for $L_{p}$-compression $p>2$, 
stated in Theorem \ref{cotype2}. As before, we first derive some upper bound using metric
cotype of \ref{sub:metric-cotype-inequalities} and Markov type inequalities
of \ref{sub:Markov-type-method}, then describe an explicit embedding
into $L_{q}$ $q\geq2$. Section \ref{sec:dihedral-compression} formally
uses only Sections \ref{section-metric} and \ref{obstructions},
but is best understood reading also Sections \ref{sec:Compression-expander}
and \ref{HwrZ}.

Finally we point out a few open questions in Section \ref{problems}.
Appendix \ref{approximation} explains a natural approximation of
regular functions by piecewise constant and linear functions. It is
used repeatedly to prove Theorem \ref{main-expanders} in Sections
\ref{sec: speed}, \ref{sec:isoperimetric-profiles} and \ref{sec:Compression-expander}.

{\bf Acknowledgment.} We thank Goulnara Arzhantseva and Yves Cornulier for interesting comments as well as Masato Mimura for a simplification in example \ref{examples} and pointing out the LEF property of Fact \ref{EA}.

\section{The construction and metric structure\label{section-metric}}

\subsection{The construction with diagonal product\label{sub: construction}}

The wreath product of a group $\Gamma$ with $\Z$ is the group $\Gamma\wr\Z=\Gamma^{(\Z)}\rtimes\Z$
where $\Gamma^{(\Z)}$ is the set of functions $f:\Z\rightarrow\Gamma$
with finite $\mbox{support}(f)=\left\{ j\in\Z:\ f(j)\neq e_{\Gamma}\right\} $.
An element is represented by a pair $(f,i)$. We refer to $f$ as
the lamp configuration and $i$ as the position of the cursor. The
product rule is 
\[
(f,i)(g,j)=(f(\cdot)g(\cdot-i),i+j).
\]
The neutral element is denoted as $\left(\boldsymbol{e},0\right)$ where $\mbox{support}(\boldsymbol{e})$
is the empty set. For $j\in\Z$ and $\gamma\in\Gamma$, we denote
by $\gamma\delta_{j}$ the function taking value $\gamma$ at $j$
and $e_{\Gamma}$ elsewhere.

Let $A=\{a_{1},\dots,a_{|A|}\}$ and $B=\{b_{1},\dots,b_{|B|}\}$
be two finite groups. Let $\left\{ \Gamma_{s}\right\} $ be a sequence
of groups such that each $\Gamma_{s}$ is marked with a generating
set of the form $A(s)\cup B(s)$ where $A(s)$ and $B(s)$ are finite
subgroups of $\Gamma_{s}$ isomorphic respectively to $A$ and $B$.
We fix the isomorphic identification and write $A(s)=\{a_{1}(s),\dots,a_{|A|}(s)\}$
and similarly for $B(s)$.

Fix a sequence $\left(k_{s}\right)_{s\ge0}$ of strictly increasing
integers. Take the wreath product $\Delta_{s}=\Gamma_{s}\wr\mathbb{Z}$,
and mark it with generating tuple $\mathcal{T}_{s}$ 
\[
\mathcal{T}_{s}=\left(\tau(s),\alpha_{1}(s),\ldots,\alpha_{|A|}(s),\beta_{1}(s),\ldots,\beta_{|B|}(s)\right)
\]
where $\tau(s)=\left(\boldsymbol{e},+1\right)$ and 
\[
\alpha_{i}(s)=\left(a_{i}(s)\delta_{0},0\right),1\leq i\leq|A|,\ \beta_{i}(s)=\left(b_{i}(s)\delta_{k_{s}},0\right),1\leq i\leq|B|
\]
With slight abuse of notation, we use the symbol $\Delta_{s}$ to
denote the marked group. Alternatively, the marked group $\Delta_{s}$
can be viewed as the projection 
\[
\boldsymbol{\pi}_{s}:\mathbf{G}=\mathbb{Z}\ast A\ast B\to\Gamma_{s}\wr\mathbb{Z},
\]
where $\mathbf{G}$ is the free product of $\left\langle \mathbf{\boldsymbol{\tau}}\right\rangle =\Z,\left\langle \boldsymbol{\alpha}_{i},1\leq i\leq|A|\right\rangle \simeq A$
and $\left\langle \boldsymbol{\beta}_{i},1\leq i\leq|B|\right\rangle \simeq B$
and the projection sends $\boldsymbol{\tau}$ to $\tau(s)$, $\boldsymbol{\alpha}_{i}$
to $\alpha_{i}(s)$ and $\boldsymbol{\beta}_{i}$ to $\beta_{i}(s)$.

The \textbf{diagonal product} $\Delta$ of the, possibly finite, sequence
of marked groups $\left\{ \Delta_{s}\right\} $ is the quotient group
$\mathbf{G}/\cap_{s}\ker\left(\boldsymbol{\pi}_{s}\right)$, with
the projection map $\boldsymbol{\pi}:\mathbf{G}\to\Delta$. It is
marked with generating tuple $\mathcal{T}=\left(\tau,\alpha_{1},\ldots,\alpha_{|A|},\beta_{1},\dots,\beta_{|B|}\right)$.
A word in $\mathcal{T}$ represents $e_{\Delta}$ if and only if the
same word in $\mathcal{T}_{s}$ represents $e_{\Delta_{s}}$ for each
$s$. We use $\pi_{s}:\Delta\to\Delta_{s}$ to denote the projection
from $\Delta$ to the component $\Delta_{s}$.

The group $\Delta$ is completely determined once are given the family
of marked groups $\left\{ \Gamma_{s}\right\} $ and the sequence of distances
$(k_{s})$. An element $g$ of $\Delta$ is completely determined
by the family of projections $\pi_{s}(g)=(f_{s},i_{s})$ and one immediately
checks that the projection onto $\Z$ is independent of $s$. Therefore
we write $\left(\left(f_{s}\right),i\right)$ for a typical element
of $\Delta$, where $f_{s}\in\Gamma_{s}^{(\Z)}$ and $i\in\Z$.

\begin{assumption}\label{Gamma_Involution}

Throughout the paper, we assume the following : 
\begin{itemize}
\item $k_{0}=0$ and $\G_{0}=A(0)\times B(0)\simeq A\times B$ 
\item We call $\G_{s}/[A(s),B(s)]^{\Gamma_{s}}$ the relative abelianization
of $\G_{s}$, where $[A(s),B(s)]^{\Gamma_{s}}$ is the normal closure
of the subgroup generated by commutators $[a_{i}(s),b_{j}(s)]$. We
assume that 
\[
\G_{s}/[A(s),B(s)]^{\Gamma_{s}}\simeq A(s)\times B(s)\simeq A\times B.
\]

\end{itemize}
\end{assumption}

The first assumption is mainly for convenience of notations. It follows
easily from Lemma \ref{abelian} below that the marked group $\left(A\times B\right)\wr\Z$
with usual generating set ($k_{0}=0$) is a quotient of $\Delta$
as soon as $(k_{s})$ is unbounded.

The second assumption is non-trivial and restrictive. It requires
that the relative abelianization, which is always a quotient of $A\times B$,
is in fact isomorphic to $A\times B$. As we will see below, we can
find interesting families of groups satisfying Assumption \ref{Gamma_Involution}.

\begin{notation}

Take a family $\left\{ \Gamma_{s}\right\} $ of quotients of an infinite
group $\Gamma$, and parametrize the group $\Gamma_{s}$ by its diameter
$l_{s}=\mbox{diam}(\Gamma_{s})$ with respect to the generating set
$A(s)\cup B(s)$. Taking the value $l_{s}=\infty$ corresponds to
the choice $\Gamma_{s}=\Gamma$, otherwise $l_{s}<\infty$, $\Gamma_{s}$
is a finite quotient group of $\Gamma$. We say that the sequences
$(k_{s}),(l_{s})$ parametrize the diagonal product $\Delta$. Formally
$\left(k_{s}\right)$ can take the value $\infty$, we make the convention
that if $k_{s}=\infty$, then $\Delta_{s}$ is the trivial group $\Delta_{s}=\{e_{\Delta_{s}}\}$.

\end{notation}

In this paper we will take a group $\Gamma$ and a family $\left\{ \Gamma_{s}\right\} $
of quotients of $\Gamma$ from the following list of specific examples.

\begin{example}\label{examples}

The groups $\left\{ \Gamma_{s}\right\} $ can be taken to form a family
of expanders. For example we obtain the following sequence of finite
groups from the Lafforgue super expanders. By Lafforgue \cite{Lafforgue2008},
for any local field $F$, the group $SL(3,F)$ has Property $(T)$
in any uniformly convex Banach space $\mathfrak{X}$. A fortiori taking $F=\F_p[[X^{-1}]]$, this is also the case for the non-uniform lattice $\Lambda=SL(3,\F_p[X])$, generated by the union of the two following finite subgroups:
\[
A =\left\langle \left(\begin{array}{ccc} 1 & 1 & 0 \\ 0 & 1 & 0 \\ 0 & 0 & 1 \end{array}\right), \left(\begin{array}{ccc} 1 & X & 0 \\ 0 & 1 & 0 \\ 0 & 0 & 1 \end{array}\right) \right\rangle \simeq \F_p^2 \textrm{, and } B=\left\langle \left(\begin{array}{ccc} 0 & 1 & 0 \\ 0 & 0 & 1 \\ 1 & 0 & 0 \end{array}\right)\right\rangle \simeq \Z/3\Z.
\]

There also exist positive constants $c,c_{1}$ such that the ``congruence subgroups'' $\Lambda_m=SL(3,\F_p[X]/(X^m-1))$ satisfy $cm-c_{1}\le\log\left|\Lambda_{m}\right|\le cm+c_{1}$ for all $m \geq 1$.

To make sure the second part of Assumption \ref{Gamma_Involution}
holds, we take $\Gamma$ (resp. $\Gamma_{m}$) to be the diagonal
product of $\Lambda$ (resp. $\Lambda_{m}$) with $A\times B$. Since
$\Lambda$ is a quotient group of $\Gamma$ of index at most $|A|\left|B\right|$,
it follows from the hereditary properties (see \cite[Section 1.7]{Bekka2008})
that $\Gamma$ has Property $(T)$ in any uniformly convex Banach
space $\mathfrak{X}$. Since $\left\{ \Gamma_{m}\right\} $ is a sequence
of finite quotient groups of $\Gamma$, by \cite[Proposition 5.2]{Lafforgue2008},
there exists a constant $\delta(\Gamma,\mathfrak{X})>0$ such that for
any function $f:\Gamma_{m}\to\mathfrak{X}$, $m\in\mathbb{N}$, 
\begin{equation}
\frac{1}{\left|\Gamma_{m}\right|^{2}}\sum_{x,y\in\Gamma_{m}}\left\Vert f(x)-f(y)\right\Vert _{\mathfrak{X}}^{2}\le\frac{1}{\delta(\Gamma,\mathfrak{X})}\frac{1}{\left|\Gamma_{m}\right|}\sum_{x\in\Gamma_{m}}\sum_{u\in A(m)\cup B(m)}\left\Vert f(x)-f(xu)\right\Vert _{\mathfrak{X}}^{2}.\label{eq:lafforgue-poincare}
\end{equation}
In particular, $\left\{ \left(\Gamma_{m},A(m)\cup B(m)\right)\right\} $
form a family of expanders in $\ell^{2}$ with spectral gap uniformly
bounded from below by $\delta(\Gamma,\ell^{2})$.

We will refer to this family $\left\{ \Gamma_{m}\right\} $ as the
Lafforgue super expanders, each $\Gamma_{m}$ is marked with generating
set $A(m)\cup B(m)$. Note that by construction 
\[
cm-c_{1}\le\log\left|\Lambda_{m}\right|\le\log\left|\Gamma_{m}\right|\le cm+c_{1}+\log(3p^{2}).
\]
From the inequality (\ref{eq:lafforgue-poincare}), by \cite[Theorem 13.8]{HLW2006}
there exists constant $c_{2}>0$ depending only on $r$ and $\delta(\Gamma,\ell^{2})$
such that the $\ell^{2}$-distortion satisfies 
\[
c_{2}\log\left|\Gamma_{m}\right|\le c_{\ell^{2}}(\Gamma_{m})\le\mbox{diam}(\Gamma_{m})\le r\log\left|\Gamma_{m}\right|,
\]
and by \cite[Corollary 3.5 ]{Arzhantseva2009}, there exists $c_{3}>0$
depending only on $r$ and $\delta(\Gamma,\mathfrak{X})$ such that
\[
c_{3}\log\left|\Gamma_{m}\right|\le c_{\mathfrak{X}}(\Gamma_{m}).
\]
See section \ref{obstructions} for the definition of distortion.
Lafforgue super expanders are a crucial tool to study compression
in arbitrary uniformly convex Banach spaces. 

In most statements of this paper, it is sufficient to use the classical Property (T) rather than its strenghtenings in uniformly convex Banach spaces. Therefore we can also take $\Lambda=EL(3,\F_p(X,Y))$ which has property (T) by Ershov-Jaikin-Zapirain \cite{ErshovJaikin} generated by the union of its finite subgroups
\[
A =\left\langle \left(\begin{array}{ccc} 1 & a & 0 \\ 0 & 1 & 0 \\ 0 & 0 & 1 \end{array}\right), a \in \{1,X,Y\} \right\rangle \simeq \F_p^3 \textrm{, and } B=\left\langle \left(\begin{array}{ccc} 0 & 1 & 0 \\ 0 & 0 & 1 \\ 1 & 0 & 0 \end{array}\right)\right\rangle \simeq \Z/3\Z.
\]
and with ``congruence subgroups'' $EL\left(3,\mathrm{Mat}_{m \times m}(\F_p)\right)\simeq SL(3m,\F_p)$ the subgroup generated by the matrices $e_{i,j}(a)$ for $a \in \mathrm{Mat}_{m \times m}(\F_p)$ and $i\neq j$ in $\{1,2,3\}$, which are identity plus the marix with only non-zero entry $a$ in position $i,j$.
\end{example}

\begin{examples} \label{example-2}

Some choices of families $\left\{ \Gamma_{s}\right\} $ permit to
obtain diagonal products in the class of solvable groups.

1) an obvious choice is to take $\G_{s}=D_{l_{s}}$ a finite dihedral
group of size $2l_{s}$ generated by two involutions with $A=\Z/2\Z$,
$B=\Z/2\Z$ and $\Gamma=D_{\infty}$. Then the diagonal product $\Delta$
is $3$-step solvable.

2) another possibility is to take $\G=\Z/2\Z\wr D_{\infty}^{k}$ the
lamplighter on an ordinary $k$-dimensionnal lattice. We can take
$A=\Z/2\Z\wr(\Z/2\Z)^{k}$ and $B=\Z/2\Z\wr(\Z/2\Z)^{k}$ in the obvious
manner. The quotients $\G_{s}=\Z/2\Z\wr D_{m_{s}}^{k}$ are lamplighter
over an ordinary discrete $k$-dimensionnal torus. The diagonal product
$\Delta$ is $4$-step solvable.

To check that this example satisfies Assumption \ref{Gamma_Involution},
we denote $\pi:D_{m}\rightarrow\Z/2\Z\times\Z/2\Z$ the abelianization
$\pi(x)=a^{\epsilon(x)}b^{\eta(x)}$. We set 
\begin{align*}
\pi_{A}:\Z/2\Z\wr D_{m}^{k} & \rightarrow\Z/2\Z\wr\left(\Z/2\Z\right)^{k}\simeq A\\
(f,x) & \mapsto(f_{A},a_{1}^{\epsilon_{1}(x)}\dots a_{k}^{\epsilon_{k}(x)})
\end{align*}
where 
\[
f_{A}(a_{1}^{\epsilon_{1}}\dots a_{k}^{\epsilon_{k}})=\prod_{x:\ \epsilon_{i}(x)=\epsilon_{i},\forall i}f(x).
\]
It is clear that $[A,B]^{\Gamma_{m}}\subset\mbox{Ker }\pi_{A}$, and
that we can proceed similarly for $B$. Finally, one can check that
$\pi_{A}\times\pi_{B}$ projects onto $A\times B$. \end{examples}

\begin{remark}

The requirement that $\G$ is a quotient of a free product of \textbf{finite}
groups is crucial, but we can generalize our construction to more
than two finite factors, positioned on an arithmetic progression of
common difference $k_{s}$. This is straightforward for metric estimate,
isoperimetry and compression, but it raises some technical questions
regarding random walks, as the walk inherited on the lamps is not
simple anymore. For simplicity, and by lack of relevant examples,
we avoid this generality.

\end{remark}

\subsection{Description of the metric}

\label{sub: metric}

We describe the metric structure of the Cayley graph $(\Delta,\mathcal{T})$.

\subsubsection{Local coincidence in relative abelianizations\label{sub:Abelian}}

We will be able to estimate the metric in our diagonal products for
sequences $(k_{s})$ growing exponentially because different factors
are largely independent. But the diagonal product differs from the usual direct product. For instance,
the lamp configurations of relative abelianizations behave jointly.

\begin{notation}\label{dust}Projection maps under Assumption \ref{Gamma_Involution}.

Let $\theta_{s}:\Gamma_{s}\to\G_{s}/[A_{s},B_{s}]\simeq A(s)\times B(s)$
denote the projection to the relative abelianization. The projection
map extends point-wise to the lamp configuration function $f_{s}:\mathbb{Z}\to\Gamma_{s}$,
\[
\left(\theta_{s}\left(f_{s}\right)\right)(x)=\theta_{s}\left(f_{s}(x)\right).
\]
We call $\theta_{s}(f_{s})$ the lamp configuration of the relative
abelianization.

Let $\theta_{s}^{A}$ and $\theta_{s}^{B}$ denote the compositions
of $\theta_{s}$ with the projection to $A(s)$ and $B(s)$ respectively.
Then we have a decomposition of $\theta_{s}\left(f_{s}\right)$ into
a commutative product of functions 
\[
\theta_{s}\left(f_{s}\right)=\theta_{s}^{A}\left(f_{s}\right)\theta_{s}^{B}\left(f_{s}\right).
\]

\end{notation}

For any element $g=\left(\left(f_{s}\right),i\right)$ in the diagonal
product $\Delta$, all the relative abelianization lamp configurations
are determined by the first one.

\begin{lemma}\label{abelian}

Let $g=\left(\left(f_{s}\right),i\right)$ be an element in the diagonal
product $\Delta$. Then under Assumption \ref{Gamma_Involution},
any one abelianized function $\theta_{s}(f_{s})$ is determined by
$\theta_{0}({f}_{0})=f_{0}$ and vice-versa. More precisely for any
$s$ 
\[
\theta_{s}^{A}(f_{s}(x))=\theta_{0}^{A}(f_{0}(x))\textrm{ and }\theta_{s}^{B}(f_{s}(x))=\theta_{0}^{B}(f_{0}(x-k_{s})).
\]

\end{lemma}

\begin{proof}

By induction on the word length of $g=\left(\left(f_{s}\right),i\right)$.
Multiplying by a generator $\alpha_{j}$, $\theta_{s}^{A}(f_{s}(x))$
is modified exactly at $x=i$, which also modifies accordingly $\theta_{0}^{A}(f_{0}(i))$.
Multiplying by a generator $\beta_{j}$, $\theta_{s}^{B}(f_{s}(x))$
is modified exactly at $x=i+k_{s}$, which also modifies accordingly
$\theta_{0}^{B}(f_{0}(i))$. \end{proof}

\subsubsection{Local finiteness of the diagonal product\label{sub:Local-finiteness}}

Denote $\pi_{\Z}:\mathbf{G}=\Z\ast A\ast B\to\Z$ the projection on
the first factor. 

\begin{definition}\label{range} The \textbf{range} $Range(w)$ of a representative word $w$ of an element in
$\mathbf{G}$ is the collection of all $\pi_{\Z}(w')$ where $w'$
is a prefix of $w$. It is a finite subinterval of $\Z$, the set
of sites visited by the cursor. We will also denote $Range(w)$ its
diameter.

For an element $g$ in $\Delta$ or in $\Delta_{s}$, we denote $Range(g)$
the diameter of a minimal range interval of a word in $\mathbf{G}$
representing it, and $s_{0}(g)$ the maximal integer with $k_{s_{0}(g)}\leq Range(g)$.

\end{definition}

Denote $\Delta_{\leq s}=\mathbf{G}/\cap_{0 \leq s'\leq s}Ker(\pi_{s'})$
the diagonal product of the $s+1$ first factors, and $\pi_{\leq s}:\Delta\rightarrow\Delta_{\leq s}$
the natural projection.

\begin{fact}\label{s_0} Under Assumption \ref{Gamma_Involution},
for any $g\in\mathbf{G}$, the evaluation $\pi(g)$ in $\Delta$ is
determined by $\pi_{\leq s_{0}(g)}(g)$. \end{fact}

\begin{proof} If $k_{s}>Range(g)$, then $f_{s}$ takes values in
the generating set $A(s)\cup B(s)$. By Assumption \ref{Gamma_Involution},
$f_{s}$ is determined by $\theta_{s}({f}_{s})$ so by $\theta_0(f_0)=f_{0}$ using
lemma \ref{abelian}. So all $\pi_{s}(g)$ for $s>s_{0}(g)$ can be
recovered from $\pi_{0}(g)$. \end{proof}

In particular, as the range is bounded above by the length, the above
argument shows that the sequence of marked groups $\left(\Delta_{s},\mathcal{T}_{s}\right)$
converges to $\left(\Delta_{0},\mathcal{T}_{0}\right)$ and the sequence
$(\Delta_{\leq s},\mathcal{T}_{\leq s})$ to $(\Delta,\mathcal{T})$
in the Chabauty topology.

We also record the following.

\begin{fact}\label{EA} Assume $(k_{s})$ is unbounded.
\begin{enumerate}
\item[(i)] If all the groups
in the family $\left\{ \Gamma_{s}\right\} $ are elementary amenable
(e.g. finite), then the diagonal product $\Delta$ is also elementary amenable.
\item[(ii)] If all the groups
in the family $\left\{ \Gamma_{s}\right\} $ are locally embeddable into finite groups
(e.g. finite), then the diagonal product $\Delta$ is also locally embeddable into finite groups.
\item[(iii)] If all the groups
in the family $\left\{ \Gamma_{s}\right\} $ are finite, then the diagonal product $\Delta$ has asymptotic dimension 1.
\end{enumerate} \end{fact}

\begin{proof}
We observe that when $(k_{s})$ is unbounded, $\ker\pi_{\Z}$
is locally included in a finite direct product of copies of the groups $\left\{ \Gamma_{s}\right\} $.
If $G$ is the subgroup generated by elements $g_{1},\dots,g_{k}$
in $\mbox{Ker }\pi_{\Z}$, then by the previous fact $G$ is isomorphic to $\pi_{\leq S}(G)$
where $S=\max_{1\leq i\leq k}s_{0}(g_{i})$. Moreover
in each copy $\Delta_{s}$ with $s\leq S$ each element $g_{i}$ is
described by a function $f_{i,s}:\Z\rightarrow\Gamma_{s}$ with finite
support, and the group law induced is point-wise multiplication. This proves (i) and also (iii) as the asymptotic dimension of an extension is less than that of quotient plus that of kernel by \cite{BellD2006}.

By Fact \ref{s_0}, the ball of radius $R$ in the Cayley graph of $\Delta$ depends only on the $R$-balls in $\Delta_s$ for $s$ below the maximal $s_0$ with $k_{s_0}\leq R$. Each $R$-ball in a group $\Gamma_s$ coincides with the $R$-ball of a finite group $H_s$, so the $R$-ball in $\Delta_s$ coincides with the $R$-ball in $H_s \wr \Z/(4R+1)\Z$. A fortiori, the $R$-ball in $\Delta$ coincides with the $R$-ball of the diagonal product of these $s_0+1$ finite groups. This proves (ii).
\end{proof}

\subsubsection{Metric in one copy $\Delta_{s}$ \label{sub:one-copy}}

In order to estimate the metric in the diagonal product, the sequence
$(k_{s})$ must grow exponentially fast. Therefore we make the

\begin{assumption}\label{k_growth} Throughout the paper, we assume
that the sequence $(k_{s})$ is a sequence of strictly increasing
even numbers such that $k_{s+1}>2k_{s}$ for all $s$. \end{assumption}
Exponential growth of the sequence $(k_{s})$ is crucial for the second
part of Lemma \ref{one_metric} below. We choose the factor $2$ for
simplicity. It could be replaced by any $m_{0}>1$, which would only
modify our estimates by some multiplicative constants.

\begin{definition}\label{essential} For $j\in\Z$, let $I_{j}^{s}=[\frac{jk_{s}}{2},\frac{(j+1)k_{s}}{2})$.
We call \textbf{essential contribution} of the function $f_{s}:\Z\rightarrow\Gamma_{s}$
the quantity 
\[
E_{s}(f_{s})=\sum_{\left\{ j:\ I_{j}^{s}\cap\mbox{Range}(f_{s},i)\neq\emptyset\right\} }k_{s}\max_{x\in I_{j}^{s}}(|f_{s}(x)|_{\Gamma_{s}}-1)_{+}
\]
where $(x)_{+}=\max\{x,0\}$. In words, we partition the range into
intervals of width $\frac{k_{s}}{2}$. Each of these intervals essentially
contributes as $k_{s}$ times the maximal distance $\left|f_{s}(x)\right|_{\Gamma_{s}}-1$.
\end{definition}

The essential contribution measures the contribution of the terms
$f_{s}(x)$ of $\Gamma_{s}$-length more than $2$ to the length of
an element $(f_{s},i)$ of $\Delta_{s}$. The range will take care of
the contribution of terms of length less than $1$.

\begin{lemma}\label{one_metric} Let $(f_{s},i)$ belong to $\Delta_{s}$.
Then 
\[
\max\left\{ \frac{1}{8}E_{s}(f_{s}),\mbox{Range}(f_{s},i)\right\} \leq|(f_{s},i)|_{\Delta_{s}}\leq9\left(E_{s}(f_{s})+Range(f_{s},i)\right).
\]

Let $\Delta$ be the diagonal product of $\{\Delta_{s}\}$. Under
Assumptions \ref{Gamma_Involution} and \ref{k_growth} and if moreover
$\theta_{s}(f_{s})=\boldsymbol{e}$, then there is a word $\omega(f_{s},i)\in\mathbf{G}$
of length less than the above upper bound such that 
\[
\pi_{s'}(\omega(f_{s},i))=\left\{ \begin{array}{ll}
(f_{s},i) & \textrm{ if }s'=s\\
(0,i) & \textrm{ if }s'\neq s.
\end{array}\right.
\]

\end{lemma}

\begin{proof} The lower bound by the range is obvious.

Let $[x]$ denote the integer part of $x$. To get the lower bound
by essential contribution, notice that in order to write $f_{s}(x)$,
the cursor has to traverse at least $\left[|f_{s}(x)|_{\G_{s}}/2\right]$
times between positions in $x-k_{s}$ and $x$ because a minimal representative
word alternates elements of $A(s)$ and $B(s)$. So $x$ contributes
at least $k_{s}\left[|f_{s}(x)|_{\G_{s}}/2\right]$ to the length
of a representative word $w$ of $(f_{s},i)$.

If the intervals $[x-k_{s},x]$ and $[x'-k_{s},x']$ are disjoint,
the contributions of $x$ and $x'$ must add up. Let $x_{j}^{s}$
achieve the maximum of $|f_{s}(x)|_{\Gamma_{s}}$ on the interval
$I_{j}^{s}$. The separation condition is satisfied for a family of
$x_{j}^{s}$ with the same congruence of $j$ modulo $4$. Therefore
\[
|w|\geq\max_{0\leq a\leq3}\left(\sum_{j\equiv a\textrm{ mod }4}k_{s}\left[|f_{s}(x)|_{\G_{s}}/2\right]\right)\geq\frac{1}{4}\sum_{j\in\Z}k_{s}\max_{I_{j}^{s}}\left[|f_{s}(x)|_{\G_{s}}/2\right]\geq\frac{1}{8}E_{s}(f_{s}).
\]

To get the upper bound, the generic strategy is the following. We
partition the convex envelope of $supp(f_{s})\cup\{0,i\}$, of length
less than $Range(f_{s},i)+k_{s}$, into its intersections with the
intervals $I_{j}^{s}$ for $j\in\Z$. The elements of the partition
are still denoted $I_{j}^{s}$. Let $J_{\max}$ (resp. $J_{\min}$)
denote the maximum (resp. minimum) index of this partition, and let
$w(f_{s}(x))$ be a fixed minimal representative words for $f_{s}(x)$.
We produce a representative word for $(f_{s},i)^{-1}$ by the following
strategy.

First apply a power $p_{1}\leq Range(f_{s},i)+k_{s}$ of the shift
$\tau$ to move the cursor from $i$ to the rightmost point of the
interval $I_{J_{\max}}^{s}$. Then for each integer $j$ from $J_{\max}$
to $J_{min}$, produce a word $\omega_{j}$ that, taking the cursor
from the rightmost point of $I_{j}^{s}$, erases all the words $w(f_{s}(x))$
for $x\in I_{j}^{s}$ and eventually leaves the cursor at the rightmost
point of $I_{j-1}^{s}$.

The description of $\omega_j$ is the following. The first run takes the cursor to $j\frac{k_{s}}{2}-k_{s}$ and then
back, so that all the last letters of $f_{s}(x)$ for $x\in I_{j}^{s}$
can be deleted. More precisely, while the cursor is in $[j\frac{k_{s}}{2},(j+1)\frac{k_{s}}{2})=I_{j}^{s}$
multiplying by $\alpha_{l}(s)$ at appropriate locations removes the
last letter to those words $w(f_{s}(x))$ ending with $a_{l}(s)$
and while the cursor is in $[j\frac{k_{s}}{2}-k_{s},(j+1)\frac{k_{s}}{2}-k_{s})=I_{j-2}^{s}$
multiplying by $\beta_{l}(s)$ at appropriate locations removes the
last letter to those words $w(f_{s}(x))$ ending with $b_{l}(s)$.
One run has length $3k_{s}$ and cancels one letter in each of the
words. The number of runs necessary to erase completely all the words
is $\max_{x\in I_{j}^{s}}|f_{s}(x)|_{\G_{s}}$. A last run takes the
cursor from the rightmost point of $I_{j}^{s}$ to the rightmost point
of $I_{j-1}^{s}$, except for $j=J_{\min}$. Thus $|\omega_{j}|\leq3k_{s}\max_{x\in I_{j}^{s}}|f_{s}(x)|_{\G_{s}}+\max(I_{j}^{s})-\max(I_{j-1}^{s})$
without last term for $j=J_{\min}$.

Finally apply a power $p_{2}\leq Range(f_{s},i)+k_{s}$ of the shift
$\tau$ to move the cursor from $\max(I_{J_{\min}}^{s})$ to $0$.
All in all 
\[
\omega(f_{s},i)=(\tau^{p_{1}}\omega_{J_{\max}}\dots\omega_{J_{\min}}\tau^{p_{2}})^{-1}
\]
is a representative word of $(f_{s},i)$ with length 
\begin{eqnarray*}
|\omega(f_{s},i)| & \leq & \sum_{j\in\Z}3k_{s}\max_{I_{j}^{s}}|f_{s}(x)|_{\Gamma_{s}}+\sum_{j=J_{\min}}^{J_{\max}}\max(I_{j}^{s})-\max(I_{j-1}^{s})+2Range(f_{s},i)+2k_{s}\\
 & \leq & 3\sum_{j\in\Z}k_{s}\max_{I_{j}^{s}}|f_{s}(x)|_{\Gamma_{s}}+3Range(f_{s},i)+2k_{s}.
\end{eqnarray*}

Now the number of indices $j$ such that $I_{j}^{s}$ intersects the
range of $(f_{s},i)$ is less than $\left[2Range(f_{s},i)/k_{s}\right]+1$.
Therefore 
\[
\sum_{j\in\Z}k_{s}\max_{I_{j}^{s}}|f_{s}(x)|_{\Gamma_{s}}\leq E_{s}(f_{s})+2Range(f_{s},i)+k_{s}.
\]

Generically, $E_{s}(f_{s})\neq0$ and then $k_{s} \leq E_{s}(f_{s})$
and the two previous inequalities give the upper bound. In the non-generic case when $E_{s}(f_{s})=0$,
then $\left|f_{s}(x)\right|_{\Gamma_{s}}\leq1$ for all $x$ and an
obvious bound is $\left|(f_{s},i)\right|_{\Delta_{s}}\leq3Range(f_{s},i)$.

To get the second part of the lemma, observe first that if $E_{s}(f_{s})=0$
and $\theta_{s}(f_{s})=\boldsymbol{e}$ then $(f_{s},i)=(\boldsymbol{e},i)$ is just a translation,
the conclusion holds trivially. In the generic case, we have to check
that for each sub-word $\omega_{j}$ in the above procedure the lamp
function $f_{s'}^{\omega_{j}}$ of $\pi_{s'}(\omega_{j})$ is trivial.

First assume $s'>s$. The cursor moves in the interval $I=[j\frac{k_{s}}{2}-k_{s},(j+1)\frac{k_{s}}{2}]$
of length $\frac{3}{2}k_{s}<2k_{s}<k_{s'}$ by Assumption \ref{k_growth}.
In this condition, multiplying by $\alpha_{l}$ contributes to $f_{s'}^{\omega_{j}}(x)$
by $a_{l}(s')$ at positions $x\in I$ and multiplying by $\beta_{l}$
contributes to $f_{s'}^{\omega_{j}}(x)$ by $b_{l}(s')$ at positions
$x\in I+k_{s'}$. These intervals are disjoint, so $f_{s'}^{\omega_{j}}$
takes values in the generating set of $\G_{s}$. Thus $f_{s'}^{\omega_{j}}(x)=\theta_{s'}(f_{s'}^{\omega_{j}})(x)=\theta_{s'}^{A}(f_{s'}^{\omega_{j}})(x)\theta_{s'}^{B}(f_{s'}^{\omega_{j}})(x)=e_{\Gamma_{s}}$,
because $\theta_{s'}^{A}(f_{s'}^{\omega_{j}})(x)=\theta_{0}^{A}(f_{0}^{\omega_{j}})(x)=e$
and $\theta_{s'}^{B}(f_{s'}^{\omega_{j}})(x)=\theta_{0}^{B}(f_{0}^{\omega_{j}})(x+k_{s}-k_{s'})=e$
using lemma \ref{abelian} and our hypothesis.

Now assume $s'<s$. The generators $a_{l}(s')$ were applied only
when the cursor $i$ was in $I_{j}^{s}$. On the other hand the generators
$\beta_{l}$ were applied only at points $i+k_{s}\in I_{j}^{s}$,
that is when the cursor $i$ was in $I_{j-2}^{s}$. Then, as $k_{s'}\in[0,\frac{k_{s}}{2})$,
by assumption \ref{k_growth}, the element $b_{l}(s')$ is applied
only at locations $x=i+k_{s'}\in I_{j-2}^{s}+[0,\frac{k_{s}}{2})\subset I_{j-2}^{s}\cup I_{j-1}^{s}$.
As the latter set is disjoint from $I_{j}^{s}$, the function $f_{s'}^{\omega_{j}}$
again takes values in the generating set of $\G_{s}$. We conclude
as above. \end{proof}

\subsubsection{Description of the metric in the diagonal product}

Now we are ready to estimate metric in the diagonal product $\Delta$.

\begin{proposition}\label{metric} Suppose the sequence $\left\{ \Gamma_{s}\right\} $
of marked groups satisfies Assumption \ref{Gamma_Involution}, and
the sequence $(k_{s})$ of integers satisfies Assumption \ref{k_growth}.
For any element $g=\left(\left(f_{s}\right),i\right)$ in the diagonal
product $\Delta$, the word distance in $\left(\Delta,\mathcal{T}\right)$
satisfies 
\[
\max_{s\ge0}\left|\left(f_{s},i\right)\right|_{\Delta_{s}}\le\left|g\right|_{\Delta}\le500\sum_{0\leq s\leq s_{0}(Z)}\left|\left(f_{s},i\right)\right|_{\Delta_{s}}.
\]
\end{proposition}

\begin{proof} The first inequality holds because $\Delta_{s}$ is
a marked quotient of $\Delta$.

For the second inequality, let $\omega(f_{0},i)$ be a minimal representative
word of $(f_{0},i)$. This is realized when the cursor moves across
the range and at each site switches appropriately the $A$ and $B$
lamps. In particular, the word $\omega(f_{s},i)$ has length $\left|(f_{0},i)\right|_{\Delta_{0}}\leq3Range(f_{0},i)$.
It represents an element $((h_{s}),i)$ in $\Delta$ with $h_{0}=f_{0}$
and $\left|h_{s}(x)\right|_{\Gamma_{s}}\leq2$ for all $x$ and $s$.

Then $g\omega(f_{0},i)^{-1}=\left((f_{s}'),0\right)$ with $f_{s}'=f_{s}h_{s}^{-1}$
for all $s$. In particular $f_{0}'=\boldsymbol{e}$, thus $\theta_{s}(f'_{s})=\boldsymbol{e}$
for all $s$ by Lemma \ref{abelian}. Lemma \ref{one_metric} applies
and we get 
\[
g\omega(f_{0},i)^{-1}\omega(f_{1}',0)^{-1}\dots\omega(f_{s_{0}(g)}',0)^{-1}=e_{}.
\]
This is true in $\Delta_{\leq s_{0}(g)}$ hence in $\Delta$ by Fact
\ref{s_0}. This shows that 
\[
\left|g\right|_{\Delta}\leq\left|\omega(f_{0},i)\right|+\sum_{1\leq s\leq s_{0}(g)}\left|\omega(f_{s}',0)\right|.
\]
We claim that $\mbox{support}(f_{s}')\subset\mbox{support}(f_{s})$.
This implies $Range(f_{s}',0)\leq2Range(f_{s},i)$. Moreover $E_{s}(f_{s}')\leq3E_{s}(f_{s})$
because $\left|f_{s}'(x)\right|_{\Gamma_{s}}\leq\left|f_{s}(x)\right|_{\Gamma_{s}}+2$
for all $x$ and $s$. Therefore we can conclude using Lemma \ref{one_metric}
that 
\begin{align*}
\left|\omega(f_{s}',0)\right| & \leq9(E_{s}(f_{s}')+Range(f_{s}',0))\leq27(E_{s}(f_{s})+Range(f_{s},i))\\
 & \leq54\max(E_{s}(f_{s}),Range(f_{s},i))\leq432\left|\left(f_{s},i\right)\right|_{\Delta_{s}}.
\end{align*}
The claim follows from Lemma \ref{abelian}. Indeed, if $f_{s}(x)=e_{\Gamma_{s}}$,
then in particular 
\[
e_{A}=\theta_{s}^{A}(f_{s}(x))=\theta_{0}^{A}(f_{0}(x))=\theta_{0}^{A}(h_{0}(x))=\theta_{s}^{A}(h_{s}(x))
\]
and 
\[
e_{B}=\theta_{s}^{B}(f_{s}(x))=\theta_{0}^{B}(f_{0}(x-k_{s}))=\theta_{0}^{B}(h_{0}(x-k_{s}))=\theta_{s}^{B}(h_{s}(x))
\]
so $\theta_{s}(h_{s}(x))=e_{A\times B}$. As $\left|h_{s}(x)\right|_{\Gamma_{s}}\leq2$,
this implies $h_{s}(x)=e_{\Gamma_{s}}$. Therefore $f_{s}'(x)=e_{\Gamma_{s}}$
as well. \end{proof}

\subsection{Metric spaces embedded in $\Delta$\label{sub: embedded-metric}}

In this section, we gather some elementary facts about embeddings
of some metric spaces into the diagonal product $\Delta$. It will
be used to obtain upper bounds on compression, in sections \ref{sec:Compression-expander}
and \ref{sec:dihedral-compression}.

\subsubsection{Embedding a lamp group $\Gamma_{s}$}

\begin{fact}

\label{embed-one}

Each group $\Gamma_{s}$ embeds homothetically in the diagonal product
$\Delta$ with ratio $k_{s}+1$, i.e. there is group homomorphism
$\vartheta_{s}:\Gamma_{s}\rightarrow\Delta$ satisfying 
\[
\left|\vartheta_{s}(\gamma)\right|_{\Delta}=(k_{s}+1)\left|\gamma\right|_{\Gamma_{s}}.
\]
\end{fact}

\begin{proof} Let $w=a_{i_{1}}(s)b_{i_{1}}(s)\dots a_{i_{n}}(s)b_{i_{n}}(s)$
belong to $A(s)\ast B(s)$. Set 
\[
\vartheta_{s}(w)=\tau^{\frac{k_{s}}{2}}\alpha_{i_{1}}\tau^{-k_{s}}\beta_{i_{1}}\tau^{k_{s}}\dots\alpha_{i_{n}}\tau^{-k_{s}}\beta_{i_{n}}\tau^{k_{s}}.
\]
The application $\vartheta_{s}:A(s)\ast B(s)\rightarrow\mathbf{G}=\Z\ast A\ast B$
induces an embedding $\vartheta_{s}:\Gamma_{s}\rightarrow\Delta$.
Indeed by Lemma \ref{abelian}, we easily check that if $w$ represents
an element $\gamma$ in $\Gamma_{s}$, then $\vartheta_{s}(w)=\left(\left(f_{s}\right),0\right)$
with 
\[
f_{s}(x)=\left\{ \begin{array}{ll}
\gamma & \mbox{for }x=\frac{k_{s}}{2}\\
e_{\Gamma_{s}} & \mbox{for }x\neq\frac{k_{s}}{2}
\end{array}\right.\mbox{ and }f_{s'}(x)=\left\{ \begin{array}{ll}
\theta_{s}^{A}(\gamma) & \mbox{for }x=\frac{k_{s}}{2}\\
\theta_{s}^{B}(\gamma) & \mbox{for }x=\frac{k_{s}}{2}-k_{s'}\\
e_{\Gamma_{s}} & \mbox{for other }x
\end{array}\right.\mbox{for }s'\neq s.
\]
By construction, $\left|\vartheta_{s}(\gamma)\right|_{\Delta}\leq(k_{s}+1)\left|\gamma\right|_{\Gamma_{s}}$.
Moreover, it is clear that if $w$ is a minimal representative of
$\gamma$, then $\vartheta_{s}(w)$ is a minimal representative of
$\pi_{s}(\vartheta_{s}(\gamma))$ in the quotient $\Delta_{s}$. This
proves the other inequality. \end{proof}

\subsubsection{Embedding products with $\ell^{\infty}$-norm}

Let us denote $\Gamma_{s}'=\left[A(s),B(s)\right]^{\Gamma_{s}}$.
By Assumption \ref{Gamma_Involution}, $\Gamma_{s}'$ is the same
as $\ker(\Gamma_{s}\to A(s)\times B(s))$, it is a subgroup of $\Gamma_{s}$
of finite index $\left|A\right|\left|B\right|$. 

Given an integer $t\geq0$, we consider 
\begin{equation}
\Pi_{s}^{t}=\left\{ \left(\left(f_{s}\right),0\right):\ \begin{array}{ll}
f_{s}(x)\in\Gamma_{s}' & \mbox{for }x\in[0,t)\\
f_{s}(x)=e_{\Gamma_{s}} & \mbox{for }x\notin[0,t)\\
f_{s'}=\boldsymbol{e} & \mbox{for }s'\neq s
\end{array}\right\} \label{eq:product-set}
\end{equation}
This is a subset of $\Delta$. Indeed, $\theta_{s}(f_{s})=\boldsymbol{e}$ by
choice of $\Gamma_{s}'$, so $\theta_{0}(f_{0})=\boldsymbol{e}$ by Lemma \ref{abelian}.
Thus all such elements $\left(\left(f_{s}\right),0\right)$ actually
belong to $\Delta$ by Lemma \ref{one_metric}. Clearly, $\Pi_{s}^{t}$
is a subgroup of $\Delta$ isomorphic to a direct product of $t$
copies of $\Gamma_{s}'$. 
\[
\Pi_{s}^{t}\simeq\prod_{t\in[0,t)}\Gamma_{s}'.
\]
By abuse of notation, we denote the elements of $\Pi_{s}^{t}$ simply
by functions $f_{s}:[0,t)\rightarrow\Gamma_{s}'$. The metric induced
by $\Delta$ on $\Pi_{s}^{t}$ can be estimated via Lemma \ref{one_metric}.

\begin{lemma}

\label{embed-product} For any $f_{s}$ in $\Pi_{s}^{t}$, 
\[
\frac{1}{2}k_{s}\max_{[0,t)}\left|f_{s}(x)\right|_{\Gamma_{s}}\leq\left|f_{s}\right|_{\Delta}\leq36t\max_{[0,t)}\left|f_{s}(x)\right|_{\Gamma_{s}}.
\]
In particular, $\mbox{diam}_{\Delta}\Pi_{s}^{t}\leq36t\mbox{diam}_{\Gamma_{s}}(\Gamma_{s}')\leq36tl_{s}$,
where $l_{s}=\mbox{diam}(\Gamma_{s})$. Moreover, 
\[
\left|\left\{ f_{s}\in\Pi_{s}^{t}:\ \left|f_{s}\right|_{\Delta}\geq\frac{1}{72}\mbox{diam}_{\Delta}\left(\Pi_{s}^{t}\right)\right\} \right|\geq\frac{1}{2}\left|\Pi_{s}^{t}\right|.
\]
\end{lemma} The last statements imply that $\Pi_{s}^{t}$ satisfies
the $(p;\frac{\mbox{diam}_{\Delta}(\Pi_s^t)}{72},\frac{1}{2})$-mass distribution condition
(\ref{eq:(p,c)-mass}), see section \ref{obstructions}.

\begin{proof} This follows from Lemma \ref{one_metric}. To get the
lower bound, notice that as $f_{s}(x)$ belongs to $\Gamma_{s}'$,
we necessarily have $\left|f_{s}(x)\right|_{\Gamma_{s}}\geq2$ when
$f_{s}(x)\neq e_{\Gamma_{s}}$. To get the upper bound, observe that
there are $\left[2t/k_{s}\right]+1$ intervals $I_{j}^{s}$ intersecting
$[0,t)$, so the essential contribution is at most 
\[
E_{s}(f_{s})\leq k_{s}\left(\left[\frac{2t}{k_{s}}\right]+1\right)\max_{[0,t)}\left|f_{s}(x)\right|_{\Gamma_{s}}\leq\left(2t+1\right)\max_{[0,t)}\left|f_{s}(x)\right|_{\Gamma_{s}}
\]
and the range is bounded by $t$.

To get the second part, observe that for more than half
of functions $I_{j}^{s}\rightarrow\Gamma_{s}'$ there exists $x\in I_{j}^{s}$
with $\left|f_{s}(x)\right|_{\Gamma_{s}}\geq\mbox{diam}_{\Gamma_{s}}(\Gamma_{s}')/2$.
This holds for each $j$. Therefore, there exists a subset $A$ of
$\Pi_{s}^{t}$ of size $|A|\geq|\Pi_{s}^{t}|/2$ such that for each
$f_{s}\in A$, more than half of the $\left[2t/k_{s}\right]+1$ intervals
$I_{j}^{s}$ intersecting $[0,t)$ satisfy $\max_{I_{j}^{s}}\left|f_{s}(x)\right|_{\Gamma_{s}}\geq\mbox{diam}_{\Gamma_{s}}(\Gamma_{s}')/2$.
This implies for any $f_{s}\in A$ 
\[
\left|f_{s}\right|_{\Delta}\geq E_{s}(f_{s})\geq\frac{1}{4}\left(\left[\frac{2t}{k_{s}}\right]+1\right)k_{s}\mbox{diam}_{\Gamma_{s}}(\Gamma_{s}')\geq\frac{t}{2}\mbox{diam}_{\Gamma_{s}}(\Gamma_{s}')\geq\frac{1}{72}\mbox{diam}_{\Delta}(\Pi_{s}^{t}).
\]
\end{proof}

\begin{example} When $\Gamma_{s}=D_{2l_{s}}$ is a dihedral group
of size $2l_{s}$, then $\Gamma_{s}'\simeq\Z_{l_{s}/2}$ is a cyclic
group. Edges of $\Z_{l_{s}/2}$ have length $4$ in the $D_{2l_{s}}$
metric. For $t=\frac{k_{s}}{2}$, Lemma \ref{embed-product} gives
\[
\left|f_{s}\right|_{\Delta}\simeq_{72}k_{s}\max_{[0,\frac{k_{s}}{2})}\left|f_{s}(x)\right|_{\Z_{l_{s}/2}}.
\]
In particular, $\Pi_{s}^{t}$ is then a copy of the discrete torus
$\Z_{l_{s}/2}^{k_{s}/2}$ with $l^{\infty}$-metric rescaled by $k_{s}$,
embedded with bounded distortion in $\Delta$. \end{example}

We fix a generating
set for $\Gamma_{s}'$ using the following classical lemma.

\begin{lemma}[Reidemeister-Schreier algorithm] \label{RS}

Let $(\Gamma,S)$ be a group marked with a finite
generating set and $\pi:\Gamma\rightarrow F$ be a surjective mapping
to a finite group $F$. Then 
\begin{enumerate}
\item there exists a set $C=\{a_{1},\dots,a_{|F|}\}$ of coset representatives
\[
\Gamma=\bigcup_{i=1}^{|F|}\left(\mbox{Ker }\pi\right)a_{i}
\]
of length $|a_{i}|_{S}\leq\mbox{diam}_{\pi(S)}(F)$. 
\item The set $R=CSC^{-1}\cap\mbox{Ker }\pi$ is a finite symmetric generating
set of $\mbox{Ker }\pi$. 
\item For any $\gamma\in\mbox{Ker }\pi$, 
\[
|\gamma|_{R}\leq|\gamma|_{S}\leq\left(2\mbox{diam}_{\pi(S)}(F)+1\right)|\gamma|_{R}.
\]

\end{enumerate}
\end{lemma}

For $\Gamma_{s}$ that satisfies Assumption \ref{Gamma_Involution}
and $F=A(s)\times B(s)$, we fix a generating set $R(s)$ for $\Gamma_{s}'=\ker\left(\Gamma_{s}\to A(s)\times B(s)\right)$
provided by the Reidemeister-Schreier algorithm. In this case $\mbox{diam}(A(s)\times B(s))=2$.
It follows from Lemma \ref{RS} that the inclusion map from $\left(\Gamma_{s}',R(s)\right)$
into $\left(\Gamma_{s},A(s)\cup B(s)\right)$ is bi-Lipschitz $|\gamma|_{R(s)}\leq|\gamma|_{\Gamma_{s}}\leq5|\gamma|_{R(s)}$
for all $\gamma\in\Gamma_{s}'$ and that $|R(s)|\leq(|A||B|)^{5}$.

Consider the direct product 
\[
H=\prod_{s\in\mathbb{N}}\left(\Gamma_{s}'\right)^{k_{s}/2}=\prod_{s\geq1}\Pi_{s}^{k_{s}/2},
\]
denote elements of $H$ as $\mathbf{h}=\left(h_{s}\right)$, where
$h_{s}$ is a vector $h_{s}=\left(h_{s}(0),\ldots h_{s}(\frac{k_{s}}{2}-1)\right)\in\left(\Gamma_{s}'\right)^{k_{s}/2}$.
Equip $H$ with a left invariant metric $l$, 
\[
l_{s}(\mathbf{h})=\frac{k_{s}}{2}\max_{0\le j\le k_{s}/2-1}\left|h_{s}(j)\right|_{R_{s}},\ l(\mathbf{h})=\sum_{s\in\mathbb{N}}l_{s}(\mathbf{h}).
\]

\begin{proposition}\label{embed-H}

Suppose $\left\{ \Gamma_{s}\right\} $ is a sequence of finite groups
satisfying Assumption \ref{Gamma_Involution} and $\left(k_{s}\right)$
satisfies growth assumption \ref{k_growth}. Let $\Delta$ be the
diagonal product constructed with $\left\{ \Gamma_{s}\right\} $ and
parameters $\left(k_{s}\right)$. Then $\Delta$ is elementary amenable,
and there exists an embedding $\theta:H\to\Delta$ such that for every
$\mathbf{h}\in H$, 
\[
\max_{s\in\mathbb{N}}l_{s}(\mathbf{h})\le\left|\theta(\mathbf{h})\right|_{\Delta}\le45000\ l(\mathbf{h}).
\]
\end{proposition}

\begin{proof} The group $\Delta$ is elementary amenable by Fact
\ref{EA}. By Proposition \ref{metric} and Lemma \ref{embed-product},
we have for each $s\geq0$ 
\[
|\theta(\mathbf{h})|_{\Delta}\geq|\pi_{s}\theta(\mathbf{h})|_{\Delta_{s}}\geq\frac{k_{s}}{2}\max_{[0,k_{s}/2)}|h_{s}(j)|_{\Gamma_{s}}\geq\frac{k_{s}}{2}\max_{[0,k_{s}/2)}|h_{s}(j)|_{R_{s}}=l_{s}(\mathbf{h}).
\]
and similarly 
\begin{align*}
|\theta(\mathbf{h})|_{\Delta}\leq & 500\sum_{s\geq0}|\pi_{s}\left(\theta(\mathbf{h})\right)|_{\Delta_{s}}\leq500\sum_{s\geq0}10k_{s}\max_{[0,k_{s}/2)}|h_{s}(j)|_{\Gamma_{s}}\\
\leq & 500\cdot10\cdot5\sum_{s\geq0}k_{s}\max_{[0,k_{s}/2)}|h_{s}(j)|_{R_{s}}=25000\ l(\mathbf{h}).
\end{align*}
\end{proof}

Since $\left(k_{s}\right)$ satisfies growth assumption \ref{k_growth},
it is clear that $H$ has at most exponential volume growth with respect
to the length function $l$. By the general theorem of Olshanskii-Osin
\cite{Olshanskii2013}, there exists an elementary amenable group
$G$ equipped with a finite generating set $S$ such that $H$ embeds
as a subgroup of $G$, and there exists a constant $c>0$ such that
$c\left|h\right|_{S}\le l(h)\le|h|_{S}$ for all $h$. In general
the group $G$ provided by the Olshanskii-Osin embedding is rather
large compared to $H$. In the current setting, although the embedding
$\theta:H\to\Delta$ is not necessarily bi-Lipschitz, the geometry
of group $\Delta$ is in some sense controlled by $H$. In particular,
we will show in Section \ref{sec:Compression-expander} and \ref{sec:dihedral-compression}
that if $\{\Gamma_{s}\}$ is taken to be an expander family or finite
dihedral groups and if the sequences $( k_{s}) $, $( \mbox{diam}_{R(s)}\left(\Gamma_{s}'\right)) $
satisfy certain growth conditions, then Hilbert compression exponent
of $\left(\Delta,d_{\Delta}\right)$ is the same as $\left(H,l\right)$,
\[
\alpha_{2}^{\ast}\left(\left(\Delta,d_{\Delta}\right)\right)=\alpha_{2}^{\ast}((H,l)).
\]

\section{Speed and entropy of random walk on $\Delta$}

\label{sec: speed}

Recall that $\Delta$ denote the diagonal product of the sequence
of marked groups $\left\{ \Delta_{s}\right\} $, it is marked with
generating tuple $\mathcal{T}=\left(\tau,\alpha_{1},\ldots,\alpha_{|A|},\beta_{1},\dots,\beta_{|B|}\right)$.
Let $U_{\alpha}$ and $U_{\beta}$  denote the uniform measure
on the subgroups $A=\left\{ \alpha_{1},\ldots,\alpha_{|A|}\right\} $ and $B=\left\{ \beta_{1},\dots,\beta_{|B|}\right\} $
respectively. Let $\mu$ denote the uniform measure on $\left\{ \tau^{\pm1}\right\} $.
For the convenience of speed calculation, we take the following specific
\textquotedbl{}switch-walk-switch\textquotedbl{} step distribution
on $\Delta$, 
\[
q=\left(U_{\alpha}\ast U_{\beta}\right)\ast\mu\ast\left(U_{\alpha}\ast U_{\beta}\right).
\]
Note that in the construction of $\Delta$, since $\Gamma_{0}=A(0)\times B(0)$
and $k_{s}>0$ for all $s>0$, it follows that $\alpha_{i}$ commutes
with $\beta_{j}$, therefore $U_{\alpha}\ast U_{\beta}$ is a symmetric
probability measure on $\Delta$. Let $A_{1,s},A_{2,s},\ldots$ ($B_{1,s},B_{2,s},\ldots$
resp.) be a sequence of independent r.v. with uniform distribution
on $A(s)$ ($B(s)$ resp.). We will refer to $A_{1,s}B_{1,s}\ldots A_{t,s}B_{t,s}$
as a random alternating word of length $t$ in $A(s)$ and $B(s)$
starting with $A$.

We first describe what the random walk trajectory with step distribution
$q$ looks like. Let $X_{1},X_{2},\ldots$ be a sequence of i.i.d.
random variables uniform on $\left\{ \pm1\right\} $, $S_{n}=X_{1}+\ldots+X_{n}$;
$\mathcal{A}_{1},\mathcal{A}_{2},\ldots$ ($\mathcal{B}_{1},\mathcal{B}_{2},\ldots$
resp.) a sequence of i.i.d. random variables with distribution $U_{\alpha}$
($U_{\beta}$ resp.). Let $W_{n}$ denote the random variable on $\Delta$
given by 
\[
W_{n}=\mathcal{A}_{1}\mathcal{B}_{1}\tau^{X_{1}}\mathcal{A}_{2}\mathcal{B}_{2}\ldots\mathcal{A}_{2n-1}\mathcal{B}_{2n-1}\tau^{X_{n}}\mathcal{A}_{2n}\mathcal{B}_{2n},
\]
then $W_{n}$ has distribution $q^{n}$. Each letter $\mathcal{A}$
can be written as $\mathcal{A}=\left(\left(f_{s}^{\mathcal{A}}\right),0\right)$
where $f_{s}^{\mathcal{A}}(0)=\mathcal{A}_{s}$, with $\mathcal{A}_{s}=a_{j}(s)$
if $\mathcal{A}=\alpha_{j}$, and $f_{s}^{\mathcal{A}}(x)=e_{\Gamma_{s}}$
for all $x\neq0$; similarly $\mathcal{B}=\left(\left(f_{s}^{\mathcal{B}}\right),0\right)$
where $f_{s}^{\mathcal{B}}(k_{s})=\mathcal{B}_{s}$, with $\mathcal{B}_{s}=b_{j}(s)$
if $\mathcal{B}=\beta_{j}$, and $f_{s}^{\mathcal{B}}(x)=e_{\Gamma_{s}}$
for all $x\neq k_{s}$.

Now we rewrite $W_{n}$ into the standard form $\left(\left(f_{s}\right),z\right)$.
Consider the projection to the copy $\Delta_{s}$, from the definition
of generators $\alpha_{i}(s)$ and $\beta_{i}(s)$, 
\begin{align*}
 & f_{s}^{W_{n}}(y)\\
 & =\mathcal{A}_{1,s}^{\mathbf{1}_{\left\{ S_{0}=y\right\} }}\mathcal{B}_{1,s}^{\mathbf{1}_{\{S_{0}=y-k_{s}\}}}\left(\prod_{j=1}^{n-1}\left(\mathcal{A}_{2j,s}\mathcal{A}_{2j+1,s}\right)^{\mathbf{1}_{\{S_{j}=y\}}}(\mathcal{B}_{2j,s}\mathcal{B}_{2j+1,s})^{\mathbf{1}_{\{S_{j}=y-k_{s}\}}}\right)\mathcal{A}_{2n,s}^{\mathbf{1}_{\left\{ S_{n}=y\right\} }}\mathcal{B}_{2n}^{\mathbf{1}_{\{S_{n}=y-k_{s}\}}}.
\end{align*}
For $x\in\mathbb{Z}$, let $T(k,x,m)$ be the number of excursions
of the simple random walk $\left\{ S_{n}\right\} $ away from $x$
that cross $x-k$ and are completed before time $m$. Then conditioned
on $\left\{ S_{k}\right\} _{0\le k\le n}$, the distribution of $f_{s}^{W_{n}}(y)$
is the same as a random alternating word in $A(s)$ and $B(s)$ of
length $T(k_{s},y,n)$ with an appropriate random letter added as
prefix/suffix.

\subsection{The case with linear speed in $\left\{ \Gamma_{s}\right\} $\label{sub:linear-speed}}

In this subsection we consider the case where speed of simple random
walk on $\Gamma_{s}$ grows linearly up to some time comparable to
the diameter $l_{s}$.

\begin{definition}\label{linear-speed-Gamma}

Let $\{\Gamma_{s}\}$ be a sequence of finite groups where each $\Gamma_{s}$
is marked with a generating set $A(s)\cup B(s)$. Let $\eta_{s}=U_{A(s)}\ast U_{B(s)}\ast U_{A(s)}$
where $U_{A(s)},U_{B(s)}$ are uniform distribution on $A(s),B(s)$.
We say $\left\{ \Gamma_{s}\right\} $ satisfies $(\sigma,T_{s})$-linear
speed assumption if in each $\Gamma_{s}$, 
\[
L_{\eta_{s}}(t)=\E\left|X_{t}^{(s)}\right|_{\Gamma_{s}}\ge\sigma t\ \mbox{for all }t\le T_{s}
\]
where $X_{t}^{(s)}$ has distribution $\eta_{s}^{\ast t}.$

\end{definition}

Recall that since $X_{t}^{(s)}$ is random walk on a transitive graph,
by \cite[Proposition 8]{AmirVirag2012}, 
\[
\max\left\{ H(X_{t}^{(s)}),1\right\} \geq\frac{1}{t}\left(\frac{1}{4}\E\left|X_{t}^{(s)}\right|_{\Gamma_{s}}\right)^{2}.
\]
Note that if $\left\{ \Gamma_{s}\right\} $ satisfies Assumption \ref{Gamma_Involution},
$\eta_{s}$ projects onto the uniform distribution on $A(s)\times B(s)$,
thus $H(X_{t}^{(s)})\ge H(X_{1}^{(s)})\ge\log\left(\left|A\right|\left|\left|B\right|\right|\right)\ge\log4$.
Therefore in this case $(\sigma,T_{s})$-linear speed assumption implies
that 
\begin{align}
H(X_{t}^{(s)})\ge\sigma't\ \mbox{for all }t\le T_{s}, \mbox{ where } \sigma'=\left(\frac{\sigma}{4}\right)^{2}.\label{entropy-linear}
\end{align}

One important class of examples that satisfies the linear speed assumption
consists of expander families.

\begin{example}\label{example-expanders}

On $\Gamma_{s}=\left\langle A(s),B(s)\right\rangle $, $A(s)\simeq A,B(s)\simeq B$,
let $d=\left|A(s)\right|+\left|B(s)\right|-2$, $\upsilon_{s}$ be
the uniform probability measure on $A(s)\cup B(s)$. Suppose there
exists $\delta>0$ such that the spectral gap $\lambda(\Gamma_{s},\upsilon_{s})$
satisfies 
\[
\lambda(\Gamma_{s},\upsilon_{s})=\inf_{f:\Gamma_{s}\to\mathbb{R},f\neq c}\left\{ \frac{\sum_{u,v\in\Gamma_{s}}\left|f(u)-f(uv)\right|^{2}\nu_{s}(v)}{\frac{1}{|\Gamma_{s}|}\sum_{u,v\in\Gamma_{s}}\left|f(u)-f(v)\right|^{2}}\right\} \ge\delta
\]
for all $s$, that is $\left\{ \Gamma_{s}\right\} $ forms a family
of $d$-regular $\delta$-expanders in $\ell^{2}$. Then $\left\{ \Gamma_{s}\right\} $
satisfies $(\sigma,c_{0}\log\left|\Gamma_{s}\right|)$-linear speed
assumption with with constants $\sigma,c_{0}>0$ only depending on
$\delta$ and $|A|,|B|$. We reproduce the proof of this fact for
completeness, see \cite[Theorem 3.6]{HLW2006}.

By standard comparison of Dirichlet forms, 
\[
\lambda(\Gamma_{s},\eta_{s})=\hat{\delta}\geq\frac{\delta}{|A||B|}.
\]
From the spectral gap we have 
\[
\left|\mathbf{P}(X_{t}^{(s)}=x)-\frac{1}{|\Gamma_{s}|}\right|\le e^{-\hat{\delta}t}.
\]
Then for $t<\frac{1}{\hat{\delta}}\log|\Gamma_{s}|,$ $\gamma=\frac{\hat{\delta}}{2\log d}$,
\[
\mathbf{P}\left(X_{t}^{(s)}\in B(e,\gamma t)\right)\le d^{\gamma t}\left(e^{-\hat{\delta}t}+\frac{1}{|\Gamma_{s}|}\right)\le2\exp\left(\left(\gamma\log d-\hat{\delta}\right)t\right)=2e^{-\hat{\delta}t/2}.
\]
Therefore $\mathbf{E}\left|X_{t}^{(s)}\right|_{\Gamma_{s}}\ge\gamma t\left(1-2e^{-\hat{\delta}t/2}\right).$
We conclude that $\left\{ \Gamma_{s}\right\} $ satisfies $(\sigma,c_{0}\log\left|\Gamma_{s}\right|)$-linear
speed assumption with $\sigma=\min\left\{ \frac{\hat{\delta}}{4\log d},\frac{\hat{\delta}}{2\log4}\right\} $,
$c_{0}=1/\hat{\delta}$.

\end{example}

The lamplighter groups over $\mathbb{Z}^{d}$, $d\ge3$, are the first
examples of solvable groups where simple random walk has linear speed,
see Kaimanovich-Vershik \cite{KV}. The following examples satisfying
the linear speed assumption are analogues of finite quotients of $\mathbb{Z}_{2}\wr\mathbb{Z}^{d}$.

\begin{example}\label{example-lamplighters}

Let $\Gamma=\mathbb{Z}_{2}\wr D_{\infty}^{d}$, $d\ge3$ as in the
2nd item of Example \ref{example-2}, marked with generating subgroups
$A=\mathbb{\mathbb{Z}}_{2}\wr\left\langle a_{j},1\le j\le d\right\rangle $,
$B=\mathbb{\mathbb{Z}}_{2}\wr\left\langle b_{j},1\le j\le d\right\rangle $.
Fix an increasing sequence $n_{s}\in\mathbb{N}$, let $\Gamma_{s}=\mathbb{Z}_{2}\wr D_{2n_{s}}^{d}$.
Then $\Gamma_{s}$ is a finite quotient of $\Gamma$, let $A(s)$,
$B(s)$ denote the projection of $A$ and $B$ to $\Gamma_{s}$. There
exists constant $\sigma_{d}>0$ only depending on $d$ such that $\left\{ \Gamma_{s}\right\} $
satisfies $\left(\sigma_{d},\left(2n_{s}\right)^{d}\right)$-linear
speed assumption. A proof of this fact is included in Lemma \ref{lamplighter-torus}
in the Appendix.

\end{example}

\subsubsection{Bounds on speed and entropy in one copy}

In the upper bound direction, we will use the trivial bound that in
each lamp group, 
\begin{equation}
\left|f_{s}^{W_{n}}(x)\right|_{\Gamma_{s}}\le\min\left\{ 2T\left(k_{s},x,n\right)+\mathbf{1}_{\left\{ L(x,n)>0\right\} }+\mathbf{1}_{\left\{ L(x-k_{s},n)>0\right\} },\mbox{diam}\left(\Gamma_{s}\right)\right\} .\label{eq:trivial-upper}
\end{equation}
Recall that we set the parameter $l_{s}=\mbox{diam}\left(\Gamma_{s}\right)$.

\begin{lemma}\label{linear-one-upper}

There exists an absolute constant $C>0$ such that for all $s\ge0$,
\[
\mathbf{E}\left[\left|\left(f_{s}^{W_{n}},S_{n}\right)\right|_{\Delta_{s}}\mathbf{1}_{\left\{ s\le s_{0}\left(W_{n}\right)\right\} }\right]\le\begin{cases}
Cn^{\frac{1}{2}}\min\left\{ \frac{n^{\frac{1}{2}}}{k_{s}},l_{s}\right\}  & \mbox{ if }k_{s}^{2}\le n\\
C\left(n^{\frac{1}{2}}+k_{s}\right)e^{-\frac{k_{s}^{2}}{8n}} & \mbox{ if }k_{s}^{2}>n.
\end{cases}
\]

\end{lemma}

\begin{proof}

From the metric upper estimate in Lemma \ref{one_metric}, 
\[
\left|\left(f_{s}^{W_{n}},S_{n}\right)\right|_{\Delta_{s}}\le9\left(\sum_{j\in\mathbb{Z}}k_{s}\max_{x\in I_{j}^{s}}\left|f_{s}^{W_{n}}(x)\right|_{\Gamma_{s}}+\mathcal{R}_{n}\right),
\]
where $\mathcal{R}_{n}=\#\left\{ S_{k}:\ 0\le k\le n\right\} $ is
the size of the range of simple random walk on $\mathbb{Z}$ and $I_{j}^{s}=\left[j\frac{k_{s}}{2},(j+1)\frac{k_{s}}{2}\right)$.
Observe that for each $x\in I_{j}^{s}$, 
\begin{equation}
T\left(k_{s},x,n\right)\le T\left(\frac{k_{s}}{2},j\frac{k_{s}}{2},n\right),\label{traverse-half}
\end{equation}
because each excursion from $x$ to the left that crosses $x-k_{s}$
must contain an excursion from $j\frac{k_{s}}{2}$ to the left that
crosses $(j-1)\frac{k_{s}}{2}$. Apply (\ref{eq:trivial-upper}),
\begin{align*}
\left|\left(f_{s}^{W_{n}},S_{n}\right)\right|_{\Delta_{s}} & \le9\sum_{j\in\mathbb{Z}}k_{s}\max_{x\in I_{j}^{s}}\left\{ 2T(k_{s},x,n)\right\} +11(\mathcal{R}_{n})\\
 & \le11\left(\sum_{j\in\mathbb{Z}}k_{s}T\left(\frac{k_{s}}{2},j\frac{k_{s}}{2},n\right)+\mathcal{R}_{n}\right).
\end{align*}
By Lemma \ref{traverse-upper}, 
\[
\mathbf{E}T\left(\frac{k_{s}}{2},j\frac{k_{s}}{2},n\right)\le\frac{2Cn^{\frac{1}{2}}}{k_{s}}\exp\left(-\frac{(jk_{s}/2)^{2}}{2n}\right).
\]
The size of range of simple random walk on $\mathbb{Z}$ satisfies
\[
\mathbf{P}\left(\mathcal{R}_{n}\ge x\right)\le\mathbf{P}\left(\max_{0\le k\le n}\left|S_{k}\right|\ge\frac{x}{2}\right)\le4\exp\left(-\frac{x^{2}}{8n}\right).
\]
Recall that by definition of $s_{0}(Z)$ in subsection \ref{sub:Local-finiteness},
\[
\left\{ s\le s_{0}(W_{n})\right\} \subseteq\left\{ \mathcal{R}_{n}\ge k_{s}\right\} .
\]
Summing up, 
\begin{align*}
 & \mathbf{E}\left[\left|\left(f_{s}^{W_{n}},S_{n}\right)\right|_{\Delta_{s}}\mathbf{1}_{\left\{ s\le s_{0}\left(W_{n}\right)\right\} }\right]\\
 & \le11k_{s}\sum_{j\in\mathbb{Z}}\mathbf{E}T\left(\frac{k_{s}}{2},j\frac{k_{s}}{2},n\right)+11\mathbf{E}\left[\left(\mathcal{R}_{n}\right)\mathbf{1}_{\left\{ \mathcal{R}_{n}\ge k_{s}\right\} }\right]\\
 & \le11k_{s}\sum_{j\in\mathbb{Z}}\frac{Cn^{\frac{1}{2}}}{k_{s}}\exp\left(-\frac{(jk_{s}/2)^{2}}{2n}\right)+C'n^{\frac{1}{2}}e^{-\frac{k_{s}^{2}}{8n}}\\
 & \le\begin{cases}
C''\left(\frac{n}{k_{s}}+n^{\frac{1}{2}}\right) & \mbox{ if }k_{s}^{2}\le n\\
C''\left(n^{\frac{1}{2}}+k_{s}\right)e^{-\frac{k_{s}^{2}}{8n}} & \mbox{ if }k_{s}^{2}>n.
\end{cases}.
\end{align*}
For $k_{s}^{2}\leq n,$ since $\left|f_{s}^{W_{n}}(y)\right|_{\Gamma_{s}}$
can't exceed the diameter of $\Gamma_{s}$, together with Lemma \ref{one_metric},
we have a second upper bound 
\[
\E|(f_{s}^{W_{n}},S_{n})|_{\Delta_{s}}\leq10l_{s}\E(\mathcal{R}_{n}+k_{s})\leq Cl_{s}n^{\frac{1}{2}}.
\]
Combine these bounds we obtain the statement.

\end{proof}

Now we turn to the lower bound direction.

\begin{lemma}\label{linear-one-lower}

Suppose $\left\{ \Gamma_{s}\right\} $ satisfies $(\sigma,c_{0}l_{s})$-linear
speed assumption. Then there exists an absolute constant $C>0$ such
that for $s$ with $k_{s}^{2}\le n$, 
\[
\mathbf{E}\left(\left|\left(f_{s}^{W_{n}},S_{n}\right)\right|_{\Delta_{s}}\right)\ge\frac{\sigma}{C}\min\left\{ \frac{n}{k_{s}},c_{0}n^{\frac{1}{2}}l_{s}\right\} ,
\]
and 
\[
H(f_{s}^{W_{n}})\ge\frac{\sigma'}{C}\min\left\{ \frac{n}{k_{s}},c_{0}n^{\frac{1}{2}}l_{s}\right\} ,
\]
\end{lemma}

\begin{proof}

We use a weaker lower bound for the metric, 
\[
\left|\left(f_{s},z\right)\right|_{\Delta_{s}}\ge\sum_{y\in\mathbb{Z}}\left|f_{s}(y)\right|_{\Gamma_{s}}.
\]
Apply the $(\sigma,c_{0}l_{s})$-linear speed assumption of Definition \ref{linear-speed-Gamma}, we have
\[
\mathbf{E}\left(\left|\left(f_{s}^{W_{n}},S_{n}\right)\right|_{\Delta_{s}}\right)\ge\mathbf{E}\left(\sum_{y\in\mathbb{Z}}\left|f_{s}^{W_{n}}(y)\right|\right)\ge\sum_{y\in\mathbb{Z}}\mathbf{E}\left[\sigma\min\left\{ T(k_{s},y,n),c_{0}l_{s}\right\} \right].
\]
Then by Lemma \ref{traverse-lower}, there exists constant $c>0$,
for $k_{s}\le c^{2}n^{\frac{1}{2}}$, 
\[
\mathbf{E}\left[\min\left\{ T(k_{s},y,n),c_{0}l_{s}\right\} \right]\ge\frac{1}{2}\min\left\{ \frac{c\sqrt{n}}{4k},c_{0}l_{s}\right\} \mathbf{P}\left(L\left(y,\frac{n}{2}\right)>0\right).
\]
Summing up over $y$, 
\[
\sum_{y\in\mathbb{Z}}\mathbf{E}\left[\min\left\{ T(k_{s},y,n),c_{0}l_{s}\right\} \right]\ge\frac{1}{2}\min\left\{ \frac{c\sqrt{n}}{4k_s},c_{0}l_{s}\right\} \mathbf{E}\mathcal{R}_{n/2},
\]
where $\mathcal{R}_{n/2}$ is the size of the range of simple random
walk on $\mathbb{Z}$ up to $n/2$. Since $\mathbf{E}\mathcal{R}_{n}\simeq n^{\frac{1}{2}}$,
it follows there exists constant $C>1$ such that for $s$ with $k_{s}^{2}\le n$,
\[
\mathbf{E}\left(\left|\left(f_{s}^{W_{n}},S_{n}\right)\right|_{\Delta_{s}}\right)\ge\frac{\sigma}{C}\min\left\{ \frac{n}{k_{s}},c_{0}n^{\frac{1}{2}}l_{s}\right\} .
\]

Concerning entropy, we condition by the traverse time function (see for instance \cite{AmirVirag2012} for the basic properties of entropy)
\begin{align*}
H(f_{s}^{W_{n}}) & \geq H(f_{s}^{W_{n}}|T(k_{s},\cdot,n))=\sum_{z\in\Z}\E H(X_{T(k_{s},z,n)}^{(s)})\\
 & \geq\sum_{z\in\Z}\E\left[\sigma'\min\{T(k_{s},z,b),c_{0}l_{s}\}\right]\geq\frac{\sigma'}{C}\min\left\{ \frac{n}{k_{s}},c_{0}n^{\frac{1}{2}}l_{s}\right\} 
\end{align*}
using (\ref{entropy-linear}) and the same computation. \end{proof}

\subsubsection{Speed and entropy estimates in the diagonal product $\Delta$}

Recall Assumption \ref{k_growth} that $(k_{s})$ grows exponentially.
In the diagonal product $\Delta$, by the metric upper estimate in
Proposition \ref{metric} and speed upper estimates in Lemma \ref{linear-one-upper},
we have 
\begin{equation}
\E(|W_{n}|_{\Delta})\leq\E\left(500\sum_{s\leq s_{0}(W_{n})}\left|\left(f_{s}^{W_{n}},S_{n}\right)\right|_{\Delta_{s}}\right)\le\sum_{s\le s_{0}(n)}C\left(\min\left\{ \frac{n}{k_{s}},n^{\frac{1}{2}}l_{s}\right\} +n^{\frac{1}{2}}\right),\label{eq:linear-upper}
\end{equation}
where 
\[
s_{0}(n)=\min\left\{ s:\ k_{s}^{2}\ge n\right\} .
\]
Indeed, denote $x_{s}=\frac{k_{s}}{\sqrt{n}}$ growing exponentially,
then 
\[
\sum_{s>s_{0}(n)}(n^{\frac{1}{2}}+k_{s})e^{-\frac{k_{s}^{2}}{8n}}\leq n^{\frac{1}{2}}\sum_{x_{s}\geq1}(1+x_{s})e^{-\frac{x_{s}^{2}}{8}}\leq Cn^{\frac{1}{2}}.
\]

In the lower bound direction, suppose $\left\{ \Gamma_{s}\right\} $
satisfies $(\sigma,c_{0})$-linear speed assumption \ref{linear-speed-Gamma},
then by the metric lower estimate in Proposition \ref{metric} and
Lemma \ref{linear-one-lower}, 
\begin{equation}
\mathbf{E}\left(\left|W_{n}\right|_{\Delta}\right)\ge\max_{s}\mathbf{E}\left(\left|\left(f_{s}^{W_{n}},S_{n}\right)\right|_{\Delta_{s}}\right)\ge\frac{\sigma}{C}\max_{s\le s_{0}(n)}\min\left\{ \frac{n}{k_{s}},n^{\frac{1}{2}}c_{0}l_{s}\right\} .\label{eq:linear-lower}
\end{equation}

To understand the upper bound (\ref{eq:linear-upper}), divide the
collection of $\Delta_{s}$ with $s\le s_{0}(n)$ into two subsets. 
\begin{description}
\item [{(i)}] Let $s_{1}(n)$ denote the index 
\[
s_{1}(n)=\max\left\{ s\ge0:\ n^{\frac{1}{2}}\ge k_{s}l_{s}\right\} .
\]
Then the contribution of these $s\le s_{1}(n)$ to the sum is bounded
by 
\[
\sum_{s\leq s_{1}(n)}\E\left|\left(f_{s}^{W_{n}},S_{n}\right)\right|_{\Delta_{s}}\le Cn^{\frac{1}{2}}\sum_{s\le s_{1}(n)}\left(l_{s}+1\right).
\]

\item [{(ii)}] The contribution of $s\in\left(s_{1}(n)+1,s_{0}(n)\right]$
to the sum is bounded by 
\[
\sum_{s=s_{1}(n)+1}^{s_{0}(n)}\E\left|\left(f_{s}^{W_{n}},S_{n}\right)\right|_{\Delta_{s}}\le C\sum_{s=s_{1}(n)+1}^{s_{0}(n)}\left(\frac{n}{k_{s}}+n^{\frac{1}{2}}\right).
\]

\end{description}
Combine these parts, we have 
\begin{equation}
\E(|W_{n}|_{\Delta})\le2C\left(n^{\frac{1}{2}}\sum_{s=0}^{s_{1}(n)}l_{s}+\sum_{s=s_{1}(n)+1}^{s_{0}(n)}\frac{n}{k_{s}}\right).\label{eq:linear-upper-2}
\end{equation}

\begin{proposition}\label{Delta-speed}

Suppose $\left\{ \Gamma_{s}\right\} $ satisfies $(\sigma,c_{0}l_{s})$-linear
speed assumption and $\mbox{diam}(\Gamma_{s})\le C_{0}l_{s}$. Suppose
there exists a constant $m_{0}>1$ such that 
\[
k_{s+1}>2k_{s},l_{s+1}\ge m_{0}l_{s}\mbox{ for all }s.
\]
 Let 
\[
s_{0}(n)=\min\left\{ s:\ k_{s}^{2}\ge n\right\} ,\ s_{1}(n)=\max\left\{ s\ge0:\ n^{\frac{1}{2}}\ge k_{s}l_{s}\right\} ,
\]
then 
\[
\frac{\sigma c_{0}}{2C}\left(n^{\frac{1}{2}}l_{s_{1}(n)}+\frac{n}{k_{s_{1}(n)+1}}\right)\le\E\left|W_{n}\right|_{\Delta}\le\frac{4C}{1-1/m_{0}}\left(n^{\frac{1}{2}}l_{s_{1}(n)}+\frac{n}{k_{s_{1}(n)+1}}\right).
\]
The same bounds hold for the entropy $H(W_{n}^{\Delta})$ with $\sigma$ replaced
by $\sigma'=\left(\frac{\sigma}{4}\right)^{2}$ and $C$ replaced
by a constant $C'>0$ that only depends on the size of generating sets
$\left|A\right|+\left|B\right|$.

\end{proposition}

\begin{proof} The lower bound is a direct consequence of (\ref{eq:linear-lower}).
For the upper bound, apply (\ref{eq:linear-upper-2}) and note that
because of the assumption on growth of $k_{s},l_{s}$, the sums satisfy
\begin{align*}
\sum_{s\le s_{1}(n)}l_{s} & \le l_{s_{1}(n)}\frac{1}{1-1/m_{0}},\\
\sum_{s\ge s_{1}(n)+1}\frac{1}{k_{s}} & \le\frac{2}{k_{s_{1}(n)+1}}.
\end{align*}
For the entropy, by \cite[Lemma 6 ]{Erschler2003}, there is $C'$
depending only on the exponential rate of volume growth in $\G$,
which can't exceed $\log(\left|A\right|+\left|B\right|)$, such that
$H(W_{n}^{\Delta})\leq C'\E\left|W_{n}\right|_{\Delta}$, giving the
upper bound. Lemma \ref{linear-one-lower} gives the lower bound (\ref{eq:linear-lower}).
\end{proof}

\subsubsection{Possible speed and entropy functions}

\begin{example} If $\G_{s}$ is an family as in Examples \ref{example-expanders}
or \ref{example-lamplighters} and $k_{s}=2^{\beta s+o(s)}$ and $l_{s}=2^{\iota s+o(s)}$,
with $\beta,\iota\in[1,\infty)$, a direct application of Proposition
\ref{Delta-speed} shows that the speed and entropy exponents are
\[
\lim\frac{\log\E|W_{n}|_{\Delta}}{\log n}=\lim\frac{\log H(W_{n}^{\Delta})}{\log n}=\frac{\beta+2\iota}{2\beta+2\iota}=1-\frac{1}{2(1+\frac{\iota}{\beta})},
\]
which can take any value in $(\frac{1}{2},1)$.

\end{example}

\begin{theorem}\label{possible-speed-linear}

There exists universal constants $c,C>0$ such that the following
holds. For any function $\varrho:[1,\infty)\to[1,\infty)$ such that
$\frac{\varrho(x)}{\sqrt{x}}$ and $\frac{x}{\varrho(x)}$ are non-decreasing,
there exists a group $\Delta$ and a symmetric probability
measure $q$ of finite support on $\Delta$ such that 
\[
c\varrho(n)\le\E\left|W_{n}\right|_{\Delta}\le C\varrho(n)\mbox{ and }c\varrho(n)\le H(W_{n}^{\Delta})\le C\varrho(n)
\]
Moreover, the group $\Delta$ can be chosen $4$-step solvable.

\end{theorem}

\begin{proof} The choice of a family of $\G_{s}$ as in Example \ref{example-expanders}
or \ref{example-lamplighters} guarantees the existence of $C_{1}>1$
such that for all $x\geq1$, there is $\G_{s}$ of diameter $l_{s}$
with $\frac{x}{C_{1}}\leq l_{s}\leq C_{1}x$.

As $\varrho$ belongs to $\mathcal{C}_{\frac{1}{2},1}$, Corollary
\ref{approximation-speed} provides two sequences $(k_{s})$ of integers
and $(l_{s})$ among diameters of $\Gamma_{s}$ with $k_{s+1}\geq m_{0}k_{s}$
and $l_{s+1}\geq m_{0}l_{s}$ for all $s$ such that the function
\[
\bar{\varrho}(x)=x^{\frac{1}{2}}l_{s}+\frac{x}{k_{s+1}},\textrm{ for }(k_{s}l_{s})^{2}\leq x\leq(k_{s+1}l_{s+1})^{2}
\]
satisfies $\bar{\varrho}(x)\simeq_{2m_{0}C_{1}^{5}}\varrho(x)$ for
all $x$.

Combining with Proposition \ref{Delta-speed}, the diagonal product
$\Delta$ associated to these sequences has a speed and entropy satisfying
\[
\frac{\sigma c_{0}}{4m_{0}CC_{1}^{5}}\varrho(x)\leq\E|W_{n}|_{\Delta}\leq\frac{4m_{0}CC_{1}^{5}}{1-\frac{1}{m_{0}}}\varrho(x).
\]
For $\{\G_{s}\}$ as in Example \ref{example-lamplighters}, the group
$\Delta$ is $4$-step solvable. \end{proof}

Note that when the speed is linear, the last term of the sequence
$(l_{s})$ is infinite, thus the last quotient $\G_{s}$ is in fact
the whole group $\G$. In our examples, $\G$ is either $\Z_{2}\wr D_{\infty}^{d}$
for $d\geq3$ or a lattice in $SL(3,F)$. In the latter case, the
finite diagonal product $\Delta$ is non amenable. When the speed
is diffusive, the last term of the sequence $(k_{s})$ is infinite,
the group $\Delta$ is a diagonal product of finitely many groups
$\Delta_{s}$ where the lamp groups $\G_{s}$ are finite.

\subsection{The case of $\Delta$ with dihedral groups}

\label{sub:dihedral-speed}

In this subsection we focus on the case where $\left\{ \Gamma_{s}\right\} $
is taken to be a sequence of finite dihedral groups. Since the unlabelled
Cayley graph of $D_{2l}$ is the same as a cycle of size $2l$, simple
random walk on $D_{2l}$ can be identified with simple random walk
on the cycle of size $2l$. Consider simple random walk on the cycle
as projection of simple random walk on $\mathbb{Z},$ then the classical
Gaussian bounds on simple random walk on $\mathbb{Z}$ imply that
there exists constants $c,C>0$ such that for all $1\le t\le l_{s}^{2},1\le x\le t^{\frac{1}{2}},$
\begin{equation}
c\exp\left(-Cx^{2}\right)\le\mathbf{P}\left(\left|X_{t}^{(s)}\right|_{D_{2l_{s}}}\ge xt^{\frac{1}{2}}\right)\le C\exp\left(-cx^{2}\right),\label{eq:dihedral-deviation}
\end{equation}
where $X_{t}^{(s)}$ is a random alternating word in $\{e,a(s)\}$
and $\{e,b(s)\}$.

\begin{lemma}\label{dihedral-one-upper}

Suppose $\Gamma_{s}=D_{2l_{s}}$. There exists an absolute constant
$C>0$ such that in each $\Delta_{s}$, 
\[
\mathbf{E}\left[\left|\left(f_{s}^{W_{n}},S_{n}\right)\right|_{\Delta_{s}}\mathbf{1}_{\left\{ s\le s_{0}\left(W_{n}\right)\right\} }\right]\le\begin{cases}
C\min\left\{ n^{\frac{3}{4}}k_{s}^{-\frac{1}{2}}\log^{\frac{1}{2}}k_{s},nk_{s}^{-1},n^{\frac{1}{2}}l_{s}\right\}  & \mbox{ if }k_{s}^{2}<n\\
C\left(n^{\frac{1}{2}}+k_{s}\right)e^{-\frac{k_{s}^{2}}{8n}} & \mbox{ if }k_{s}^{2}\ge n.
\end{cases}
\]

\end{lemma}

\begin{proof} From the upper bounds in Lemma \ref{linear-one-upper}
which is valid for any choice of $\{\Gamma_{s}\}$, the only bound
we need to show here is that if $k_{s}\le n^{\frac{1}{2}}$ , then
\begin{equation}
\mathbf{E}\left[\left|\left(f_{s}^{W_{n}},S_{n}\right)\right|_{\Delta_{s}}\mathbf{1}_{\left\{ s\le s_{0}\left(W_{n}\right)\right\} }\right]\le Cn^{\frac{3}{4}}k_{s}^{-\frac{1}{2}}\log^{\frac{1}{2}}k_{s}.\label{eq:dihedral-upper-middle}
\end{equation}

To prove this, note that the collection $\left(\left|f_{s}^{W_{n}}(z)\right|_{D_{\infty}}\right)_{z\in I_{j}^{s}}$
as a vector is stochastically dominated by the random vector 
\[
\left(\left|X_{T(k/2,jk_{s}/2,n)}(z)\right|_{D_{\infty}}+\mathbf{1}_{\{L(n,z)>0\}}+\mathbf{1}_{\{L(n,z-k)>0\}}\right)_{z\in I_{j}^{s}},
\]
where $\left\{ X_{t}(z)\right\} $ is a sequence of independent random
alternating words in $A(s)$ and $B(s)$ of length $t$. Then $\max_{z\in I_{j}^{s}}\left|f_{s}^{W_{n}}(z)\right|_{D_{\infty}}$
is stochastically dominated by 
\[
\max_{z\in I_{j}^{s}}\left|X_{T(k/2,jk_{s}/2,n)}(z)\right|_{D_{\infty}}+\mathbf{1}_{\{L(n,I_{j}^{s})>0\}}+\mathbf{1}_{\{L(n,I_{j-2}^{s})>0)\}}.
\]
Plug in the metric estimate in Lemma \ref{one_metric}, 
\begin{multline}
\mathbf{E}\left|\left(f_{s}^{W_{n}},S_{n}\right)\right|\mathbf{1}_{\left\{ s\le s_{0}\left(W_{n}\right)\right\} }\\
\le9k_{s}\left(\sum_{j\in\mathbb{Z}}\E\max_{z\in I_{j}^{s}}\left|X_{T(k/2,jk_{s}/2,n)}(z)\right|_{D_{\infty}}\right)+11\mathbf{E}\left[\mathcal{R}_{n}\mathbf{1}_{\left\{ \mathcal{R}_{n}\ge k_{s}\right\} }\right].\label{eq:sublinear-upper-1}
\end{multline}
From the upper bound in (\ref{eq:dihedral-deviation}), since $\max_{z\in I_{j}^{s}}\left|X_{t}(z)\right|$
is maximum of $k_{s}/2$ i.i.d. random variables, 
\[
\mathbf{P}\left(\max_{z\in I_{j}^{s}}\left|X_{t}(z)\right|\le xt^{\frac{1}{2}}\right)\ge\left(1-c_{1}\exp\left(-c_{2}x^{2}\right)\right)^{\left|I_{j}^{s}\right|}.
\]
Then 
\[
\mathbf{E}\left[\left.\max_{z\in I_{j}^{s}}\left|X_{T(k/2,jk_{s}/2,n)}(z)\right|\right||T(k/2,jk_{s}/2,n)\right]\le C_{1}T(k/2,jk_{s}/2,n)^{\frac{1}{2}}\log^{\frac{1}{2}}k_{s},
\]
where $C_{1}$ depends on $c_{1},c_{2}$. Apply Lemma \ref{traverse-upper},
\[
\mathbf{E}\left[\max_{z\in I_{j}^{s}}\left|X_{T(k/2,jk_{s}/2,n)}(z)\right|\right]\le C_{2}\left(\frac{n^{\frac{1}{2}}}{k_{s}}\exp\left(-\frac{(jk_{s}/2)^{2}}{2n}\right)\right)^{\frac{1}{2}}\log^{\frac{1}{2}}k_{s}.
\]
Plug the estimates in (\ref{eq:sublinear-upper-1}) and summing up
over $j$, we obtain (\ref{eq:dihedral-upper-middle}), because of
the Gaussian tail, the main contribution comes from $1\le j\le n^{\frac{1}{2}}/k_{s}$.
\end{proof}

\begin{lemma}\label{dihedral-one-lower}

Suppose $\Gamma_{s}=D_{2l_{s}}$. There exists an absolute constant
$c>0$ such that in each $\Delta_{s}$, for $n\geq ck_{s}^{2}$, 
\[
\E\left|(f_{s}^{W_{n}},S_{n})\right|_{\Delta_{s}}\geq c\min\left\{ n^{\frac{3}{4}}k_{s}^{-\frac{1}{2}}\log^{\frac{1}{2}}k_{s},nk_{s}^{-1},n^{\frac{1}{2}}l_{s}\right\} .
\]

\end{lemma}

\begin{proof} For each $z\in I_{j}^{s}=[\frac{j}{2}k_{s},\frac{j+1}{2}k_{s})$,
the traverse time satisfies $T(k_{s},z,n)\geq T(2k_{s},\frac{j+1}{2}k_{s},n)$.
So 
\[
\E\max_{z\in I_{j}^{s}}\left|f_{s}^{W_{n}}(z)\right|_{D_{2l_{s}}}\geq\E\max_{z\in I_{j}^{s}}\left|X_{T(2k_{s},\frac{j+1}{2}k_{s},n)}^{s}(z)\right|_{D_{2l_{s}}},
\]
where $\left\{ X_{t}(z)\right\} $ is a sequence of independent random
alternating words in $A(s)$ and $B(s)$ of length $t$.

For any $t\geq1$ and $1\leq x\leq t^{\frac{1}{2}}$, 
\[
\mathbf{P}\left(\left|X_{t}(z)\right|_{D_{2l_{s}}}\geq\min\left\{ xt^{\frac{1}{2}},\frac{l_{s}}{2}\right\} \right)\geq c_{1}e^{-c_{2}x^{2}}.
\]
By independence, this implies the existence of $c>0$ with 
\[
\mathbf{P}\left(\max_{z\in I_{j}^{s}}\left|X_{t}(z)\right|_{D_{2l_{s}}}\geq\min\left\{ c\log^{\frac{1}{2}}k_{s}t^{\frac{1}{2}},ct,\frac{l_{s}}{2}\right\} \right)\geq c.
\]
Using Lemma \ref{one_metric} and Lemma \ref{traverse-lower}, for
some $c>0$ and for all $n\geq ck_{s}^{2}$ 
\begin{eqnarray*}
\E|(f_{s}^{W_{n}},S_{n})|_{\Delta_{s}} & \geq & \frac{k_{s}}{4}\sum_{j\in\Z}\E\max_{z\in I_{j}^{s}}\left|X_{T(2k_{s},\frac{j+1}{2}k_{s},n)}^{s}(z)\right|_{D_{2l_{s}}}\\
 & \geq & ck_{s}\sum_{j\in\Z}\min\left\{ n^{\frac{1}{4}}k_{s}^{-\frac{1}{2}}\log^{\frac{1}{2}}k_{s},n^{\frac{1}{2}}k_{s}^{-1},\frac{l_{s}}{2}\right\} P_{0}\left(L\left(\frac{j+1}{2}k_{s},\frac{n}{2}\right)>0\right).
\end{eqnarray*}
Finally $\sum_{j\in\Z}P_{0}\left(L(\frac{j+1}{2}k_{s},\frac{n}{2})>0\right)\geq\E\mathcal{R}_{\frac{n}{2}}/k_{s}\geq c\frac{n^{\frac{1}{2}}}{k_{s}}$.
\end{proof}

\begin{proposition}\label{dihedral-speed}

Suppose $\Gamma_{s}=D_{2l_{s}}$ and there exists $m_{0}>1$ such
that $k_{s+1}>2k_{s}$, $l_{s+1}>m_{0}l_{s}$ for all $s$. Define
\[
t_{1}(n)=\max\left\{ s:\ \frac{l_{s}^{2}k_{s}}{\log k_{s}}<n^{\frac{1}{2}}\mbox{ and }l_{s}k_{s}<n^{\frac{1}{2}}\right\} .
\]
Then in the diagonal product $\Delta$,

\begin{multline*}
\frac{n^{\frac{1}{2}}}{C}\left(l_{t_{1}(n)}+\min\left\{ n^{\frac{1}{4}}\left(\frac{\log k_{t_{1}(n)+1}}{k_{t_{1}(n)+1}}\right)^{\frac{1}{2}},\frac{n^{\frac{1}{2}}}{k_{t_{1}(n)+1}}\right\} \right)\le\E(|W_{n}|_{\Delta})\\
\le\frac{2Cn^{\frac{1}{2}}}{1-1/m_{0}}\left(l_{t_{1}(n)}+\min\left\{ n^{\frac{1}{4}}\left(\frac{\log k_{t_{1}(n)+1}}{k_{t_{1}(n)+1}}\right)^{\frac{1}{2}},\frac{n^{\frac{1}{2}}}{k_{t_{1}(n)+1}}\right\} \right).
\end{multline*}

\end{proposition}

\begin{remark}

The bounds here are more complicated than the linear case because
in Lemma \ref{dihedral-one-upper} and \ref{dihedral-one-lower} we
have to consider minimum over three quantities. If we further assume
that $l_{s}\ge\log k_{s}$ for all $s$, then the bounds simplify
to 
\[
\E(|W_{n}|_{\Delta})\simeq_{Cm_{0}}n^{\frac{1}{2}}l_{t_{1}(n)}+n^{\frac{3}{4}}\left(\frac{\log k_{t_{1}(n)+1}}{k_{t_{1}(n)+1}}\right)^{\frac{1}{2}}.
\]

\end{remark}

\begin{proof} By Proposition \ref{metric} and the third line of
Lemma \ref{dihedral-one-upper} 
\[
c\max_{s\leq s_{0}(n)}\E\left|W_{n}\right|_{\Delta_{s}}\leq\E\left|W_{n}\right|_{\Delta}\leq C\sum_{s\leq s_{0}(n)}\E\left|W_{n}\right|_{\Delta_{s}}.
\]
The choice of $t_{1}(n)$ implies that 
\[
\min\left\{ n^{\frac{3}{4}}k_{s}^{-\frac{1}{2}}\log^{\frac{1}{2}}k_{s},nk_{s}^{-1},n^{\frac{1}{2}}l_{s}\right\} =\left\{ \begin{array}{ll}
n^{\frac{1}{2}}l_{s} & \forall s\leq t_{1}(n)\\
\min\left\{ n^{\frac{3}{4}}k_{s}^{-\frac{1}{2}}\log^{\frac{1}{2}}k_{s},nk_{s}^{-1}\right\}  & \forall s>t_{1}(n).
\end{array}\right.
\]
Using Lemma \ref{dihedral-one-upper} and \ref{dihedral-one-lower}
and the exponential growth of $(k_{s}),(l_{s})$ gives the proposition.
\end{proof}

\begin{example}

Let $k_{s}=2^{\beta s}$ and $l_{s}=2^{\iota s}$, with $\beta>1$,
$\iota>0$, and $\Gamma_{s}=D_{2l_{s}}$. Proposition \ref{dihedral-speed}
implies with this choice of parameters, 
\[
\E\left(\left|W_{n}\right|_{\Delta}\right)\simeq n^{\frac{3\iota+\beta}{4\iota+2\beta}}\left(\log n\right)^{\frac{\iota}{2\iota+\beta}}.
\]

\end{example}

\begin{theorem} \label{possible-speed-dihedral}

There exists universal constants $c,C>0$ such that the following
holds. For any continuous function $\varrho:[1,\infty)\rightarrow[1,\infty)$
satisfying $\varrho(1)=1$ and $\frac{\varrho(x)}{x^{\frac{1}{2}}\log^{1+\epsilon}x}$,
$\frac{x^{\frac{3}{4}}}{\varrho(x)}$ non-decreasing for some $\epsilon>0$,
there exists a $3$-step solvable group $\Delta$ with dihedral groups
$\G_{s}$ and a symmetric probability measure $q$ of finite support
on $\Delta$ such that 
\[
c\varrho(n)\leq\E\left|W_{n}\right|_{\Delta}\leq C\varrho(n).
\]
\end{theorem}

\begin{remark} The lower condition on $\varrho(x)$ is only technical.
There is no gap that would isolate diffusive behaviors among $3$-step
solvable groups. Indeed, it easily follows from Lemma \ref{dihedral-one-upper}
and \ref{dihedral-one-lower} that for any $\varrho(x)$ such that
$\frac{\varrho(x)}{\sqrt{x}}$ tends to infinity, there is a group
$\Delta$ with dihedral $\Gamma_{s}$ such that $cn^{\frac{1}{2}}\leq\E|W_{n}|_{\Delta}\leq C\varrho(n)$
for all $n$.

\end{remark}

\begin{proof} By Corollary \ref{approximation-speed}, we can find
two sequences $(\kappa_{s}),(l_{s})$ satisfying $\log\kappa_{s}\leq l_{s}$
such that $\bar{\varrho}(x)$ and $\varrho(x)$ agree up to multiplicative
constants.

Let us set $\kappa_{s}=\left(\frac{k_{s}}{\log k_{s}}\right)^{\frac{1}{2}}$.
Then $\log k_{s}\simeq\log\kappa_{s}\leq l_{s}^{\frac{1}{1+\e}}$.
By Corollary \ref{dihedral-speed}, the diagonal product $\Delta$
with dihedral groups $\Gamma_{s}=D_{2l_{s}}$ and sequence $(k_{s})$
satisfies for any $(l_{s}\kappa_{s})^{4}\leq n\leq(l_{s+1}\kappa_{s+1})^{4}$
\[
c\bar{\varrho}(x)=c\left(n^{\frac{1}{2}}l_{s}+\frac{n^{\frac{3}{4}}}{\kappa_{s+1}}\right)\leq\E\left|W_{n}\right|_{\Delta}\leq C\left(n^{\frac{1}{2}}l_{s}+\frac{n^{\frac{3}{4}}}{\kappa_{s+1}}\right)=C\bar{\varrho}(x).
\]
The group $\Delta$ clearly has a trivial third derived subgroup.
\end{proof}

\begin{proposition}\label{dihedral-entropy} There exists two constants
$c,C>0$ such that on any diagonal product $\Delta$ with dihedral
$\G_{s}$ satisfying Assumption \ref{k_growth}, the entropy of the
switch-walk-switch random walk satisfies 
\[
c\sqrt{n}\leq H(W_{n}^{\Delta})\leq C\sqrt{n}\log^{2}n.
\]
\end{proposition}

\begin{proof} By Proposition \ref{metric}, the entropy of random
walk on $\Delta$ is related to the entropy on the factors $\Delta_{s}$
by the following 
\[
\max_{s\geq0}H(f_{s}^{W_{n}},S_{n})\leq H((f_{s}^{W_{n}}),S_{n})\leq H(S_{n})+\sum_{s\leq s_{2}(n)}H(f_{s}^{W_{n}}),
\]
where $s_{0}(W_{n})\leq s_{2}(n)=\max\{s:\ k_{s}\leq n\}$ and $H(S_{n})\simeq\log n$.
The lower bound comes from the first factor $H(f_{0}^{W_{n}},S_{n})$
which is the usual random walk on lamplighter group with finite lamps.

Denote $T_{s}^{n}=T(k_{s},\cdot,n)$ the traverse time function and
$[supp{T}_{s}^{n}]$ the convex envelope of its support. Using conditional
entropy (see for instance \cite{KV} or \cite{AmirVirag2012} for
usual properties of conditional entropy) 
\begin{align}
H(f_{s}^{W_{n}}) & \leq H(f_{s}^{W_{n}}|{T}_{s}^{n})+H({T}_{s}^{n})\nonumber \\
 & \leq H(f_{s}^{W_{n}}|{T}_{s}^{n})+H({T}_{s}^{n}|[supp{T}_{s}^{n}])+H([supp{T}_{s}^{n}]).\label{eq:entropy-upper}
\end{align}

The convex envelope $[supp{T}_{s}^{n}]$ is included in $Range(W_{n})$,
therefore $H([supp{T}_{s}^{n}])\simeq\log n$ and $\E|[supp{T}_{s}^{n}]|\leq C\sqrt{n}$.
As for all $z$, $0\leq T_{s}^{n}(z)\leq n$, we deduce that 
\[
H({T}_{s}^{n}|[supp{T}_{s}^{n}])\leq\E\left[|[supp{T}_{s}^{n}]|\right]\log n\leq C\sqrt{n}\log n
\]
.

As each $f_{s}^{W_{n}}(z)$ is distributed as an independent sample
of a random walk on $\Gamma_{s}$ of length $T(k_{s},z,n)$, we have
\[
H(f_{s}^{W_{n}}|T_{s}^{n})=\sum_{z\in\Z}\E H(X_{T(k_{s},z,n)}^{(s)}).
\]
The groups $\Gamma_{s}$ being dihedral $H(X_{t}^{(s)})\leq\log t$,
therefore 
\[
H(f_{s}^{W_{n}}|T_{s}^{n})\leq\sum_{z\in\Z}\E\log T(k_{s},z,n)\leq C\sqrt{n}\log n.
\]
Obviously $n\geq k_{s_{0}(W_{n})}$ so finally piling up the inequalities
\[
H((f_{s}^{W_{n}}),S_{n})\leq C\log(n)\left(1+2\sqrt{n}\log n+\log n\right)
\]
\end{proof}

\subsection{Joint evaluation of speed and entropy}

\label{sub: joint}

Using the idea of Amir \cite{Amir2015} of taking direct product of
two groups, we can combine the speed and entropy estimates on $\Delta$,
together with the results of Amir-Virag \cite{AmirVirag2012} to show
the following result concerning the joint behavior of growth of entropy
and speed.

Recall that for symmetric probability measure $\mu$ on $G$ with
finite support, entropy and speed satisfies 
\[
\frac{1}{n}\left(\frac{1}{4}L_{\mu}(n)\right)^{2}-1\le H_{\mu}(n)\leq(v+\varepsilon)L_{\mu}(n)+\log n+C,
\]
where $v$ is the volume growth rate of $\left(G,\mbox{supp}\mu\right)$ and
$C>0$ is an absolute constant, \cite{Erschler2003,AmirVirag2012},\cite{AmirVirag2012}
.

\begin{proposition}\label{entropy-speed-function}

Let $f,h:\mathbb{N}\to\mathbb{N}$ be two functions such that $h(1)=1$
and 
\begin{itemize}
\item either $\frac{f(n)}{n^{\frac{3}{4}}}$ and $\frac{n^{1-\epsilon}}{f(n)}$
are non-decreasing for some $\epsilon>0$ and 
\[
h(n)\le f(n)\le\sqrt{\frac{nh(n)}{\log n}},
\]

\item or $\frac{h(n)}{n^{\frac{1}{2}}\log^{2}n}$ and $\frac{n^{\frac{3}{4}}}{f(n)}$
are non-decreasing and 
\[
h(n)\le f(n)\le\sqrt{nh(n)},
\]

\end{itemize}
Then there exists a constant $C>0$ depending only on $\epsilon>0$,
a finitely generated group $G$ and a symmetric probability measure
$\mu$ of finite support on $G$ such that 
\[
L_{\mu}(n)\simeq_{C}f(n)\mbox{ and }H_{\mu}(n)\simeq_{C}h(n).
\]

\end{proposition}

As a corollary, we derive the Corollary \ref{entropy-speed-exponent}
regarding joint entropy and speed exponents (Conjecture 3 in \cite{Amir2015}).

\begin{proof}[Proof of Corollary \ref{entropy-speed-exponent}]
If $\theta\in\left[\frac{1}{2},1\right)$, take functions $h(n)=\max\left\{ n^{\frac{1}{2}}\log^{2}n,n^{\theta}\right\} $,
$f(n)=\max\left\{ n^{\gamma},h(n)\right\} $. Note that in this case
$\gamma<1$, the pair of functions $f$ and $h$ is covered by one
of the cases in Proposition \ref{entropy-speed-function}, the statement
follows.

If $\theta=1$, in this case $\gamma=1$ as well, we can take $G$
to be any finitely generated group which admits a symmetric probability
measure $\mu$ of finite support such that $(G,\mu)$ has linear entropy
growth. \end{proof}

\begin{proof} We follow Amir's approach in \cite{Amir2015} to take
direct product of two appropriate groups such that one would control
the speed function and the other would control the entropy function.

In the first case where both $f(n)/n^{\frac{3}{4}}$ and $n^{1-\epsilon}/f(n)$
are non-decreasing, consider the direct product of the following two
groups. By \cite{AmirVirag2012}, there exists a group $G_{1}=\mathbb{Z}\wr_{\mathcal{S}}\mathcal{M}_{m}$
and a step distribution $q_{1}$ on $G_{1}$ such that 
\[
L_{q_{1}}(n)\simeq f(n)\mbox{ and }H_{q_{1}}(n)\simeq\frac{f(n)^{2}}{n}\log(n+1).
\]
By Theorem \ref{possible-speed-linear}, there exists a group $\Delta$
and step distribution $q_{2}$ on $\Delta$ such that 
\[
L_{q_{2}}(n)\simeq H_{q_{2}}(n)\simeq h(n).
\]
Then on the direct product $G_{1}\times\Delta$ with step distribution
$q_{1}\otimes q_{2}$, 
\begin{align*}
L_{q_{1}\otimes q_{2}}(n) & \simeq\max\left\{ L_{q_{1}}(n),L_{q_{2}}(n)\right\} \simeq f(n),\\
H_{q_{1}\otimes q_{2}}(n) & \simeq\max\left\{ H_{q_{1}}(n),H_{q_{2}}(n)\right\} \simeq\max\left\{ \frac{f(n)^{2}}{n}\log(n+1),h(n)\right\} =h(n).
\end{align*}

In the second case where $f(n)/\left(n^{\frac{1}{2}}\log^{2}n\right)$
and $n^{\frac{3}{4}}/f(n)$ are non-decreasing, by Theorem \ref{possible-speed-dihedral}
and Proposition \ref{dihedral-entropy}, there exists group $\Delta_{1}$ and
step distribution $q'_{1}$ on $\Delta_{1}$ such that 
\[
L_{q'_{1}}(n)\simeq f(n),\mbox{ and }\frac{1}{C}n^{\frac{1}{2}}\le H_{q'_{1}}(n)\le Cn^{\frac{1}{2}}\log^{2}n.
\]
By Theorem \ref{possible-speed-linear}, there exists a group $\Delta_{2}$
and step distribution $q'_{2}$ such that 
\[
L_{q'_{2}}(n)\simeq H_{q'_{2}}(n)\simeq h(n).
\]
Then on the direct product $\Delta_{1}\times\Delta_{2}$ with step
distribution $q'_{1}\otimes q'_{2}$, 
\begin{align*}
L_{q'_{1}\otimes q'_{2}}(n)\simeq\max\left\{ L_{q'_{1}}(n),L_{q'_{2}}(n)\right\}  & \simeq f(n),\\
H_{q'_{1}\otimes q'_{2}}(n)\simeq\max\left\{ H_{q'_{1}}(n),H_{q'_{2}}(n)\right\}  & \simeq h(n).
\end{align*}
\end{proof}

\begin{remark}

This method permits to prove Corollary \ref{entropy-speed-exponent}
but can't handle functions that oscillate cross the $n^{\frac{3}{4}}$
borderline. The problem can be reduced to finding extremal examples
where the speed follows a prescribed function while entropy growth
is as slow as possible. The Amir-Virag result covers the case where
speed grows at least like $n^{\frac{3}{4}}$. The construction of
diagonal product $\Delta$ is not designed to achieve such a goal.

\end{remark}

\section{Isoperimetric profiles and return probability\label{sec:isoperimetric-profiles}}

In this section we consider the isoperimetric profiles and return
probability of $\Delta$ when $\left\{ \Gamma_{s}\right\} $ are taken
to be expanders or dihedral groups. For convenience of calculation,
we take the switch-or-walk measure $\mathfrak{q}=\frac{1}{2}(\mu+\nu)$
on $\Delta$ where $\mu$ is the simple random walk measure on the
base $\mathbb{Z}$, $\mu\left(\tau^{\pm1}\right)=\frac{1}{2}$, and
$\nu$ is uniform on $\left\{ \alpha_{i},\beta_{j}:\ 1\le i\le\left|A\right|,1\le j\le\left|B\right|\right\} $.
Let $\mathfrak{q}_{s}$ be the projection of $\mathfrak{q}$ to the
quotient $\Delta_{s}$.

\subsection{Isoperimetric profiles}

\label{sub: profile}

We first recall some background information. Given a symmetric probability
measure $\phi$ on $G$, the $p$-Dirichlet form associated with $(G,\phi)$
is 
\[
\mathcal{E}_{p,\phi}(f)=\frac{1}{2}\sum_{x,y\in G}|f(xy)-f(x)|^{p}\phi(y)
\]
and the $\ell^{p}$-isoperimetric profile $\Lambda_{p,G,\phi}:[1,\infty)\to\mathbb{R}$
is defined as 
\begin{equation}
\Lambda_{p,G,\phi}(v)=\inf\left\{ \mathcal{E}_{p,\phi}(f):|\mbox{support}(f)|\le v,\|f\|_{p}=1\right\} .\label{p-profile}
\end{equation}
The most important ones are the $\ell{}^{1}$ and $\ell^{2}$-isoperimetric
profiles. Using an appropriate discrete co-area formula, $\Lambda_{1,\phi}$
can equivalently be defined by 
\[
\Lambda_{1,G,\phi}(v)=\inf\left\{ |\Omega|^{-1}\sum_{x,y}\mathbf{1}_{\Omega}(x)\mathbf{1}_{G\setminus\Omega}(xy)\phi(y):|\Omega|\le v\right\} .
\]
If we define the boundary of $\Omega$ to be the set 
\[
\partial\Omega=\left\{ (x,y)\in G\times G:x\in\Omega,y\in G\setminus\Omega\right\} 
\]
and set $\phi(\partial\Omega)=\sum_{x\in\Omega,xy\in G\setminus\Omega}\phi(y)$
then 
\[
\Lambda_{1,G,\phi}(v)=\inf\{\phi(\partial\Omega)/|\Omega|:|\Omega|\le v\}.
\]
When $\phi$ a symmetric measure supported by a finite generating
set $S$, $\Lambda_{1,G,\phi}(v)$ is closely related to the Følner
function $\mbox{Føl}_{G,S}:(0,\infty)\to\mathbb{N}$ defined as 
\[
\mbox{Føl}_{G,S}(r)=\min\left\{ |\Omega|:\ \Omega\subset G,\ \frac{\left|\partial_{S}\Omega\right|}{\left|\Omega\right|}<\frac{1}{r}\right\} ,
\]
where $\left|\partial_{S}\Omega\right|=\{x\in\Omega:\ \exists u\in S,\ xu\notin\Omega\}.$
Namely, let $p_{\ast}=\min\{\phi(u):\ u\in S,u\neq id\}$, then 
\[
\Lambda_{1,G,\phi}^{-1}(1/r)\le\mbox{Føl}_{G,S}(r)\le\Lambda_{1,G,\phi}^{-1}(p_{\ast}/r),
\]
where $\Lambda_{1,G,\phi}^{-1}$ is the generalized inverse of $\Lambda_{1,G,\phi}$
.

We will repeatedly use the following two facts.

For any $1\le p\le q\le2$, the isoperimetric profiles $\Lambda_{p,G,\phi}$
and $\Lambda_{q,G,\phi}$ are related by Cheeger type inequality

\begin{equation}
c_{0}\Lambda_{p,G,\phi}^{q/p}\le\Lambda_{q,G,\phi}\le C_{0}\Lambda_{p,G,\phi},\label{eq:cheeger-pq}
\end{equation}
where $c_{0},C_{0}$ are absolute constants, see \cite{Lawler1988},
\cite[Proposition 2.8]{Saloff-Coste2015a}.

Let $H$ be a quotient group of $G$, $\bar{\phi}$ be the projection
of $\phi$ on $H$, then by \cite[Proposition 4.5]{Tessera2013},
for all $1\le p<\infty$, 
\[
\Lambda_{p,G,\phi}\ge\Lambda_{p,H,\bar{\phi}}.
\]

\subsubsection{Isoperimetric profiles of one factor}

Let $\left\{ \Gamma_{s}\right\} $ be a family of expanders, as in
Example \ref{examples}. Let $\nu=\nu_s$ be the uniform distribution on
$A(s)\cup B(s)$ in $\Gamma_s$. Denote $h(\Gamma_{s},\nu)$ the Cheeger constant
\[
h\left(\Gamma_{s},\nu\right)=\inf\left\{ \Lambda_{1,\Gamma_{s},\nu}(v):\ v\le\frac{|\Gamma_{s}|}{2},\ v<\infty\right\} .
\]
In one copy $\Delta_{s}=\Gamma_{s}\wr\mathbb{Z}$, we establish the
following lower bound of $\ell^{p}$-isoperimetric profile of $\Delta_{s}$.

\begin{lemma}\label{iso-one}

Let $p\in[1,2]$, $\Gamma_{s}$ be a finite group marked with generating
subgroups $A(s),B(s)$. Let $\mathfrak{q}_{s}$ be the uniform distribution on
$\{\tau^{\pm 1}\}\cup A(s)\cup B(s)$ in $\Delta_{s}$. Then there exists an absolute constant $C>1$ such
that the following is true. 
\begin{itemize}
\item (Slow phase) for $1\le v\le2^{k_{s}/2}$, 
\[
\Lambda_{p,\Delta_{s},\mathfrak{q}_{s}}(v)\ge\frac{\left(\log_{2}v\right)^{-p}}{C\left(|A(s)|+|B(s)|\right)}.
\]

\item (Fast, then slow down) for $\left|\Gamma_{s}\right|^{r}\le v\le\left|\Gamma_{s}\right|^{r+1}$,
$r\ge k_{s}$, 
\[
\Lambda_{p,\Delta_{s},\mathfrak{q}_{s}}(v)\ge\frac{h(\Gamma_{s},\nu)^{p}}{C\left(|A(s)|+|B(s)|\right)r^{p}}.
\]

\end{itemize}
\end{lemma}

\begin{remark}

By monotonicity of the profile function $\Lambda_{p,\Delta_{s},q_{s}}$,
we have that 
\[
\Lambda_{p,\Delta_{s},\mathfrak{q}_{s}}(v)\ge\frac{h(\Gamma_{s},\nu)^{p}}{C\left(|A(s)|+|B(s)|\right)k_{s}^{p}},\ \mbox{for }v\in\left(2^{k_{s}/2},\left|\Gamma_{s}\right|^{k_{s}}\right).
\]

\end{remark}

\begin{proof} We prove a lower bound for the $\ell^{1}$-isoperimetric
profile $\Lambda_{1,\Delta_{s},\mathfrak{q}_{s}}$ and use Cheeger
inequality (\ref{eq:cheeger-pq}) to derive a lower bound on $\Lambda_{p,\Delta_{s},\mathfrak{q}_{s}}$.
Given $r\ge1$, consider the product of $\Gamma_{s}$ in the zero-section
over the segment $[0,r]$, namely 
\[
\Pi_{r}=\prod_{x\in[0,r]}\left(\Gamma_{s}\right)_{x}.
\]
We now construct a product kernel $\zeta_{r}$ on $\Pi_{r}$ and discuss
its $\ell^{1}$-isoperimetry.

Phase I\emph{:} in the first phase $r<k_{s}/2$. Let $\eta_{x}$ denote
the uniform measure on the finite subgroup $A_{s}\simeq\mathbb{Z}/2\mathbb{Z}$
in the copy $\left(\Gamma_{s}\right)_{x}$. Let $\zeta_{r}$ denote
the product kernel 
\[
\zeta_{r}=\eta_{0}\otimes\ldots\otimes\eta_{r}.
\]
Then $\zeta_{r}$ is indeed the uniform measure on the subgroup $\prod_{x\in[0,r]}\left(A_{s}\right)_{x}$
in the zero section of $\Delta_{s}$, it follows that 
\[
\Lambda_{1,\Delta_{s},\zeta_{r}}(v)\ge\frac{1}{2}\ \mbox{for all }v\le\frac{1}{2}\left(2^{r+1}\right).
\]

Phase II\emph{:} in the second phase $r>k_{s}$. Let $\nu_{x}$ denote
the uniform measure on the generating set $A(s)\cup B(s)$ in the
copy $\left(\Gamma_{s}\right)_{x}$. Let $\zeta_{r}$ denote the product
kernel 
\[
\zeta_{r}=\nu_{0}\otimes\ldots\otimes\nu_{r}.
\]
As a transition kernel, $\zeta_{r}$ changes every copy $\left(\Gamma_{s}\right)_{x}$
independently according to $\nu_{x}$. By \cite[Theorem 1.1]{Bobkov1997},
the Cheeger constant of $\zeta_{r}$ on $\Pi_{r}$ satisfies 
\[
h(\Pi_{r},\zeta_{r})\ge\frac{1}{2\sqrt{6}}h(\Gamma_{s},\nu).
\]
In other words, for $r>k_{s}$ 
\[
\Lambda_{1,\Delta_{s},\zeta_{r}}(v)\ge\frac{1}{2\sqrt{6}}h(\Gamma_{s},\nu)\ \mbox{for all }v\le\frac{1}{2}\left|\Gamma_{s}\right|^{r+1}.
\]
By Cheeger inequality (\ref{eq:cheeger-pq}), for $p\in[1,2]$, 
\[
\Lambda_{p,\Delta_{s},\zeta_{r}}(v)\ge c_{0}\Lambda_{1,\Delta_{s},\zeta_{r}}(v)^{p}.
\]
Now we go back to the simple random walk kernel $\mathfrak{q}$. By
construction of the transition kernel $\zeta_{r}$ in both cases,
the metric estimate in Lemma \ref{one_metric} implies every element
$g$ in the support of $\zeta_{r}$ satisfies 
\[
\left|g\right|_{\Delta_{s}}\le40r.
\]
By the standard path length argument (see \cite[Lemma 2.1]{PittetSaloffCoste2000}),
we have comparison of Dirichlet forms on $\Delta_{s}$ that 
\[
2\left(|A|+|B|\right)\mathcal{E}_{p,\Delta_{s},\mathfrak{q}_{s}}\ge\frac{1}{(40r)^{p}}\mathcal{E}_{p,\Delta_{s},\zeta_{r}}.
\]
It follows that 
\begin{itemize}
\item for $r<\frac{k_{s}}{2}$, 
\[
\Lambda_{p,\Delta_{s},\mathfrak{q}_{s}}\left(v\right)\ge\frac{c_{0}}{2\left(|A|+|B|\right)(40r)^{p}}\ \mbox{for all }v\le2^{r};
\]

\item for $r>k_{s}$, 
\[
\Lambda_{p,\Delta_{s},\mathfrak{q}_{s}}\left(v\right)\ge\frac{c_{0}h(\Gamma_{s},\nu)^{p}}{2\left(|A|+|B|\right)\left(80\sqrt{6}r\right)^{p}}\ \mbox{for all }v\le\frac{1}{2}\left|\Gamma_{s}\right|^{r+1}.
\]

\end{itemize}
\end{proof}

When $\Gamma_{s}=\Gamma$ is an infinite group, we have the following
bound.

\begin{lemma}\label{iso-one-infty}

Let $p\in[1,2]$, $\Gamma_{s}=\Gamma$ be an infinite group marked
with generating subgroups $A,B$. Let $\nu$ be the uniform distribution
on $A\cup B$. Then there exists an absolute constant $C>1$ such
that the following is true. 
\begin{itemize}
\item for $1\le v\le2^{k_{s}/2}$, 
\[
\Lambda_{p,\Delta_{s},\mathfrak{q}_{s}}(v)\ge\frac{\left(\log_{2}v\right)^{-p}}{C\left(|A(s)|+|B(s)|\right)}.
\]

\item for $v>2^{k_{s}/2}$, 
\[
\Lambda_{p,\Delta_{s},\mathfrak{q}_{s}}(v)\ge\frac{h(\Gamma,\nu)^{p}}{C\left(|A(s)|+|B(s)|\right)k_{s}^{p}}.
\]

\end{itemize}
\end{lemma}

\begin{proof}

The first item is the same as in proof of Lemma \ref{iso-one}. For
the second item, consider the copy of $\Gamma$ at $0$ and regard
$\nu$ as a measure supported on the subgroup $(\Gamma)_{0}$, then
\[
C\left(|A(s)|+|B(s)|\right)k_{s}^{p}\mathcal{E}_{p,\Delta_{s},\mathfrak{q}_{s}}\ge\mathcal{E}_{p,\Gamma,\nu},
\]
and the result follows.

\end{proof}

\subsubsection{Isoperimetric profile of the diagonal product\label{sub:Isoperimetry-delta}}

First we put together isoperimetric estimates on the copies $\Delta_{s}$
to describe isoperimetric profile of the diagonal product $\Delta$.
Let us denote $\ell_{s}=\log|\Gamma_{s}|$. Mind the difference with
the diameter $l_{s}$ of $\Gamma_{s}$. For a family of expanders,
these two quantities differ only by multiplicative constants depending
only on the volume growth and the spectral gap of $\left\{ \Gamma_{s}\right\} $.

\begin{proposition} \label{iso-Delta}

Suppose $\left\{ \left(\Gamma_{s},A(s)\cup B(s)\right)\right\} $
with $A(s)\simeq A,B(s)\simeq B$ is a family of groups with Cheeger
constant $h(\Gamma_{s},\nu_{s})\ge\delta>0$, where $\nu_{s}$ is
uniform on $A(s)\cup B(s)$. Suppose $\{k_{s}\}$ satisfies growth
assumption \ref{k_growth}. Let $\Delta$ be the diagonal product
constructed with $\left\{ \Gamma_{s}\right\} $ and parameters $\{k_{s}\}$, and $\mathfrak{q}$ be the uniform measure on $\{\tau^{\pm 1}\} \cup A\cup B$ in $\Delta$.

There exists an absolute constant $C>1$ such that the following estimates
hold for any $s\ge0$, $p\in[1,2]$. 
\begin{enumerate}
\item For volume $v\in\left[e^{k_{s}\ell_{s}},e^{k_{s+1}\ell_{s}}\right)$,
\begin{align*}
\Lambda_{p,\Delta,\mathfrak{q}}(v) & \ge\frac{1}{|A|+|B|}\left(\frac{\delta\ell_{s}}{C\log v}\right)^{p},\\
\Lambda_{p,\Delta,\mathfrak{q}}\left(v^{\frac{\sum_{j\le s}\ell_{j}}{\ell_{s}}}\right) & \le\left(\frac{C\ell_{s}}{\log v}\right)^{p}.
\end{align*}

\item For volume $v\in\left[e^{k_{s+1}\ell_{s}},e^{k_{s+1}\ell_{s+1}}\right]$,
\begin{align*}
\Lambda_{p,\Delta,\mathfrak{q}}(v) & \ge\frac{1}{|A|+|B|}\left(\frac{\delta}{Ck_{s+1}}\right)^{p},\\
\Lambda_{p,\Delta,\mathfrak{q}}\left(v\right) & \le\left(\frac{C}{k_{s+1}}\right)^{p}\ \mbox{if }v\ge\exp\left(\left(\sum_{j\le s}\ell_{j}\right)k_{s+1}\right).
\end{align*}

\end{enumerate}
\end{proposition}

The upper bounds are valid without the requirement of positive Cheeger constant.

\begin{proof} Let $U_{r}^{\Delta}=\{((f_{s}),z):\ Range(f_{s},z)\subset[-r,r]\}$
and take a function supported on the subset $U_{r}^{\Delta}$, 
\[
\varphi_{r}\left((f_{s}),z\right)=\left(1-\frac{\left|z\right|}{r}\right)\mathbf{1}_{U_{r}^{\Delta}}\left(\left((f_{s}),z\right)\right).
\]
Let $U_{r}^{\Delta}(0)=\{((f_{s}),0):\ Range(f_{s},0)\subset[-r,r]\}$,
then $U_{r}^{\Delta}$ can be viewed as the product of $U_{r}^{\Delta}(0)$
and the interval $[-r,r]$. To compute the Rayleigh quotient of the
function $\varphi_{r}$, first note that $\varphi_{r}\left(Z\alpha_{i}\right)=\varphi_{r}(Z\beta_{j})=\varphi_{r}(Z)$
for all $Z\in\Delta$ and $\alpha_{i}\in A$, $\beta_{j}\in B$. For
the generator $\tau$, 
\begin{align*}
\sum_{\left(\left(f_{s}\right),z\right)\in U_{r}^{\Delta}}\left|\varphi_{r}\left(\left(f_{s}\right),z+1\right)-\varphi_{r}\left(\left(f_{s}\right),z\right)\right|^{p} & =\frac{1}{r^{p}}(2r)\left|U_{r}^{\Delta}(0)\right|,\\
\sum_{\left(\left(f_{s}\right),z\right)\in U_{r}^{\Delta}}\varphi_{r}\left(\left(f_{s}\right),z\right)^{p} & =\sum_{z\in[-r,r]}\left(1-\frac{\left|z\right|}{r}\right)^{p}\left|U_{r}^{\Delta}(0)\right|.
\end{align*}
Therefore 
\[
\frac{\mathcal{E}_{p,\Delta,\mathfrak{q}}(\varphi_{r})}{\left\Vert \varphi_{r}\right\Vert _{p}^{p}}\sim\frac{1+p}{2r^{p}}.
\]
For the size of support of $\varphi_{r}$, 
\[
|\mbox{supp}\varphi_{r}|\leq\prod_{k_{s}\leq2r}|\mbox{supp}\varphi_{r}^{s}|\leq\prod_{k_{s}\leq2r}|\Gamma_{s}|^{r}=e^{r\sum_{k_{s}\leq2r}\ell_{s}}.
\]

In the first interval $v\in\left[e^{k_{s}\ell_{s}},e^{k_{s+1}\ell_{s}}\right)$,
let 
\[
r=\frac{\log v}{\ell_{s}}
\]
and test function $\varphi_{r}^{\Delta}$ gives the upper bound on
$\Lambda_{p,\Delta,\mathfrak{q}}$ stated. For the lower bound on
$\Lambda_{p,\Delta,\mathfrak{q}}$, consider the projection to the
quotient $\Delta_{s}$. Then from the first item in Lemma \ref{iso-one},
we have for $v\in\left[e^{k_{s}\ell_{s}},e^{k_{s+1}\ell_{s}}\right)$,
\[
\Lambda_{p,\Delta,\mathfrak{q}}(v)\ge\Lambda_{p,\Delta_{s},\mathfrak{q}_{s}}(v)\ge\frac{1}{|A|+|B|}\left(\frac{\delta}{Ck_{s}}\right)^{p}.
\]

In the second interval $v\in\left[e^{k_{s+1}\ell_{s}},e^{k_{s+1}\ell_{s+1}}\right]$,
first consider the projection to the quotient $\Delta_{s+1}$. The
second item in Lemma \ref{iso-one} provides

\[
\Lambda_{p,\Delta,\mathfrak{q}}\left(\frac{1}{2}\left|\Gamma_{s+1}\right|^{k_{s+1}}\right)\ge\Lambda_{p,\Delta_{s},\mathfrak{q}_{s}}\left(\frac{1}{2}\left|\Gamma_{s+1}\right|^{k_{s+1}}\right)\ge\frac{1}{|A|+|B|}\left(\frac{\delta}{Ck_{s+1}}\right)^{p}.
\]
In the upper bound direction, note that the right end point in the
first interval gives 
\[
\Lambda_{p,\Delta,\mathfrak{q}}\left(\exp\left(\left(\sum_{j\le s}\ell_{j}\right)k_{s+1}\right)\right)\le\left(\frac{C}{k_{s+1}}\right)^{p}.
\]
The statement follows from monotonicity of $\Lambda_{p,\Delta,\mathfrak{q}}$.
\end{proof}

\begin{example}\label{exponents-iso} A direct application of Proposition
\ref{iso-Delta} shows that when $k_{s}=2^{\beta s}$ and $\ell_{s}=2^{\iota s}$
with $\beta,\iota>0$, then for $p\in[1,2]$ 
\[
\Lambda_{p,\Delta,\mathfrak{q}}(v)\simeq(\log v)^{-\frac{p}{1+\frac{\iota}{\beta}}},
\]
and the exponent $\frac{p}{1+\frac{\iota}{\beta}}$ can take any value
in $(0,p)$. \end{example}

We allow the sequence $(k_{s})$, $(l_{s})$ to take the value $\infty$,
the bounds are still valid. In our convention, $k_{s+1}=\infty$ means
$\Delta_{s+1}$ is trivial, in this case we only use the first item
in Proposition \ref{iso-Delta}, which covers $v\in\left[e^{k_{s}\ell_{s}},\infty\right)$.
The bounds in Proposition \ref{iso-Delta} are good when $\left\{ \ell_{s}\right\} $
grows at least exponentially. In particular, from these estimates
of isoperimetric profiles we deduce that $\Lambda_{p,\Delta,\mathfrak{q}}\circ\exp$
can follow a prescribed function satisfying some log-Lipschitz condition.

\begin{theorem} \label{possible-profile}

There exists universal constants $c,C>0$ such that for any $p\in[1,2]$
and for any non-decreasing function $\varrho(x)$ such that $\frac{x^{p}}{\varrho(n)}$
is non-decreasing, there is a group $\Delta$ such that 
\[
\forall v\geq3,\ \frac{c}{\varrho(\log v)}\leq\Lambda_{p,\Delta,\mathfrak{q}}(v)\leq\frac{C}{\varrho(\log v)}.
\]
\end{theorem}

\begin{proof} We write $\varrho(x)=\left(\frac{x}{f(x)}\right)^{p}$
with $f(x)$ between $1$ and $x$. The sets $K=\Z_{+}\cup\{\infty\}$ and $L=\left\{ \log|\Gamma_{m}|,m\geq1\right\}\cup\{\infty\} $
where $\{\Gamma_s\}$ are groups in the family of Examples \ref{examples}
satisfy the assumptions of Proposition \ref{function-approximation}.
So we can find sequences $(k_{s}),(l_{s})$ taking values in $K$
and $L$ such that the function defined by $\tilde{f}(x)=l_{s}$ on
$[k_{s}l_{s},k_{s+1}l_{s}]$ and $\tilde{f}(x)=\frac{x}{k_{s+1}}$
on $[k_{s+1}l_{s},k_{s+1}l_{s+1}]$ satisfies $\tilde{f}(x)\simeq_{m_{0}C_{1}^{5}}f(x)$.
Since the infinite group $\Gamma$ in Example \ref{examples} has
Property $(T)$, there exists a constant $\delta>0$ such that the
Cheeger constants $h(\Gamma_{s},\nu_{s})\ge\delta$ for all $s\ge1$.

We use Proposition \ref{iso-Delta} to evaluate the profile of the
group $\Delta$ associated to these sequences. The lower bounds show
that for all $x\geq1$ 
\[
\Lambda_{p,\Delta,\mathfrak{q}}\circ\exp(x)\geq\left(\frac{\delta\tilde{f}(x)}{Cx}\right)^{p}\geq\frac{c\delta^{p}}{\varrho(x)}.
\]

As $\sum_{j\leq s}\ell_{j}\leq\frac{1}{1-\frac{1}{m_{0}}}\ell_{s}$,
making the change of variable $x=\left(1-\frac{1}{m_{0}}\right)y=\log v$,
the first upper bound shows that 
\[
\Lambda_{p,\Delta,\mathfrak{q}}\circ\exp(y)\leq\left(\frac{Cl_{s}}{\log v}\right)^{p}=\left(\frac{C\tilde{f}\left((1-\frac{1}{m_{0}})y\right)}{(1-\frac{1}{m_{0}})y}\right)^{p}\leq\left(\frac{C\tilde{f}(y)}{y}\right)^{p}\leq\frac{C'}{\varrho(x)}
\]
for $\frac{1}{1-\frac{1}{m_{0}}}k_{s}\ell_{s}\leq y\leq\frac{1}{1-\frac{1}{m_{0}}}k_{s+1}\ell_{s}$.
The second upper bound shows that 
\[
\Lambda_{p,\Delta,\mathfrak{q}}\circ\exp(y)\leq\left(\frac{C}{k_{s+1}}\right)^{p}=\left(\frac{C\tilde{f}(k_{s+1}l_{s})}{k_{s+1}l_{s}}\right)^{p}=\left(\frac{C\tilde{f}\left((1-\frac{1}{m_{0}})y\right)}{(1-\frac{1}{m_{0}})y}\right)^{p}\leq\frac{C''}{\varrho(y)}
\]
for $\frac{1}{1-\frac{1}{m_{0}}}k_{s+1}\ell_{s}\leq y\leq\frac{1}{1-\frac{1}{m_{0}}}k_{s+1}\ell_{s+1}$.
We used the fact that $\tilde{\varrho}(x)$ is constant on the interval
$[k_{s+1}\ell_{s},k_{s+1}\ell_{s+1}]$.

\end{proof}

We derive the following corollary regarding Følner functions from
Theorem \ref{main-expanders}. The definition of Følner function is
recalled in the beginning of Section \ref{sec:isoperimetric-profiles}. We use the
convention that on a non-amenable group $G$, if $1/r\le\inf\frac{\left|\partial S\right|}{|S|}$,
then $\mbox{Føl}_{G,S}(r)=\infty$.

\begin{corollary}\label{folner}

There exists an universal constant $C>1$. Let $g:[1,\infty)\to[1,\infty]$
be any non-decreasing function with $g(1)=1$ and $\frac{\log(g(x))}{x}$
non-decreasing. Then there exists a group $\Delta$ marked with
finite generating set $T$ such that 
\[
g(r/C)\le\mbox{Føl}_{\Delta,T}(r)\le g(Cr).
\]
Further, when range of $g$ is contained in $[1,\infty)$, the group
$\Delta$ constructed is elementary amenable and there exists a symmetric
probability measure $q$ with finite generating support on $\Delta$
such that $\left(\Delta,q\right)$ is Liouville.

\end{corollary}

\begin{proof}

Let $p_{\ast}=\min\{q(u):\ u\in\mathcal{T},u\neq id\}$, then $p_{\ast}\ge\frac{1}{2\left(|A|+|B|\right)}$.
Recall that by definition of the Følner function 
\[
\Lambda_{1,\Delta,q}^{-1}(1/r)\le\mbox{Føl}_{\Delta,\mathcal{T}}(r)\le\Lambda_{1,\Delta,q}^{-1}(p_{\ast}/r).
\]
By Proposition \ref{possible-profile}, exists a universal constants
$C>0$ such that for any function $\varrho(x)$ between $1$ and $x$,
there is a group $\Delta$ such that 
\[
\forall v\geq3,\ \frac{1}{C\varrho(\log v)}\leq\Lambda_{1,\Delta,\mathfrak{q}}(v)\leq\frac{C}{\varrho(\log v)}.
\]
In particular, in the construction of $\Delta$ we can choose $\{\Gamma_{s}\}$
from Lafforgue's expanders as in Example \ref{examples}, where $\left|A\right|=2$,
$\left|B\right|=r_{0}$ for some fixed $r_{0}$. Therefore 
\[
\exp\left(\varrho^{-1}(r/C)\right)\le\mbox{Føl}_{\Delta,\mathcal{T}}(r)\le\exp(\varrho^{-1}(Cr_{0}r)).
\]
Since $\varrho$ is any function between $1$ and $x$, the statement
about Følner function follows.

When the range of $g$ is in $[1,\infty)$, the group $\Delta$ in
the proof of Proposition \ref{possible-profile} is constructed with
an infinite sequence finite groups $\left\{ \Gamma_{s}\right\} $
and $\{k_{s}\}$ satisfying growth assumption (\ref{k_growth}). By
Fact \ref{EA}, $\Delta$ is elementary amenable. Apply
Theorem \ref{Delta-speed}, we have that $L_{q}(\mu)$ and $H_{q}(\mu)$
has sub-linear growth, thus $(\Delta,q)$ is Liouville.

\end{proof}

\subsection{Return probability of simple random walk on $\Delta$}

\label{return}

By Theorem \ref{possible-profile} with $p=2$, we have that $\Lambda_{2,\Delta,\mathfrak{q}}\circ\exp$
can follow a prescribed function satisfying some log-Lipschitz condition.
Now we turn the $\ell^{2}$-isoperimetric profile estimates into return
probability bounds using the Coulhon-Grigor'yan theory. Let $\mu$
be a symmetric probability measure on a group $G$. Between discrete
time random walk and continuous time random walk, we have, see \cite[Section 3.2]{PittetSaloffCoste2000},
\[
\mu^{(2n+2)}(e)\le2h_{2n}^{\mu}(e)\mbox{ and }h_{4n}^{\mu}(e)\le e^{-2n}+\mu^{(2n)}(e)
\]
where 
\begin{equation}
h_{t}^{\phi}=e^{-t}\sum_{0}^{\infty}\frac{t^{k}}{k!}\phi^{(k)}.\label{cont-time}
\end{equation}
Define the function $\psi:[0,+\infty)\to[1,+\infty)$ implicitly by
\begin{equation}
t=\int_{1}^{\psi(t)}\frac{dv}{v\Lambda_{2,G,\mu}(v)}.\label{eq:GC}
\end{equation}
Then by \cite[Proposition II.1]{Coulhon1996}, we have 
\[
\mu^{(2n+2)}(e)\le\frac{8}{\psi(8n)}.
\]
In the current context, it is convenient to do a change of variable
in (\ref{eq:GC}), set $v=\exp(s)$, 
\begin{equation}
t=\int_{1}^{w(t)}\frac{ds}{\Lambda_{2,\Delta,\mathfrak{q}}\circ\exp(s)}.\label{eq:GC-exp}
\end{equation}
If in addition $\Lambda_{2,G,\mu}\circ\exp$ is doubling, namely $\Lambda_{2,G,\mu}\circ\exp(2s)\ge c\Lambda_{2,G,\mu}\circ\exp(s)$
for all $s>1$, then by \cite[Proposition 2.3]{Bendikov2012}, $w'(t)$
is doubling with the same constant. Apply \cite[Theorem 3.2 ]{Coulhon1997},
\[
\mu^{(2n)}(e)\ge\frac{1}{\exp\circ\psi(8n/c)}-e^{-2n}.
\]
Combine the upper and lower bounds, if $\Lambda_{2,G,\mu}\circ\exp(2s)\ge c\Lambda_{2,G,\mu}\circ\exp(s)$,
we have 
\begin{equation}
-\log\mu^{(2n)}(e)\simeq_{C}w(2n)\label{eq:double-equiv}
\end{equation}
with constant $C>0$ only depending on the doubling constant $c$.

\begin{theorem}\label{possible-return}

There exists universal constants $c,C>0$ such that the following
is true. Let $\gamma:[1,\infty)\to[1,\infty)$ be any function such
that $\frac{\gamma(n)}{n^{\frac{1}{3}}}$ and $\frac{n}{\gamma(n)}$
are non-decreasing. Then there is a group $\Delta$ such that 
\[
\forall t\geq1,\ c\gamma(t)\leq-\log\left(\mathfrak{q}^{(2t)}\left(e_{\Delta}\right)\right)\leq C\gamma(t).
\]

\end{theorem}

\begin{proof}

Given such a function $\gamma:[1,\infty)\to[1,\infty)$, it is strictly
increasing and continuous, define $\varrho:[1,\infty)\to[1,\infty)$
by 
\[
\varrho(x)=\frac{1}{x}\gamma^{-1}(x).
\]
From the assumption on $\gamma$ we have $\gamma(1)=1$ and $a^{\frac{1}{3}}\gamma(x)\le\gamma(ax)\le a\gamma(x)$
for any $a,x\ge1$. Thus 
\[
a\gamma^{-1}(x)\le\gamma^{-1}(ax)\le a^{3}\gamma^{-1}(x),
\]
and therefore 
\[
\varrho(x)\le\varrho(ax)\le a^{2}\varrho(x),
\]
which satisfies the assumption of Proposition \ref{possible-profile}
with $p=2$.

By Proposition \ref{possible-profile}, there exists universal constants
$c,C>0$ such that there is a group $\Delta$, for all $v\ge3$, 
\[
\frac{c}{\varrho\left(\log v\right)}\le\Lambda_{\Delta,\mathfrak{q}}(v)\le\frac{C}{\varrho(\log v)}.
\]
Note that since $\varrho(2x)\le4\varrho(x),$ it follows that for
all $s>0$, 
\[
\Lambda_{\Delta,\mathfrak{q}}\circ\exp(2s)\ge\frac{c}{4C}\Lambda_{\Delta,\mathfrak{q}}\circ\exp(s).
\]
In particular, the function $\Lambda_{\Delta,\mathfrak{q}}\circ\exp:(0,\infty)\to\mathbb{R}$
is doubling at infinity, then by \cite[ Lemma 2.5]{Bendikov2012},
the solution $w(t)$ to (\ref{eq:GC-exp}) satisfies 
\[
\Lambda_{\Delta,\mathfrak{q}}\circ\exp\circ w(t)\le\frac{w(t)}{t}\le D\Lambda_{\Delta,\mathfrak{q}}\circ\exp\circ w(t),
\]
where $D$ is a constant that only depends on the doubling constant
$c/4C$. Plug in the estimate of $\Lambda_{\Delta,\mathfrak{q}}$,
we have 
\[
\frac{c}{\varrho(w(t))}\le\frac{w(t)}{t}\le\frac{DC}{\varrho(w(t))}.
\]
By definition of $\varrho$, 
\[
\gamma(t)\varrho\left(\gamma(t)\right)=t,
\]
note that $x\varrho(x)$ is strictly increasing, therefore 
\[
c^{\frac{1}{3}}\gamma(t)\le\gamma(ct)\le w(t)\le\gamma(DCt)\le DC\gamma(t).
\]
Since the constants $c,C,D$ are universal, from (\ref{eq:double-equiv})
provided by the Coulhon-Grigor'yan theory, we conclude that 
\[
\forall t\geq1,\ c'\gamma(t)\leq-\log\left(\mathfrak{q}^{(2t)}\left(e_{\Delta}\right)\right)\leq C'\gamma(t)
\]
where $c',C'>0$ are universal constants.

\end{proof}

\begin{example} When $k_{s}=2^{\beta s}$ and $l_{s}=2^{\iota s}$
with $\beta>1,\ \iota>0$, the $L^{2}$-profile given in Example \ref{exponents-iso}
turns to return probability 
\[
-\log\left(\mathfrak{q}^{(2t)}\left(e_{\Delta}\right)\right)\simeq t^{\frac{\text{\ensuremath{\beta}}+\iota}{3\beta+\iota}},
\]
where the exponent can take any value in $(\frac{1}{3},1)$. \end{example}

\subsection{The case of dihedral groups}

\label{iso-dihedral}

In this subsection we estimate decay of return probability of simple
random walk on $\Delta$ where $\Gamma_{s}=D_{2l_{s}}$ are dihedral
groups. We show that in this case the return exponent of simple random
walks is $\frac{1}{3}$. Obtaining more precise estimates requires
further work.

\begin{proposition}\label{return-dihedral}

There exists an absolute constant $C>0$ such that the following holds.
Let $\Delta$ be the diagonal product constructed with $\Gamma_{s}=D_{2l_{s}}$
and parameters $\{k_{s}\}$ satisfying Assumption (\ref{k_growth}),
then 
\[
\frac{1}{C}n^{\frac{1}{3}}\le-\log\mathfrak{q}^{(2n)}\left(e_{\Delta}\right)\le Cn^{\frac{1}{3}}\log^{\frac{4}{3}}n.
\]

\end{proposition}

\begin{proof}

Since $\Delta$ projects onto $\Delta_{0}\simeq(\mathbb{Z}/2\mathbb{Z}\times\mathbb{Z}/2\mathbb{Z})\wr\mathbb{Z}$,
we have 
\[
\mathfrak{q}^{(2n)}\left(e_{\Delta}\right)\le\mathfrak{q}_{0}^{(2n)}\left(e_{\Delta_{0}}\right).
\]
The lower bound on $-\log\mathfrak{q}^{(2n)}\left(e_{\Delta}\right)$
follows from the decay of return probability on $\Delta_{0}$, see
\cite{Pittet2002}, 
\[
\mathfrak{q}_{0}^{(2n)}\left(e_{\Delta_{0}}\right)\le\exp\left(-\frac{1}{C}n^{\frac{1}{3}}\right).
\]

In the other direction we construct a test function on $\Delta$.
First take a test function $\psi_{r}$ on $D_{2l_{s}}$ 
\[
\psi_{r}(g)=\max\left\{ 1-\frac{\left|g\right|_{D_{2l_{s}}}}{r},0\right\} ,\ \mbox{for }1\le r\le l_{s}.
\]
Recall that the set $U_{r}^{\Delta}$ is defined as $U_{r}^{\Delta}=\{((f_{s}),z):\ Range(f_{s},z)\subset[-r,r]\}$.
Let $\mathcal{S}(r)=\left\{ s:\ k_{s}\le r,\ l_{s}\ge r^{2}\right\} $
and take 
\[
\Psi_{r}^{\Delta}\left((f_{s}),z\right)=\left(1-\frac{\left|z\right|}{r}\right)\mathbf{1}_{U_{r}^{\Delta}}\left(\left((f_{s}),z\right)\right)\prod_{s\in\mathcal{S}(r)}\prod_{x\in\left[-r+k_{s},r\right]}\psi_{r^{2}}(f_{s}(x)).
\]
Depending on the sequences $(k_{s}),(l_{s})$, the set $\mathcal{S}(r)$
might be empty, in which case we recover the test function of Proposition
\ref{iso-Delta}. (Recall the notation $\ell_{s}=\log\left|D_{2l_{s}}\right|=\log2l_{s}$.)
As in the proof of Proposition \ref{iso-Delta}, we have 
\[
\frac{\sum_{Z\in\Delta}\left(\Psi_{r}^{\Delta}(Z)-\Psi_{r}^{\Delta}(Z\tau)\right)^{2}}{\left\Vert \Psi_{r}^{\Delta}\right\Vert _{2}^{2}}\le\frac{C_{1}}{r^{2}}.
\]
By Cauchy-Schwarz inequality, we have 
\begin{align*}
\frac{\sum_{Z\in\Delta}\left(\Psi_{r}^{\Delta}(Z)-\Psi_{r}^{\Delta}(Z\alpha)\right)^{2}}{\left\Vert \Psi_{r}^{\Delta}\right\Vert _{2}^{2}} & \le\left|\mathcal{S}(r)\right|\sum_{s\in\mathcal{S}(r)}\frac{\sum_{g\in D_{2l_{s}}}\left(\psi_{r^{2}}(ga(s))-\psi_{r^{2}}(g)\right)^{2}}{\left\Vert \psi_{r^{2}}\right\Vert _{l^{2}\left(D_{2l_{s}}\right)}^{2}}\\
 & \le C_{1}\left|\mathcal{S}(r)\right|^{2}r^{-4}.
\end{align*}
The same estimates holds for $\beta$ with $a(s)$ replaced by $b(s)$.
Since $\left\{ k_{s}\right\} $ satisfies the growth assumption (\ref{k_growth}),
we have $\left|\mathcal{S}(r)\right|\le\log_{2}r$, therefore 
\[
\frac{\mathcal{E}_{\Delta,\mathfrak{q}}\left(\Psi_{r}^{\Delta}\right)}{\left\Vert \Psi_{r}^{\Delta}\right\Vert _{2}^{2}}\le C_{1}r^{-2}.
\]
The support of function $\Psi_{r}^{\Delta}$ is bounded by 
\begin{align*}
\left|\mbox{supp}\Psi_{r}^{\Delta}\right| & \le(2r+1)\left(\prod_{s:k_{s}\le r,l_{s}<r^{2}}\left(2l_{s}\right)^{4r}\right)\left(\prod_{s:k_{s}\le r,l_{s}>r^{2}}\left(2r^{2}\right)^{4r}\right)\\
 & \le(2r+1)\left(2r^{2}\right)^{4r\log_{2}r}.
\end{align*}
From these test functions we have 
\begin{equation}
\Lambda_{2,\Delta,\mathfrak{q}}(v)\le C'_{1}\frac{(\log\log v)^{4}}{\log^{2}v}.\label{eq: Lambda-dihedral}
\end{equation}
By the Coulhon-Grigor'yan theory, we conclude that there exists absolute
constant $C$, 
\[
\mathfrak{q}^{(2n)}\left(e_{\Delta}\right)\ge\exp\left(-Cn^{\frac{1}{3}}\log^{\frac{4}{3}}(2n)\right).
\]

\end{proof}

\begin{remark}

To get an estimate for $\Lambda_{p,\Delta,\mathfrak{q}}$, $p\in[1,2]$,
note that by projecting onto $\Delta_{0}$, we have 
\[
\Lambda_{1,\Delta,\mathfrak{q}}(v)\ge\Lambda_{1,\Delta,\mathfrak{q}_{0}}(v)\ge\frac{1}{C\log v},
\]
and in the proof above we have an upper bound (\ref{eq: Lambda-dihedral})
for $\Lambda_{2,\Delta,\mathfrak{q}}$. By Cheeger inequality (\ref{eq:cheeger-pq}),
we have 
\[
\frac{1}{C_{p}}\frac{1}{\log^{p}v}\le\Lambda_{p,\Delta,\mathfrak{q}}(v)\le C_{p}\frac{\left(\log\log v\right)^{2p}}{\log^{p}v}.
\]

\end{remark}

\section{Review of obstructions for embeddings}

\label{obstructions}

We first recall the standard definition of distortion of a map between
metric spaces. Given an injective map $f:X\to Y$ between two metric
spaces $\left(X,d_{X}\right)$ and $\left(Y,d_{Y}\right)$, the distortion
of $f$ measures quantitatively how far away $f$ is from being a
homothety, 
\[
\mbox{distortion}(f)=\left(\sup_{u,v\in X,u\neq v}\frac{d_{Y}(f(u),f(v))}{d_{X}(u,v)}\right)\left(\sup_{u,v\in X,u\neq v}\frac{d_{X}(u,v)}{d_{Y}(f(u),f(v))}\right).
\]
When $f$ is $C$-Lipschitz, the first sup is bounded by $C$, we
often focus on the second term. The expansion ratio of $f$ is defined
to be 
\[
\mbox{ratio}(f)=\inf_{u,v\in X,u\neq v}\frac{d_{Y}(f(u),f(v))}{d_{X}(u,v)}.
\]
The smallest distortion with which $X$ can be embedded in $Y$ is
denoted by $c_{Y}(X)$, 
\[
c_{Y}(X)=\inf\left\{ \mbox{distortion}(f):\ f:X\hookrightarrow Y\right\} .
\]

To connect with uniform embedding of an infinite group $G$, it is
well known that a sequence of finite metric spaces $\left(X_{k},d_{k}\right)$
embedded in the group $G$ can provide obstruction for good embedding
of the whole space. See e.g. Arzhantseva-Drutu-Sapir \cite{Arzhantseva2009}
and Austin \cite{Austin2011}. We quote a special case of the Austin
Lemma.

\begin{lemma}\label{(Lemma-3.1-Austin)}[The Austin Lemma \cite{Austin2011}]

Let $\mathfrak{X}$ be a metric space.
Let $\Gamma$ be a finitely generated infinite group equipped with
a finite generating set $S$, let $d$ denote the word distance on
the Cayley graph $(\Gamma,S)$. Suppose that we can find a sequence
of finite graphs $(X_{n},\sigma_{n})$ where $\sigma_{n}$ is a $1$-discrete
metric on $X_{n}$, and embeddings $\vartheta_{n}:X_{n}\hookrightarrow\Gamma$
such that there are constant $C,L\ge1$, $\delta>0$ that are independent
of $n$, 
\begin{itemize}
\item $diam(X_{n},\sigma_{n})\to\infty$ as $n\to\infty$; 
\item there exists a sequence of positive reals $(r_{n})_{n\ge1}$ such
that 
\[
\frac{1}{L}r_{n}\sigma_{n}(u,v)\le d\left(\vartheta_{n}(u),\vartheta_{n}(v)\right)\le Lr_{n}\sigma_{n}(u,v),\ \mbox{for all }u,v\in X_{n},n\ge1,
\]
moreover, $r_{n}\le Cdiam(X_{n},\sigma_{n})^{\beta}$ for all $n\ge1$; 
\item distortion of $(X_{n},\sigma_{n})$ into $\mathfrak{X}$ is large
in the sense that 
\[
c_{\mathfrak{X}}(X_{n},\sigma_{n})\ge\delta diam(X_{n},\sigma_{n})^{\eta}.
\]
Then 
\[
\alpha_{\mathfrak{X}}^{\ast}(\Gamma,d)\le1-\frac{\eta}{1+\beta}.
\]

\end{itemize}
\end{lemma}

The second assumption in Lemma \ref{(Lemma-3.1-Austin)} requires
that under the embedding $\vartheta_{k}$, the induced metric $d_{(G,S)}$
only dilates $d_{k}$ with uniformly bounded distortion. This point-wise
assumption is rather restrictive. In what follows we will present
some bounds that are more flexible.

The term \textquotedbl{}Poincaré inequalities\textquotedbl{} in the
context of metric embeddings was first systematically used in Linial-Magen-Naor
\cite{Linial2002}. It is a key ingredient for many existing lower
bounds for distortion of finite metric spaces. We review the basic
idea now. Let $\left(\mathcal{M},d_{\mathcal{M}}\right)$ be a finite
metric space, $\mathbf{a}=\left(a_{u,v}\right)$, $\mathbf{b}=(b_{u,v})$
where $u,v\in\mathcal{M}$ be two nonzero arrays of non-negative real
numbers. A $p$-Poincaré type inequality for $f:\mathcal{M}\to\mathfrak{X}$
is an inequality of the form 
\begin{equation}
\sum_{u,v\in\mathcal{M}}a_{u,v}d_{\mathfrak{X}}(f(u),f(v))^{p}\le C\sum_{u,v\in\mathcal{M}}b_{u,v}d_{\mathfrak{X}}(f(u),f(v))^{p}.\label{eq:a-b-poincare}
\end{equation}
The infimum of the constant $C$ such that the inequality holds for
all nontrivial $f:\mathcal{M}\to\mathfrak{X}$ is known as the $\mathfrak{X}$-valued
Poincaré constant associated with $\mathbf{a},\mathbf{b}$, 
\[
P_{\mathbf{a},\mathbf{b},p}(\mathcal{M},\mathfrak{X})=\sup\frac{\sum_{u,v}a_{u,v}d_{\mathfrak{X}}(f(u),f(v))^{p}}{\sum_{u,v}b_{u,v}d_{\mathfrak{X}}(f(u),f(v))^{p}},
\]
where the sup is taken over all $f:\mathcal{M}\to\mathfrak{X}$ such
that $\sum_{u,v}a_{u,v}d_{\mathfrak{X}}(f(u),f(v))^{p}\neq0.$ It
follows from definition of the Poincaré constant that 
\[
\left(\inf_{u,v\in\mathcal{M},u\neq v}\frac{d_{\mathfrak{X}}(f(u),f(v))^{p}}{d_{\mathcal{M}}(u,v)^{p}}\right)\left(\sum_{u,v\in\mathcal{M}}a_{u,v}d_{\mathcal{M}}\left(u,v\right)^{p}\right)\le P_{\mathbf{a},\mathbf{b},p}(\mathcal{M},\mathfrak{X})\left(\sum_{u,v}b_{u,v}d_{\mathfrak{X}}(f(u),f(v))^{p}\right),
\]
that is the expansion ratio of $f$ satisfies 
\[
\inf_{u,v\in\mathcal{M},u\neq v}\frac{d_{\mathfrak{X}}(f(u),f(v))^{p}}{d_{\mathcal{M}}(u,v)^{p}}\le P_{\mathbf{a},\mathbf{b},p}(\mathcal{M},\mathfrak{X})\left(\frac{\sum_{u,v}b_{u,v}d_{\mathfrak{X}}(f(u),f(v))^{p}}{\sum_{u,v}a_{u,v}d_{\mathcal{M}}\left(u,v\right)^{p}}\right).
\]

To relate to compression function, we need an extra ingredient that
resembles a mass distribution assumption. We say that the array $\mathbf{a}$
satisfies $\left(p;l,c\right)$-mass distribution condition if 
\begin{equation}
\frac{\sum_{d_{\mathcal{M}}(u,v)\ge l}a_{u,v}d_{\mathcal{M}}(u,v)^{p}}{\sum_{d_{\mathcal{M}}(u,v)}a_{u,v}d_{\mathcal{M}}(u,v)^{p}}\ge c,\label{eq:(p,c)-mass}
\end{equation}
in words, $c$-fraction of the total $\mathbf{a}$ array sum is from
vertices at least $l$ apart. Under this additional assumption, for
any $f:\mathcal{M}\to\mathfrak{X}$, there exists $u,v\in\mathcal{M}$
with $d_{\mathcal{M}}(u,v)\ge l$ such that 
\begin{equation}
\rho_{f}(l)\le d_{\mathfrak{X}}\left(f(u),f(v)\right)\le\mbox{diam}(\mathcal{M})P_{\mathbf{a},\mathbf{b},p}(\mathcal{M},\mathfrak{X})^{\frac{1}{p}}\left(\frac{\sum_{u,v}b_{u,v}d_{\mathfrak{X}}(f(u),f(v))^{p}}{c\sum_{u,v}a_{u,v}d_{\mathcal{M}}\left(u,v\right)^{p}}\right)^{\frac{1}{p}}.\label{eq:Poincare-compression}
\end{equation}

\medskip{}

This compression upper bound is very useful. In practice, to apply
this we need to choose the arrays $\mathbf{a}$, $\mathbf{b}$ and
obtain a good Poincaré inequality of the form (\ref{eq:a-b-poincare}).
This is not an easy task in general. In what follows we review some
special cases. These settings have been investigated extensively in
literature, thus established results are available for application
to metric embeddings.

\subsection{Poincaré inequalities in the classical form\label{sub:Spectral-method}}

Pioneered by work of Enflo \cite{Enflo1969}, it is well known that
spectral gap of certain Markov operators on a finite metric space
$(X,d)$ can be used to show lower bound for distortion of embedding
of $(X,d)$ into Hilbert spaces. This method appeared in Linial-Magen
\cite{Linial2000}, Newman-Rabinovich \cite{Newman2003} and was extended
in Grigorchuk-Nowak \cite{Grigorchuk2012}, Jolissaint-Valette \cite{Jolissaint2014} and Mimura \cite{Mimura2014}.
Interested readers may also consult Chapter 13.5 in the book \cite{LyonsPeresBook}
for a nice introduction to this topic.

Let $\left(\mathcal{M},d_{\mathcal{M}}\right)$ be a finite metric
space, $K:\mathcal{M}\times\mathcal{M}\to[0,1]$ a Markov transition
kernel kernel on $\mathcal{M}$. Suppose $K$ is reversible with respect
to stationary distribution $\pi$. The most familiar Poincaré inequality
for such a finite Markov chain takes the following form: for $f:\mathcal{M}\to\mathbb{R}$,
\[
\sum_{u,v}\left|f(u)-f(v)\right|^{2}\pi(u)\pi(v)\le C\sum_{u,v}\left|f(u)-f(v)\right|^{2}K(u,v)\pi(u).
\]
The reciprocal of the Poincaré constant is known as the spectral gap,
\begin{equation}
\lambda(K)=\inf_{f:\mathcal{M}\to\mathbb{R},f\neq c}\left\{ \frac{\sum_{u,v\in\mathcal{M}}\left|f(u)-f(v)\right|^{2}K(u,v)\pi(v)}{\sum_{u,v\in\mathcal{M}}\left|f(u)-f(v)\right|^{2}\pi(u)\pi(v)}\right\} .\label{eq:spectral-gap}
\end{equation}
In this case the Poincaré constant is often referred to as the relaxation
time of $K$. Mixing times of finite Markov chains have been a very
active research area in the past decades. For a great variety of Markov
chains good estimates of their spectral gaps are known, examples can
be found in \cite{Saloff-Coste1997}, \cite{Levin2009}. Note that
the same Poincaré inequality holds for Hilbert space valued functions
$f:\mathcal{M}\to\mathcal{H}$, this fact can be checked by eigenbasis
expansion. In some examples, based on the $\ell^{2}$-Poincaré inequality,
one can apply Matoušek extrapolation (see \cite{Matousek} and the
version in \cite{Naor2011a}) to obtain useful Poincaré inequalities
for $\ell^{p}$-valued functions.

In the setting of inequality (\ref{eq:a-b-poincare}), having variance
of $f$ on the left side of the inequality and Dirichlet form on the
right side corresponds to taking 
\begin{equation}
a_{u,v}=\pi(u)\pi(v)\ \mbox{and }b_{u,v}=\pi(u)K(u,v).\label{eq:b-K}
\end{equation}
Define $\lambda_{p}(\mathcal{M},K,\mathfrak{X})$ of the Markov operator
$K$ on $Y$ to be 
\begin{equation}
\lambda_{p}(\mathcal{M},K,\mathfrak{X})=\frac{1}{P_{\mathbf{a},\mathbf{b},p}(\mathcal{M},\mathfrak{X})},\label{eq:spectralgap}
\end{equation}
where $\mathbf{a},\mathbf{b}$ are specified by (\ref{eq:b-K}). When
$\mathfrak{X}$ is a Hilbert space and $p=2$, this definition agrees
with the standard variational formula of the spectral gap.

We now formulate an analogue of Lemma \ref{(Lemma-3.1-Austin)}. Since
the bound relies crucially on the $\mathfrak{X}$-valued Poincaré
constants $1/\lambda_{p}(X_{n},K_{n},\mathfrak{X})$ of the Markov
operator $K_{n}$ on $X_{n}$, we refer to it as the spectral method
for bounding compression functions.

\begin{lemma}\label{spectral-compression}

Let $G$ be an infinite group equipped with a metric $d$ and $p\in[1,\infty)$.
Let $X_{n}$ be a sequence of finite subsets in $G$ and $K_{n}$
be reversible Markov kernels on $X_{n}$ with stationary distribution
$\pi_{n}$. Suppose there exists a constant $c\in(0,1)$ and an increasing
sequence $\left\{ l_{n}\right\} $such that the array $\mathbf{a}_{n}$
defined as $\mathbf{a}_{n}(u,v)=\pi(u)\pi(v)$ satisfies $\left(p;l_{n},c\right)$-mass
distribution condition (\ref{eq:(p,c)-mass}).

Let $f:G\to\mathfrak{X}$ be a 1-Lipschitz uniform embedding. Then
the compression function of $f$ satisfies 
\[
\rho_{f}\left(l_{n}\right)\le\mbox{diam}_{d}\left(X_{n}\right)\left(\frac{1}{\lambda_{p}(X_{n},K_{n},\mathfrak{X})}\left(\frac{\sum_{u,v\in X_{n}}d_{\mathfrak{X}}(f(u),f(v))^{p}K_{n}(u,v)\pi_{n}(v)}{c\sum_{u,v\in X_{n}}d(u,v)^{p}\pi(u)\pi(v)}\right)\right)^{\frac{1}{p}}.
\]

\end{lemma}

\begin{proof}

Equip $X_{n}$ with the metric induced by the metric $d$ on $G$,
the inequality follows from (\ref{eq:Poincare-compression}).

\end{proof}

\begin{example}\label{group-poincare}

Consider the special case where $X_{n}$ is a sequence of finite subgroups
in $G$ and $d$ is a left invariant metric on $G$, e.g. the word
metric. Take $\mu_{n}$ to be a symmetric probability measure on $X_{n}$
and $K_{n}(u,v)=\mu_{n}\left(u^{-1}v\right)$. Then the Markov chain
with transition kernel $K_{n}$ is the random walk on $X_{n}$ with
step distribution $\mu_{n}$. It is reversible with respect to the
uniform distribution $U_{n}$ on $X_{n}$. In this case, because of
transitivity, the mass distribution condition is easily satisfied,
namely 
\[
\sum_{v:\ d(u,v)\ge\frac{1}{2}diam_{d}\left(X_{n}\right)}U_{n}(v)\ge\frac{1}{2}\ \mbox{for every }u\in X_{n}.
\]
It follows that $\mathbf{a}_{n}=\left(U_{n}(u)U_{n}(v)\right)$ satisfies
$\left(p,\frac{1}{2}\mbox{diam}_{d}\left(X_{n}\right),\frac{1}{2}\right)$-mass
distribution condition, and the bound in Lemma \ref{spectral-compression}
simplifies to 
\[
\rho_{f}\left(\frac{\mbox{diam}_{d}(X_{n})}{2}\right)\le\left(\frac{2^{p+2}\sum_{u,v\in X_{n}}d_{\mathfrak{X}}(f(u),f(v))^{p}K_{n}(u,v)\pi_{n}(v)}{\lambda_{p}(X_{n},K_{n},\mathfrak{X})}\right)^{\frac{1}{p}}.
\]

\end{example}

\subsection{Markov type inequalities \label{sub:Markov-type-method}}

The notion of the Markov type of a metric space was introduced by
K. Ball in \cite{Ball}. It has found important applications in metric
geometry. In \cite{Linial2002}, Linal, Magen and Naor pointed out
that the basic assumption of this concept can be viewed as Poincaré
inequalities. The Markov type method for bounding compression exponent
was first introduced by Naor and Peres in \cite{Naor2008} and later
significantly extended in \cite{Naor2011}.

\begin{definition}\label{def-Ball}[K. Ball \cite{Ball}] Given a metric space
$\left(\mathfrak{X},d_{\mathfrak{X}}\right)$ and $p\in[1,\infty)$,
we say that $\mathfrak{X}$ has Markov-type $p$ if there exists a
constant $C>0$ such that for every stationary reversible Markov chain
$\left\{ Z_{t}\right\} _{t=0}^{\infty}$ on $\left\{ 1,\ldots,n\right\} $,
every mapping $f:\{1,\ldots,n\}\to\mathfrak{X}$ and every time $t\in\mathbb{N}$,
\begin{equation}
\mathbf{E}d_{\mathfrak{X}}(f(Z_{t}),f(Z_{0}))^{p}\le C^{p}t\mathbf{E}d_{\mathfrak{X}}(f(Z_{1}),f(Z_{0}))^{p}.\label{eq:Markov-type}
\end{equation}
The least such constant $C$ is called the Markov-type $p$ constant
of $\mathfrak{X}$ and is denoted by $M_{p}(\mathfrak{X})$.

\end{definition}

Theorem 2.3 in Naor-Peres-Sheffield-Schramm \cite{Naor2006} implies
the following results for the classical Lebesgue spaces $L_{p}$.
For $p\in(1,2]$, the space $L_{p}$ has Markov type $p$ and $M_{p}\left(L_{p}\right)\le\frac{8}{\left(2^{p+1}-4\right){}^{1/p}}$;
and for every $p\in[2,\infty)$, $L_{p}$ has Markov type $2$ and
$M_{2}\left(L_{p}\right)\le4\left(p-1\right)^{\frac{1}{2}}$ . See
\cite{Naor2006} for more examples of metric spaces of known Markov
type.

In the setting of (\ref{eq:a-b-poincare}), the inequality (\ref{eq:Markov-type})
in the definition of Markov type $p$ can be viewed as a Poincaré
inequality with 
\[
a_{u,v}=K^{t}(u,v)\pi(u)\mbox{ and }b_{u,v}=K(u,v)\pi(u),
\]
where $K$ is the transition kernel of a reversible Markov chain on
state space $\mathcal{M}$ of $n$ points, and $\pi$ is its stationary
distribution. The Poincaré inequality provided by (\ref{eq:Markov-type})
reads 
\[
\sum_{u,v\in\mathcal{M}}d_{\mathfrak{X}}(f(u),f(v))^{p}K^{t}(u,v)\pi(u)\le M_{p}^{p}\left(\mathfrak{X}\right)t\sum_{u,v\in\mathcal{M}}d_{\mathfrak{X}}(f(u),f(v))K(u,v)\pi(u)
\]
for all functions $f:\mathcal{M}\to\mathfrak{X}$. Note that the notion
of Markov type is very powerful, if $\mathfrak{X}$ has Markov type
$p$, then the inequality above is valid for any finite state space
$\mathcal{M}$ and any reversible Markov transition kernel $K$ on
$\mathcal{M}$.

Now we examine the mass distribution condition. Let $\left(\mathcal{M},d_{\mathcal{M}}\right)$
be a finite metric space, and $K$ be a reversible Markov kernel on
$\mathcal{M}$ with stationary distribution $\pi$. Let $\left\{ Z_{t}\right\} _{t=0}^{\infty}$
be a stationary Markov chain on $\mathcal{M}$ with transition kernel
$K$. At time $t$, set 
\[
\gamma(t)^{p}=\frac{1}{2}\mathbf{E}_{\pi}\left[d_{\mathcal{M}}(Z_{t},Z_{0})^{p}\right],
\]
then 
\begin{align*}
\mathbf{E}_{\pi}\left[d_{\mathcal{M}}(Z_{t},Z_{0})^{p}\mathbf{1}_{\left\{ d_{\mathcal{M}}(Z_{t},Z_{0})>\gamma(t)\right\} }\right] & =\mathbf{E}_{\pi}\left[d_{\mathcal{M}}(Z_{t},Z_{0})^{p}\right]-\mathbf{E}_{\pi}\left[d_{\mathcal{M}}(Z_{t},Z_{0})^{p}\mathbf{1}_{\left\{ d_{Y}(Z_{t},Z_{0})\le\gamma(t)\right\} }\right]\\
 & \ge\mathbf{E}_{\pi}\left[d_{\mathcal{M}}(Z_{t},Z_{0})^{p}\right]-\gamma(t)^{p}=\frac{1}{2}\mathbf{E}_{\pi}\left[d_{\mathcal{M}}(Z_{t},Z_{0})^{p}\right].
\end{align*}
That is, the array $\mathbf{a}$ satisfies $\left(p;\left(\frac{1}{2}\mathbf{E}_{\pi}\left[d_{\mathcal{M}}(Z_{t},Z_{0})^{p}\right]\right)^{\frac{1}{p}},\frac{1}{2}\right)$-mass
distribution condition where $\mathbf{a}$ is defined by $a_{u,v}=K^{t}(u,v)\pi(u)$.
From the inequality (\ref{eq:Poincare-compression}) we derive the
following upper bound on compression function.

\begin{lemma}\label{Markov-type-compression}

Let $G$ be an infinite group equipped with a metric $d$. Let $f:G\to\mathfrak{X}$
be a 1-Lipschitz uniform embedding. Assume that $\mathfrak{X}$ has
Markov type $p$.

Let $X_{n}$ be a sequence of finite sets of $G$, $K_{n}$ be a reversible
Markov kernel on $X_{n}$ with stationary distribution $\pi_{n}$.
Let $\left\{ Z_{t}^{(n)}\right\} _{t=0}^{\infty}$ be a stationary
Markov chain on $X_{n}$ with transition kernel $K_{n}$. Then for
any $t_{n}\in\mathbb{N}$, the compression function of $f$ satisfies
\begin{multline*}
\rho_{f}\left(\left(\frac{1}{2}\mathbf{E}_{\pi_{n}}\left[d\left(Z_{t_{n}}^{(n)},Z_{0}^{(n)}\right)^{p}\right]\right)^{\frac{1}{p}}\right)\\
\le\left(2M_{p}^{p}(\mathfrak{X})t_{n}\mbox{diam}_{(G,d)}\left(X_{n}\right)^{p}\frac{\mathbf{E}_{\pi_{n}}\left[d_{\mathfrak{X}}\left(f\left(Z_{1}^{(n)}\right),f\left(Z_{0}^{(n)}\right)\right)^{p}\right]}{\mathbf{E}_{\pi_{n}}\left[d\left(Z_{t_{n}}^{(n)},Z_{0}^{(n)}\right)^{p}\right]}\right)^{\frac{1}{p}}.
\end{multline*}

\end{lemma}

\begin{remark}

This upper bound on the compression function is in the same spirit
as the argument of Naor and Peres in Section 5 of \cite{Naor2011}.
The difference is that in \cite{Naor2011} the authors considered
random walks on the infinite group $G$ starting at identity and $f$
is taken to be a $1$-cocycle on $G$, then Markov type inequality
for $1$-cocycles was applied to bound the compression function. One
restriction for such an approach is that the step distribution of
the random walk needs to have finite $p$-moment. While in the finite
subsets, in principle one can experiment with any reversible transition
kernel and choose the best one available. Examples that illustrate
this point can be found in Subsection \ref{sub:-compression-wr-Z}.

\end{remark}

\subsection{Comparing spectral and Markov type methods}

It is interesting to compare the classical Poincaré inequalities and
the ones from Markov type method. Suppose in the infinite group $G$,
we have chosen a sequence of subsets $\left\{ X_{n}\right\} $ and
reversible Markov kernels $K_{n}$ on $X_{n}$. With this sequence
$\left\{ \left(X_{n},K_{n}\right)\right\} $ we compare the results
given by the two methods. Let $\mathfrak{X}$ be a metric space of
Markov type $p$ and $f:G\to\mathfrak{X}$ be a 1-Lipschitz embedding
from $(G,d)$ to $(\mathfrak{X},d_{\mathfrak{X}})$. To compare terms
in the bounds of Lemma \ref{spectral-compression} and \ref{Markov-type-compression},
first note that 
\[
\mathbf{E}d_{\mathfrak{X}}\left(f\left(Z_{1}^{(n)}\right),f\left(Z_{0}^{(n)}\right)\right)^{p}=\sum_{u,v\in X_{n}}d_{\mathfrak{X}}(f(u),f(v))^{p}K_{n}(u,v)\pi_{n}(u).
\]
Now we choose $t_{n}$ to be the comparable to the Poincaré constant
$P_{p}(X_{n},K_{n},\mathfrak{X})$ (it corresponds to relaxation time
when $\mathfrak{X}$ is a Hilbert space and $p=2$). Suppose in addition
that $\pi_{n}$ satisfies the $(p;\theta\mbox{diam}(X_{n}),c)$-mass
distribution condition, then essentially the difference in the two
bounds comes from the ratio 
\[
\frac{\mbox{diam}_{(G,d)}\left(X_{n}\right)^{p}}{\mathbf{E}d\left(Z_{t_{n}}^{(n)},Z_{0}^{(n)}\right)^{p}}.
\]
Thus if there is a constant $c_{1}>0$ such that for $t_{n}\simeq P_{p}\left(X_{n},K_{n},\mathfrak{X}\right)$,
\[
\mathbf{E}d_{X_{n}}\left(Z_{t_{n}}^{(n)},Z_{0}^{(n)}\right)^{p}\ge c_{1}^{p}\mbox{diam}_{(G,d)}(X_{n})^{p},
\]
then up to some multiplicative constants, the two methods give the
same compression upper bound.

It is important in applications that the choice of the sequence of
finite subsets $X_{n}$ and Markov kernels $K_{n}$ is flexible. For
example, in order to use Poincaré inequalities to obtain an upper
bound on the compression function of uniform embedding $f$ from $G$
into a Hilbert space, the subsets $X_{n}$ should be chosen to capture
some worst distorted elements in the group under $f$, and the Markov
kernel $K_{n}$ on $X_{n}$ should be chosen so that 
\[
\frac{1}{\lambda(K_{n})}\left(\sum_{u,v\in X_{n}}d_{\mathcal{H}}(f(u),f(v))^{2}K_{n}(u,v)\pi_{n}(v)\right)
\]
is as small as possible. That is, $K_{n}$ needs to achieve a balance
between spectral gap and Dirichlet form $\mathcal{E}_{K_{n}}(f)$.
This point will be the guideline for the choice of $\left(X_{n},K_{n}\right)$
in the examples we treat.

\subsection{Metric cotype inequalities\label{sub:metric-cotype-inequalities}}

The notion of type and cotype plays a central role in the local theory
of Banach spaces. The classical linear notion of type and cotype is
defined as follows. A Banach space $\mathfrak{X}$ is said to have
(Rademacher) type $p>0$ if there exists a constant $T>0$ such that
for every $n$ and every $x_{1},\ldots,x_{n}\in\mathfrak{X}$, 
\[
\mathbb{E}\left\Vert \sum_{j=1}^{n}\varepsilon_{j}x_{j}\right\Vert _{\mathfrak{X}}^{p}\le T^{p}\sum_{j=1}^{n}\left\Vert x_{j}\right\Vert _{\mathfrak{X}}^{p},
\]
where $\mathbb{E}$ is the expectation with respect to uniform distribution
on $\left(\varepsilon_{1},\ldots,\varepsilon_{n}\right)\in\left\{ -1,1\right\} ^{n}$.
A Banach space $\mathfrak{X}$ is said to have (Rademacher) cotype
$q>0$ if there exists a constant $C>0$ such that for every $n$
and every $x_{1},\ldots,x_{n}\in\mathfrak{X}$, 
\[
\mathbb{E}\left\Vert \sum_{j=1}^{n}\varepsilon_{j}x_{j}\right\Vert _{\mathfrak{X}}^{q}\ge\frac{1}{C^{q}}\sum_{j=1}^{n}\left\Vert x_{j}\right\Vert _{\mathfrak{X}}^{q}.
\]
Given a Banach space $\mathfrak{X}$, define 
\[
p_{\mathfrak{X}}=\sup\left\{ p:\ \mathfrak{X}\mbox{ has type }p\right\} ,\ q_{\mathfrak{X}}=\inf\left\{ q:\ \mathfrak{X}\mbox{ has cotype }q\right\} .
\]
The space $\mathfrak{X}$ is said to be of nontrivial type if $p_{\mathfrak{X}}>1$,
and it is of nontrivial cotype if $q_{\mathfrak{X}}<\infty$.

Mendel and Naor \cite{Mendel2008} introduced the nonlinear notion
of metric cotype. By \cite[Definition 1.1]{Mendel2008}, $\left(\mathfrak{X},d_{\mathfrak{X}}\right)$
has metric cotype $q$ with constant $\Gamma$ if for every integer
$n\in\mathbb{N}$, there exists an even integer $m$, such that for
every $f:\mathbb{Z}_{m}^{n}\to\mathfrak{X}$, 
\[
\sum_{j=1}^{n}\sum_{u\in\mathbb{Z}_{m}^{n}}d_{\mathfrak{X}}\left(f\left(u+\frac{m}{2}\mathbf{e}_{j}\right),f(u)\right)^{q}\pi(u)\le\Gamma^{q}m^{q}\sum_{u\in\mathbb{Z}_{m}^{n}}\mathbb{E}\left[d_{\mathfrak{X}}\left(f\left(u+\boldsymbol{\varepsilon}\right),f(u)\right)^{q}\right]\pi(u),
\]
where $\pi$ is the uniform distribution on $\mathbb{Z}_{m}^{n}$
and $\mathbb{E}$ is the expectation taken with respect with uniform
distribution on $\boldsymbol{\varepsilon}=\left(\varepsilon_{1},\ldots,\varepsilon_{n}\right)\in\left\{ -1,0,1\right\} ^{n}$,
and $\left\{ \mathbf{e}_{j}\right\} $ is the standard basis of $\mathbb{R}^{n}$.
Mendel and Naor proved in \cite{Mendel2008} that for a Banach space
$\mathfrak{X}$ and $q\in[2,\infty)$, $\mathfrak{X}$ has metric
cotype $q$ if and only if it has Rademacher cotype $q$. As a key
step, they established the following sharp estimate, which we will
refer to as the metric cotype inequality.

\begin{theorem}\label{metric-cotype}[Theorem 4.2 \cite{Mendel2008}]
Let $\pi$ be the uniform distribution on $\mathbb{Z}_{m}^{n}$ and
$\sigma$ be the uniform distribution on $\left\{ -1,0,1\right\} ^{n}$.
Let $\mathfrak{X}$ be a Banach space of Rademacher type $p>1$ and
cotype $q\in[2,\infty)$. Then for every $f:\mathbb{Z}_{m}^{n}\to\mathfrak{X}$,
\begin{multline*}
\sum_{u\in\mathbb{Z}_{m}^{n}}\sum_{j=1}^{n}d_{\mathfrak{X}}\left(f\left(u+\frac{m}{2}\mathbf{e}_{j}\right),f(u)\right)^{q}\pi(u)\\
\le\left(5\max\left\{ C(\mathfrak{X})m,n^{\frac{1}{q}}\right\} \right)^{q}\sum_{u\in\mathbb{Z}_{m}^{n}}\sum_{\boldsymbol{\varepsilon}\in\{-1,0,1\}^{n}}d_{\mathfrak{X}}\left(f(u+\boldsymbol{\varepsilon}),f(u)\right)^{q}\sigma(\boldsymbol{\varepsilon})\pi(u),
\end{multline*}
where $C(\mathfrak{X})>0$ is a constant that only depends on the
cotype constant and $K_{q}$-convexity constant of $\mathfrak{X}$.

\end{theorem}

This metric cotype inequality can be viewed as a Poincaré inequality
with rather unusual choice of arrays $\mathbf{a},\mathbf{b}$ on $\mathbb{Z}_{m}^{n}$,
namely 
\[
a_{u,v}=\sum_{j=1}^{n}\pi(u)\mathbf{1}_{\left\{ v=u+\frac{m}{2}\mathbf{e}_{j}\right\} }\mbox{ and }b_{u,v}=\sum_{\boldsymbol{\varepsilon}\in\{-1,0,1\}^{n}}\pi(u)\mathbf{1}_{\{v=u+\boldsymbol{\varepsilon}\}}\sigma(\boldsymbol{\varepsilon}),
\]
then the Poincaré constant is bounded by 
\[
P_{\mathbf{a},\mathbf{b},2}\left(\mathbb{Z}_{m}^{n},\mathfrak{X}\right)\le\left(5\max\left\{ C(\mathfrak{X})mn^{\frac{1}{2}-\frac{1}{q}},n^{\frac{1}{2}}\right\} \right)^{2}.
\]
It captures a subtle comparison between a transition kernel that moves
far in one fiber and another kernel that moves by $\pm1$ across the
whole product.

Among many applications of metric cotype, such inequalities provides
sharp lower bound for the distortion of embeddings of the $\ell^{\infty}$
integer lattice $\left[m\right]_{\infty}^{n}$ into Banach spaces
of nontrivial type and cotype $q$, see Theorem 1.12 in \cite{Mendel2008}.
In Section \ref{sec:dihedral-compression} we will apply these metric
cotype inequalities in the study of compression of diagonal product
$\Delta$ constructed with dihedral groups, exactly because of the
presence of $l^{\infty}$-lattices of growing side length in the group.

\section{Compression of $\Delta$ with embedded expanders\label{sec:Compression-expander}}

In this section we consider compression of the diagonal product $\Delta$
constructed with $\left\{ \Gamma_{s}\right\} $ chosen to be certain
families of expanders. Let $\mathfrak{X}$ be a Banach space, a map
$\Psi:G\to\mathfrak{X}$ is called $G$-equivariant if there exists
an action $\tau$ of $G$ on $\mathfrak{X}$ by affine isometries
and a vector $v\in\mathfrak{X}$ such that $\Psi(g)=\tau(g)v$ for
all $x\in G$. Such a map is called a $1$-cocycle, see \cite{Cornulier2007}.

A couple of functions $(g_{1},g_{2})$ is an \textbf{equivariant}-$\mathfrak{X}$\textendash{}\textbf{compression
gap} of $G$ if any $1$-Lipschitz $G$-equivariant embedding
$\varphi:G\to\mathfrak{X}$ satisfies $\rho_{\varphi}(t)\le g_{2}(t)$
for all $t\ge1$ and there exists a $1$-Lipschitz $G$-equivariant
embedding $\Psi:G\to\mathfrak{X}$ such that $\rho_{\Psi}(t)\ge g_{1}(t)$
for all $t\ge1$.

We address the question regarding possible $L_{p}$-compression exponents
of finitely generated amenable groups.

\begin{proposition}

For any $\gamma\in[0,1]$, there exists a finitely generated elementary
amenable group $\Delta$ such that for all $p\ge1$, 
\[
\alpha_{p}^{\#}(\Delta)=\gamma.
\]

\end{proposition}

This result follows from a more precise result about equivariant compression
gap of the diagonal product group $\Delta$, see Theorem \ref{compression-expander}.
We will see that when the lamp groups $\left\{ \Gamma_{s}\right\} $
are chosen to be expanders, single copies of these lamp groups provide
sufficient obstruction for embedding. In some sense this case can
be viewed as an amenable analogue of \cite{Arzhantseva2009}.

\subsection{An upper bound in any uniformly convex Banach spaces}

In this subsection we take $\left\{ \Gamma_{s}\right\} $ as a subsequence
in the Lafforgue super expanders $\left\{ \Gamma_{m}\right\} $ described
in Example \ref{examples}. By Fact \ref{embed-one}, each group $\Gamma_{s}$
embeds homotheticaly in the diagonal product $\Delta$ with ratio
$k_{s}+1$, i.e. there is group homomorphism $\vartheta_{s}:\Gamma_{s}\rightarrow\Delta$
satisfying 
\[
\left|\vartheta_{s}(\gamma)\right|_{\Delta}=(k_{s}+1)\left|\gamma\right|_{\Gamma_{s}}.
\]
From these distortion estimates and the embeddings $\vartheta_{s}:\Gamma_{s}\hookrightarrow\Delta$,
we immediately derive an upper bound on compression function of $\Delta$
into $\mathfrak{X}$ by Lemma \ref{spectral-compression}.

\begin{lemma}\label{compression-expander-upper}

Let $\Delta$ be the diagonal product with parameters $(k_{s})$
and lamp groups $\{\Gamma_{s}\}$ chosen as a subsequence of Lafforgue
super expanders in Example \ref{examples}, $\mbox{diam}(\Gamma_{s})=l_{s}$.
Then for any uniformly convex Banach space $\mathfrak{X}$, there
exists a constant $\delta=\delta\left(\Gamma,\mathfrak{X},|A|+|B|\right)>0$
such that the compression function of any 1-Lipschitz embedding $\Psi:\Delta\to\mathfrak{X}$ satisfies
\[
\rho_{\Psi}\left(\frac{1}{2}(k_{s}+1)l_{s}\right)\le4\delta^{-\frac{1}{2}}\left(k_{s}+1\right),
\]

\end{lemma}

\begin{proof} Take $X_{s}=\vartheta_{s}\left(\Gamma_{s}\right)$
and $K_{s}(u,v)=\nu_{s}\left(\vartheta_{s}^{-1}(u^{-1}v)\right)$
where $\nu_{s}$ is uniform on the generating set $A(s)\cup B(s)$.
To apply Lemma \ref{spectral-compression}, note that $\mbox{diam}_{d_{\Delta}}(X_{s})=(k_{s}+1)l_{s}$,
the Poincaré constant $P_{2}\left(X_{s},K_{s},\mathfrak{X}\right)\le1/\delta$
by Lafforgue's result (\ref{eq:lafforgue-poincare}), where $\delta$
is a constant only depending on $\Gamma,\mathfrak{X}$ and $|A|+|B|$.
Since $\Psi$ is $1$-Lipschitz with respect to $\left|\cdot\right|_{\Delta}$,
\[
\sum_{u,v\in X_{s}}d_{\mathfrak{X}}(\Psi(u),\Psi(v))^{2}K_{s}(u,v)\pi_{s}(u)\le(k_{s}+1)^{2}.
\]
Since $\pi_{s}$ is the uniform distribution on the subgroup $X_{s}$,
the upper bound on $\rho_{\Psi}$ then follows from the Poincaré inequalities
(\ref{eq:lafforgue-poincare}) in Example \ref{group-poincare} with
$p=2$. \end{proof}

\subsection{Compression gap of embedding of $\Delta$ into $L_{p}$}

In this subsection we focus on the case with $L_{p}$, $p\ge1$, as
target spaces for embedding.

\subsubsection{Upper bound when $\{\Gamma_{s}\}$ are expanders}

When the target space is $L_{p}$, $p\ge1$, a more precise piecewise
upper bound of the compression gap can be obtained. Recall that a
symmetric probability measure $\mu$ on a group $G$ defines a Markov
transition kernel $K(u,v)=\mu(u^{-1}v)$ which is reversible with
respect to the uniform distribution on $G$. Its $\ell^{2}$-spectral
gap $\lambda\left(G,\mu\right)=\lambda\left(G,K\right)$ is defined
as in (\ref{eq:spectral-gap}).

\begin{proposition}\label{expander-upper-Lp}

Let $\Delta$ be the diagonal product with parameters $(k_{s})$
and lamp groups $\{\Gamma_{s}\}$ expanders where $\mbox{diam}(\Gamma_{s})=l_{s}<\infty$.
Suppose $\left\{ \Gamma_{s}\right\} $ satisfies Assumption \ref{Gamma_Involution}
and 
\[
\lambda\left(\Gamma_{s},\nu_{s}\right)\ge\delta>0\ \mbox{for all }s\mbox{ with }l_{s}<\infty,
\]
where $\nu_{s}$ is uniform on $A(s)\cup B(s)$. Then there exists
a constant $C_{0}$ depending only on $\left|A\right|,|B|$ such that
for any 1-Lipschitz embedding $\Psi:\Delta\to L_{p}$, the
compression function of $\Psi$ satisfies for all $s\ge1$ with $k_{s},l_{s}<\infty$
\begin{align}
\rho_{\Psi}\left(\frac{1}{2}x\right) & \le C(\delta,p)\frac{x}{l_{s}},\ \mbox{if }x\in\left[k_{s}l_{s},k_{s+1}l_{s}\right],\nonumber \\
\mbox{where } & C(\delta,p)=\begin{cases}
C_{0}\delta{}^{-\frac{1}{p}} & \mbox{ if }1\le p\le2\\
C_{0}p\delta^{-\frac{1}{2}} & \mbox{ if }p>2.
\end{cases}\label{eq:c(delta,p)}
\end{align}

\end{proposition}

\begin{remark}

If $k_{s+1},l_{s+1}<\infty$, then by monotonicity of the compression
function, the bound extends to the interval $\left[k_{s+1}l_{s},k_{s+1}l_{s+1}\right]$
, namely, for $x\in\left[k_{s+1}l_{s},k_{s+1}l_{s+1}\right]$, 
\[
\rho_{\Psi}\left(\frac{1}{2}x\right)\le\rho_{\Psi}\left(\frac{1}{2}k_{s+1}l_{s+1}\right)\le C(\delta,p)k_{s+1}.
\]
If $l_{s+1}=\infty$, the situation is different, we need to have
information regarding compression of the infinite group $\Gamma$.
See Lemma \ref{expander-Lp-infty}.

\end{remark}

\begin{proof} Consider the subgroup $\Gamma_{s}'=[A(s),B(s)]^{\Gamma_{s}}=\ker\left(\Gamma_{s}\to A(s)\times B(s)\right)$.
Take the symmetric generating set $R(s)$ for $\Gamma_{s}'$ using
the Reidemeister-Schreier algorithm in Lemma \ref{RS}, where $F=A(s)\times B(s)$,
$S=A(s)\cup B(s)$. Then the inclusion map from $\left(\Gamma_{s}',R(s)\right)$
into $\left(\Gamma_{s},A(s)\cup B(s)\right)$ is bi-Lipschitz, $|\gamma|_{R(s)}\leq|\gamma|_{\Gamma_{s}}\leq5|\gamma|_{R}$
for all $\gamma\in\Gamma_{s}'$. Let $\mu_{s}$ be the uniform distribution
on $R(s)$. It is known that if there is a $(C,C)$-quasi isometric
map $\psi:(G,S)\to(H,T)$ and image $\psi(G)$ is $R$ dense in $H$,
then the Poincaré constant of $(H,\nu)$ is comparable to the Poincaré
constant of $(G,\mu)$ with constants only depending on $C,R,\left|S\right|,\left|T\right|$,
where $\mu$ ($\nu$ resp.) is the uniform distribution on $S\cup S^{-1}$
($T\cup T^{-1}$ resp.), see the proof of \cite[Proposition 4.2	]{Coulhon1995}
or \cite[Theorem 1.2]{PittetSaloffCoste2000}. In the current situation,
since the inclusion map $\left(\Gamma_{s}',R(s)\right)$ into $\left(\Gamma_{s},A(s)\cup B(s)\right)$
is a $(5,5)$-quasi-isometry, and $\Gamma_{s}'$ is $2$-dense in
$\Gamma_{s}$, there exists a constant $c_{0}$ only depending on
$|A|$ and $|B|$ such that the spectral gap of $\mu_{s}$ satisfies
\[
\lambda(\Gamma_{s}',\mu_{s})=\tilde{\delta}\ge c_{0}\delta.
\]

Let $t\in\left[k_{s},k_{s+1}\right]$, consider the direct product
$\Pi_{s}^{t}$ of $t$ copies of $\Gamma'_{s}$ in the factor $\Delta_{s}$
at site $0,1,\ldots,t-1$. By Subsection \ref{sub: embedded-metric},
$\Pi_{s}^{t}$ is an embedded subgroup of $\Delta$, denote such an
embedding by $\theta_{s}:\Pi_{s}^{t}\hookrightarrow\Delta$. On $\Pi_{s}^{t}$,
take the product kernel $\zeta_{t}=\left(\mu_{s}\right)_{0}\otimes\ldots\otimes\left(\mu_{s}\right)_{t-1}$.
By tensorizing property of classical Poincaré inequalities, we have
that for any function $f:\Pi_{t}\to\mathbb{R}$, 
\[
\sum_{u,v\in\Pi_{t}}\left|f(u)-f(v)\right|^{2}\pi_{\Pi_{s}^{t}}(u)\pi_{\Pi_{s}^{t}}(v)\le\tilde{\delta}^{-1}\sum_{u,v\in\Pi_{t}}\left|f(u)-f(v)\right|^{2}\pi_{\Pi_{t}}(u)\zeta_{t}\left(u^{-1}v\right).
\]
By Matousek's extrapolation lemma for Poincaré inequalities \cite{Matousek},
see the version in \cite[Lemma 4.4]{Naor2011a}, it follows that for
any $f:\Pi_{t}\to \ell^{p}$, 
\begin{itemize}
\item if $1\le p\le2$, 
\[
\sum_{u,v\in\Pi_{s}^{t}}\left\Vert f(u)-f(v)\right\Vert _{p}^{p}\pi_{\Pi_{s}^{t}}(u)\pi_{\Pi_{s}^{t}}(v)\le\tilde{\delta}^{-1}\sum_{u,v\in\Pi_{s}^{t}}\left\Vert f(u)-f(v)\right\Vert _{p}^{p}\pi_{\Pi_{s}^{t}}(u)\zeta_{t}\left(u^{-1}v\right);
\]

\item if $p>2$, 
\[
\sum_{u,v\in\Pi_{s}^{t}}\left\Vert f(u)-f(v)\right\Vert _{p}^{p}\pi_{\Pi_{s}^{t}}(u)\pi_{\Pi_{s}^{t}}(v)\le(2p)^{p}\tilde{\delta}^{-p/2}\sum_{u,v\in\Pi_{s}^{t}}\left\Vert f(u)-f(v)\right\Vert _{p}^{p}\pi_{\Pi_{s}^{t}}(u)\zeta_{t}\left(u^{-1}v\right).
\]

\end{itemize}
Let $\varphi:\Delta\to \ell^{p}$ be a $1$-Lipschitz uniform embedding
of $\Delta$ with respect to word metric $\left|\cdot\right|_{\Delta}$.
Apply Lemma \ref{spectral-compression} to the subset $\theta_{s}\left(\Pi_{t}\right)$
equipped with kernel $\zeta_{t}\circ\theta_{s}^{-1}$, with mass distribution condition satisfied by Lemma \ref{embed-product}, we have 
\[
\rho_{\varphi}\left(\frac{1}{2}tl_{s}\right)\le c(\tilde{\delta},p)t,
\]
where the constant $c(\tilde{\delta},p)$ is given by 
\begin{equation}
c(\tilde{\delta},p)=\begin{cases}
4(2/\tilde{\delta})^{\frac{1}{p}} & \mbox{ if }1\le p\le2\\
4\cdot2^{1+\frac{1}{p}}\tilde{\delta}^{-\frac{1}{2}}p & \mbox{ if }p>2.
\end{cases}\label{eq:c(delta-p)}
\end{equation}
From the standard fact that $L_{p}$ is $(1+\varepsilon)$-finitely
presentable in $l_{p}$, see for example in the proof of \cite[Theorem 1.1]{Jolissaint2014},
we conclude that for for any 1-Lipschitz uniform embedding $\Psi:\Delta\to L_{p}$
the same bound holds, 
\[
\rho_{\Psi}\left(\frac{1}{2}tl_{s}\right)\le c(\tilde{\delta},p)t,
\]

\end{proof}

\subsubsection{Upper bound with an infinite group $\Gamma_{s}$ having strong property
$(T)$}

Next we consider the case where $\Gamma_{s}=\Gamma$ is an infinite
group (it corresponds to $l_{s}=\infty$). Let $\Gamma$ be a discrete
group equipped with finite generating set $S$ and $\mathfrak{X}$ be
a Banach space. A linear isometric $\Gamma$-representation on $\mathfrak{X}$
is a homomorphism $\varrho:\Gamma\to O(\mathfrak{X})$, where $O(\mathfrak{X})$
denotes the groups of all invertible linear isometries of $\mathfrak{X}$.
Denote by $\mathfrak{X}^{\varrho(\Gamma)}$ the closed subspace of
$\Gamma$-fixed vectors. When $\mathfrak{X}$ is uniformly convex,
by \cite[Proposition 2.6]{Bader2007} the subspace of $\mathfrak{X}^{\varrho(\Gamma)}$
is complemented in $\mathfrak{X}$, $\mathfrak{X}=\mathfrak{X}^{\varrho(\Gamma)}\oplus\mathfrak{X}'(\varrho)$,
and the decomposition is canonical.

\begin{definition}Let $\Gamma$ be a discrete group equipped with
finite generating set $S$, $\mathfrak{X}$ be a uniformly convex
Banach space. 
\begin{itemize}
\item Following \cite{Bader2007}, we say that $\Gamma$ has Property $\left(F_{\mathfrak{X}}\right)$
if any action of $\Gamma$ on $\mathfrak{X}$ by affine isometries
has a $\Gamma$-fixed point. 
\item We say $\Gamma$ has Property $(T_{\mathfrak{X}})$ if there exists a constant $\varepsilon>0$
such that for any representation
$\varrho:\Gamma\to O(\mathfrak{X})$, 
\[
\max_{s\in S}\left\Vert \varrho(s)v-v\right\Vert _{\mathfrak{X}}\ge\varepsilon\left\Vert v\right\Vert _{\mathfrak{X}}\ \mbox{for all }v\in\mathfrak{X}'(\varrho).
\]
The maximal $\varepsilon$ with this property is called the $\mathfrak{X}$-Kazhdan
constant of $\Gamma$ with respect to $S$ and is denoted by $\kappa_{\mathfrak{X}}(\Gamma,S)$. 
\end{itemize}
\end{definition}

By \cite[Theorem 1.3]{Bader2007}, Property $\left(F_{\mathfrak{X}}\right)$
implies Property $\left(T_{\mathfrak{X}}\right)$ in any Banach space
$\mathfrak{X}$.

\begin{lemma}\label{expander-Lp-infty}

Let $\Delta$ be the diagonal product with parameters $(k_{s})_{s\le\mathfrak{s}_{0}}$
and lamp groups $\{\Gamma_{s}\}_{s\le\mathfrak{s}_{0}}$, where $\Gamma_{\mathfrak{s}_{0}}=\Gamma$
is an infinite group marked with generating subgroups $A,B$. Suppose
$\mathfrak{X}$ is a uniformly convex Banach space and $\Gamma$ has
Property $(F_{\mathfrak{X}})$. Then for any equivariant 1-Lipschitz
embedding $\Psi:\Delta\to\mathfrak{X}$, the compression function
of $\Psi$ satisfies 
\[
\rho_{\Psi}\left(x\right)\le\frac{2}{\kappa_{\mathfrak{X}}(\Gamma,A\cup B)}\left(k_{\mathfrak{s}_{0}}+1\right),\ \mbox{for all }x\in\left[k_{\mathfrak{s}_{0}}+1,\infty\right].
\]

\end{lemma}

\begin{proof}

Since the embedding $\vartheta_{\mathfrak{s}_{0}}:\Gamma\hookrightarrow\Delta$
is a homothety with $\left|\vartheta_{\mathfrak{s}_{0}}(\gamma)\right|_{\Delta}=(k_{\mathfrak{s}_{0}}+1)\left|\gamma\right|_{\Gamma}$,
$\psi=\Psi\circ\vartheta_{\mathfrak{s}_{0}}:\Gamma\to\mathfrak{X}$
is a $(k_{\mathfrak{s}_{0}}+1)$-Lipschitz equivariant embedding.
Consider $\tilde{\psi}=\frac{\psi}{k_{\mathfrak{s}_{0}}+1}$. Since
$\tilde{\psi}$ equivariant, it is a $1$-cocycle with respect to
some representation $\varrho:\Gamma\to O\left(\mathfrak{X}\right)$.
Since $\Gamma$ has Property $(F_{\mathfrak{X}})$, $H^{1}(\Gamma,\varrho)=Z^{1}(\Gamma,\varrho)/B^{1}(\Gamma,\varrho)$
vanishes, it follows that $\tilde{\psi}$ is a $1$-coboundary, that
is there exists $v\in\mathfrak{X}$ such that 
\[
\tilde{\psi}(g)=\varrho(g)v-v.
\]
We may take $v$ in the complement $\mathfrak{X}'(\varrho)$, Then
by Property ($T_{\mathfrak{X}}$), 
\[
\max_{s\in A\cup B}\left\Vert \varrho(s)v-v\right\Vert _{\mathfrak{X}}\ge\kappa\left\Vert v\right\Vert _{\mathfrak{X}},\ \mbox{where }\kappa=\kappa_{\mathfrak{X}}(\Gamma,A\cup B).
\]
Since $\tilde{\psi}$ is $1$-Lipschitz, we have $\kappa\left\Vert v\right\Vert _{\mathfrak{X}}\le\max_{s\in A\cup B}\left\Vert \tilde{\psi}(s)\right\Vert _{\mathfrak{X}}\le1$,
it follows that for any $g\in\Gamma$, 
\[
\left\Vert \tilde{\psi}(g)\right\Vert _{\mathfrak{X}}=\left\Vert \varrho(g)v-v\right\Vert _{\mathfrak{X}}\le\left\Vert \varrho(g)v\right\Vert _{\mathfrak{X}}+\left\Vert v\right\Vert _{\mathfrak{X}}=2\left\Vert v\right\Vert _{\mathfrak{X}}\le2/\kappa.
\]
Now we get back to $\Psi$. Since $\Psi\circ\vartheta_{\mathfrak{s}_{0}}=(k_{\mathfrak{s}_{0}}+1)\tilde{\psi}$,
and $\left|\vartheta_{\mathfrak{s}_{0}}(\gamma)\right|_{\Delta}=(k_{\mathfrak{s}_{0}}+1)\left|\gamma\right|_{\Gamma}$,
we deduce from $\left\Vert \tilde{\psi}(g)\right\Vert _{\mathfrak{X}}\le2/\kappa$
that 
\[
\rho_{\Psi}\left(x\right)\le\frac{2}{\kappa_{\mathfrak{X}}(\Gamma,A\cup B)}\left(k_{\mathfrak{s}_{0}}+1\right),\ \mbox{for all }x\in\left[k_{\mathfrak{s}_{0}}+1,\infty\right].
\]

\end{proof}

\begin{remark}

In practice, we use the bound in Lemma \ref{expander-Lp-infty} for
the interval $\left[\left(k_{\mathfrak{s}_{0}}+1\right)l_{\mathfrak{s}_{0}-1},\infty\right]$,
because for smaller length $x$, the copies $\Gamma_{s}$ with $s\le\mathfrak{s}_{0}-1$
provide better upper bounds.

\end{remark}

Property $\left(F_{\mathfrak{X}}\right)$ is very strong. By Bader-
Furman-Gelander-Monod \cite[Theorem B]{Bader2007} and standard Hereditary
properties (\cite[Section 2.5]{Bekka2008}), the lattice $\Gamma$
in Example \ref{examples} has Property $F_{L_{p}}$ for all $1<p<\infty$.

When we specialize to Lebesgue spaces $L_{p}$, $p\in(1,\infty)$, the
$p$-Kazhdan constant can be estimated in terms of the Kazhdan constant
in Hilbert space, via the explicit Mazur map. 

\begin{fact}\label{kazhdan-Lp}[follows from \cite{Bader2007}, \cite{Mazur1929}]

Let $\Gamma$ be a discrete group equipped with finite generating
set $S$, suppose $\Gamma$ has Kazhdan property $(T_{L_{2}})$. Then
for $p>2$, $\kappa_{L_{p}}(\Gamma,S)\ge\frac{1}{p2^{p/2}}\kappa_{L_{2}}(\Gamma,S)$;
for $1<p<2$, $\kappa_{L_{p}}(\Gamma,S)\ge2^{-\frac{p+2}{p}}\kappa_{L_{2}}^{2/p}(\Gamma,S)$.

\end{fact}

\begin{proof}

Let $\varrho:\Gamma\to O(L_{p})$ be a $\Gamma$-representation in
$L_{p}$, take any unit vector $f$ in the complement $ $$\mathfrak{X}'(\varrho)$.
Let $M_{p,q}:L_{p}\to L_{q}$ be the Mazur map 
\[
M_{p,q}(f)=\mbox{sign}(f)|f|^{p/q}.
\]
By \cite[Lemma 4.2]{Bader2007}, the conjugation $U\mapsto M_{p,2}\circ U\circ M_{2,p}$
sends $O(L_{p})$ to $O(L_{2})$. Define $\pi:\Gamma\to O(L_{2})$
by $\pi(g)=M_{p,2}\circ\varrho(g)\circ M_{2,p}$. By definition of
the Mazur map, we have $\left\Vert M_{p,2}(f)\right\Vert _{2}^{2}=\left\Vert f\right\Vert _{p}^{p}=1$.

Consider first the case $p>2$. From \cite{Mazur1929} 
\begin{equation}
\left||a|^{\frac{2}{p}}\mbox{sign}(a)-|b|^{\frac{2}{p}}\mbox{sign}(b)\right|\le2|a-b|^{\frac{2}{p}},\label{eq:mazur-1}
\end{equation}
we have for $u,v\in L_{2}$, $\left\Vert M_{2,p}(u)-M_{2,p}(v)\right\Vert _{p}^{p}\le2^{p}\left\Vert u-v\right\Vert _{2}^{2}$.
Note that $M_{p,2}$ maps $L_{p}^{\varrho(\Gamma)}$ onto $L_{2}^{\pi(\Gamma)}$,
therefore for any unit vector $f\in\mathfrak{X}'(\varrho)$, we have
\[
\inf_{v\in L_{2}^{\pi(\Gamma)}}\left\Vert M_{p,2}(f)-v\right\Vert _{2}^{2}\ge\frac{1}{2^{p}}\inf_{h\in L_{p}^{\varrho(\Gamma)}}\left\Vert f-h\right\Vert _{p}^{p}\ge\frac{1}{2^{p}}.
\]
That is the projection of $M_{p,2}(f)$ to $\left(L_{2}^{\pi(\Gamma)}\right)^{\perp}$
has $L_{2}$-norm at least $2^{-p/2}$. By \cite{Mazur1929} 
\begin{equation}
\left||a|^{p/2}\mbox{sign}(a)-|b|^{p/2}\mbox{sign}(b)\right|\le\frac{p}{2}|a-b|\left(|a|^{\frac{p}{2}-1}+|b|^{\frac{p}{2}-1}\right),\label{eq:mazur-2}
\end{equation}
we have 
\begin{align}
\left\Vert \pi(s)u-u\right\Vert _{2}^{2} & =\int\left||\varrho(s)f|^{p/2}\mbox{sign}(\varrho(s)f)-|f|^{p/2}\mbox{sign}(f)\right|^{2}dm\nonumber \\
 & \le\int\left(\frac{p}{2}\right)^{2}\left|\varrho(s)f-f\right|^{2}\left(\left|\varrho(s)f\right|^{\frac{p}{2}-1}+|f|^{\frac{p}{2}-1}\right)^{2}dm\nonumber \\
 & \le\left(\frac{p}{2}\right)^{2}\left(\int\left|\varrho(s)f-f\right|^{p}dm\right)^{2/p}\left(\int\left(\left|\varrho(s)f\right|^{\frac{p}{2}-1}+|f|^{\frac{p}{2}-1}\right)^{\frac{2p}{p-2}}dm\right)^{\frac{p-2}{p}}.\label{eq:mazur-holder}
\end{align}
The last step uses Hölder inequality. By triangle inequality in $L_{\frac{p}{p-2}}$
and $\varrho(s)\in O(L_{p})$, 
\[
\left(\int\left(\left|\varrho(s)f\right|^{\frac{p}{2}-1}+|f|^{\frac{p}{2}-1}\right)^{\frac{2p}{p-2}}dm\right)^{\frac{p-2}{p}}\le4\left\Vert f\right\Vert _{p}^{p-2}=4.
\]
Let $u'$ be the projection of $u$ to $L_{2}'(\pi)$, then 
\[
\max_{s\in S}\left\Vert \pi(s)u-u\right\Vert _{2}=\max_{s\in S}\left\Vert \pi(s)u'-u'\right\Vert _{2}\ge\kappa_{L_{2}}(\Gamma,S)\left\Vert u'\right\Vert _{2}\ge2^{-p/2}\kappa_{L_{2}}(\Gamma,S).
\]
Combine with (\ref{eq:mazur-holder}), 
\[
1\le\frac{2^{p}}{\kappa_{L_{2}}(\Gamma,S)^{2}}\max_{s\in S}\left\Vert \pi(s)u-u\right\Vert _{2}^{2}\le\left(\frac{p2^{p/2}}{\kappa_{L_{2}}(\Gamma,S)}\right)^{2}\max_{s\in S}\left\Vert \varrho(s)f-f\right\Vert _{p}^{2}.
\]
We conclude that $\kappa_{L_{p}}(\Gamma,S)\ge\frac{1}{p2^{p/2}}\kappa_{L_{2}}(\Gamma,S)$.

In the case $1<p<2$, rewrite (\ref{eq:mazur-1}) into 
\[
|a-b|^{p}\le2^{p}\left||a|^{p/2}\mbox{sign}(a)-|b|^{p/2}\mbox{sign}(b)\right|^{2},
\]
we deduce that the projection of $M_{p,2}(f)$ to $\left(L_{2}^{\pi(\Gamma)}\right)^{\perp}$
has $L_{2}$-norm at least $2^{-p/2}$. Apply (\ref{eq:mazur-1}),
\begin{align*}
\left\Vert \pi(s)u-u\right\Vert _{2}^{2} & =\int\left||\varrho(s)f|^{p/2}\mbox{sign}(\varrho(s)f)-|f|^{p/2}\mbox{sign}(f)\right|^{2}dm\\
 & \le\int4|\varrho(s)f-f|^{p}dm=4\left\Vert \varrho(s)f-f\right\Vert _{p}^{p}.
\end{align*}
It follows that $\kappa_{L_{p}}(\Gamma,S)\ge2^{-\frac{p+2}{p}}\kappa_{L_{2}}^{2/p}(\Gamma,S).$

\end{proof}

\subsubsection{Basic test functions and 1-cocycles on $\Delta$\label{sub:Basic-test-functions}}

The discussion in this subsection is valid for any choice of $\{\Gamma_{s}\}$
that satisfies Assumption \ref{Gamma_Involution}. The $1$-cocycle
constructed using the basic test functions will be useful in later
sections as well.

First recall some basic test functions on $\Delta$ constructed in
Subsection \ref{sub:Isoperimetry-delta}. They capture the feature
that in each copy of $\Gamma_{s}$ in $\Delta_{s}$, the generators
$a_{i}(s)$ and $b_{i}(s)$ are kept distance $k_{s}$ apart. In the
group $\Delta$, for $r\ge2$, define subset $U_{r}$ as 
\[
U_{r}=\left\{ Z\in\Delta:\ \mbox{Range}(Z)\subseteq[-r,r]\right\} .
\]
Recall that $\mbox{Range}(Z)$ is defined in subsection \ref{sub:Local-finiteness},
it is the minimal interval of $\mathbb{Z}$ visited by the cursor
of a path representing $Z$. Take a function supported on the subset
$U_{r}$, 
\begin{equation}
\varphi_{r}\left(\left(f_{s}\right),z\right)=\max\left\{ 0,1-\frac{\left|z\right|}{r}\right\} \mathbf{1}_{U_{r}}\left(\left(f_{s}\right),z\right).\label{eq:test-function}
\end{equation}
Let $\mathfrak{q}$ be the switch-or-walk measure $\mathfrak{q}=\frac{1}{2}(\mu+\nu)$
on $\Delta$ where $\nu$ is the uniform measure on $\left\{ \alpha_{i},\beta_{j}:\ 1\le i\le\left|A\right|,1\le j\le\left|B\right|\right\} $
and $\mu$ is the simple random walk measure on the base $\mathbb{Z}$,
$\mu\left(\tau^{\pm1}\right)=\frac{1}{2}$. We have seen in the proof of Proposition
\ref{iso-Delta} that for $p=2$, 
\[
\frac{\mathcal{E}_{\Delta,\mathfrak{q}}\left(\varphi_{r}\right)}{\left\Vert \varphi_{r}\right\Vert _{2}^{2}}\sim\frac{3}{2r^{2}}.
\]
Define $\varphi_{1}$ to be the indicator function of the identity
$e_{\Delta}$, 
\[
\varphi_{1}=\mathbf{1}_{e_{\Delta}}.
\]
Motivated by Tessera's embedding in \cite[Section 3]{Tessera2011},
given a non-decreasing function $\gamma:\mathbb{\mathbb{Z}}_{\ge0}\to\mathbb{R}_{+}$
with $\gamma(0)=1$ and 
\begin{equation}
C(\gamma)=\sum_{t=1}^{\infty}\left(\frac{1}{\gamma(t)}\right)^{2}<\infty,\label{eq:C_p}
\end{equation}
we define a 1-cocycle $b_{\gamma}:\Delta\to\oplus_{j=0}^{\infty}\ell^{2}(\Delta)\subset L_{2}$
\begin{equation}
b_{\gamma}(Z)=\bigoplus_{j=0}^{\infty}\left(\frac{1}{\gamma(j)}\frac{\varphi_{2^{j}}-\tau_{Z}\varphi_{2^{j}}}{\mathcal{E}_{\Delta,\mathfrak{q}}\left(\varphi_{2^{j}}\right)^{\frac{1}{2}}}\right)\label{eq:cocycle}
\end{equation}
where $\tau_{g}$ denote right translation of functions, $\tau_{g}\varphi(h)=\varphi\left(hg^{-1}\right).$

\begin{lemma}\label{compression-expander-lower}

The 1-cocycle $b_{\gamma}:\Delta\to L_{2}$ defined in (\ref{eq:cocycle})
is $\sqrt{2C(\gamma)}$-Lipschitz. Suppose in addition there exists
$m_{0}\ge2$ such that $k_{s+1}\ge m_{0}k_{s}$, $l_{s+1}\ge m_{0}l_{s}$
for any $s\ge1$. Then there is a constant $C(m_{0})>0$ depending
only on $m_{0}$ such that 
\begin{align*}
\rho_{b_{\gamma}}(x) & \ge C(m_{0})\frac{x/l_{s}}{\gamma(\log_{2}(x/l_{s}))},\ \mbox{if }x\in[k_{s}l_{s},k_{s+1}l_{s}),\\
\rho_{b_{\gamma}}(x) & \ge C(m_{0})\frac{k_{s+1}}{\gamma\left(\log_{2}k_{s+1}\right)},\ \mbox{if }x\in[k_{s+1}l_{s},k_{s+1}l_{s+1}).
\end{align*}

\end{lemma}

\begin{proof} Clearly $\Vert b_{\gamma}(Z)\Vert=1$ when $Z$ is
a generator in $A\cup B$ and $\Vert b_{\gamma}(Z)\Vert=\sqrt{2C(\gamma)}$
when $Z=\tau$ is the generator of the cyclic base $\Z$.

The $1$-cocycle $b_{\gamma}$ captures the size of $Range(Z)$. Denote
$Z=((f_{s}),i)$ and observe that if $Range(Z)>2^{j+2}$, then 
\[
\left(\mbox{supp}\varphi_{2^{j}}\right)\cap\left(\mbox{supp}\tau_{Z}\varphi_{2^{j}}\right)=\emptyset.
\]
Indeed, either there exists $\iota\notin[-2^{j+1},2^{j+1}]$ and $s\leq s_{0}(Z)$
with $f_{s}(\iota)\neq e{}_{\Delta_{s}}$ and then this also holds
for all elements of $\mbox{supp}\tau_{Z}\varphi_{2^{j}}$ and none
of $\mbox{supp}\varphi_{2^{j}}$, or $|i|>2^{j+1}$ and then the projections
on the base of the two supports are disjoint. Therefore 
\[
\frac{\left\Vert \varphi_{2^{j}}-\tau_{Z}\varphi_{2^{j}}\right\Vert _{2}^{2}}{\mathcal{E}_{\Delta,\mathfrak{q}}\left(\varphi_{2^{j}}\right)}=\frac{2\left\Vert \varphi_{2^{j}}\right\Vert _{2}^{2}}{\mathcal{E}_{\Delta,\mathfrak{q}}\left(\varphi_{2^{j}}\right)}\sim\frac{4\cdot2^{2j}}{3}.
\]
By construction of $b_{\gamma}$, this implies if $Range(Z)>2^{j+2}$,
\[
\left\Vert b_{\gamma}(Z)\right\Vert _{2}\ge\frac{2^{j}}{\sqrt{3}\gamma(j)}.
\]
By Definition \ref{range}, for $Z$ with $\mbox{Range}(Z)=r\in[k_{s},k_{s+1})$,
we have $s_{0}(Z)\le s$.

Denote $Z=((f_{s}),i)$, then by Lemma \ref{one_metric} 
\[
|(f_{s},i)|_{\Delta_{s}}=|\pi_{s}(Z)|_{\Delta_{s}}\leq18(r+1)l_{s}
\]
because at most $2r/k_{s}+1$ intervals contribute to the essential
contribution. By Proposition \ref{metric}, the word distance of $Z$
to $e_{\Delta}$ is bounded 
\[
\left|Z\right|_{\Delta}\le500\ .\ 18(r+1)\left(l_{0}+\ldots+l_{s}\right)\le\frac{9000(r+1)l_{s}}{1-1/m_{0}}.
\]
It follows that 
\[
\rho_{b_{\gamma}}\left(\frac{9000(r+1)l_{s}}{1-1/m_{0}}\right)\ge\frac{r}{8\gamma(\log_{2}r)}.
\]
To write it into the first inequality stated, note that since $b_{\gamma}$
is equivariant, $\rho_{b_{\gamma}}$ is subadditive.

The second bound follows from the first bound evaluated at $x=k_{s+1}l_{s}$
and the monotonicity of the compression function $\rho_{b_{\gamma}}$.
\end{proof}

\begin{remark}

The function $t/\gamma\circ\log(t)$ with $\gamma$ satisfying (\ref{eq:C_p})
does not satisfy any a priori majoration by a sublinear function.
More precisely, \cite[Proposition 8]{Tessera2011} implies that for
any increasing sublinear function $h:\mathbb{R}_{+}\to\mathbb{R}_{+}$,
there exists a nondecreasing function $\gamma$ satisfying (\ref{eq:C_p}),
a constant $c>0$ and an increasing subsequence of integers $n_{i}$,
such that 
\[
\frac{k_{n_{i}}}{\gamma\circ\log_{2}\left(k_{n_{i}}\right)}\ge ch\left(k_{n_{i}}\right)\ \mbox{for all }i.
\]
It follows that the upper bound in Proposition \ref{expander-upper-Lp}
cannot be improved.

\end{remark}

\subsubsection{Possible compression gap of embedding into $L_{p}$}

Comparing Lemma \ref{compression-expander-lower} to Proposition \ref{expander-upper-Lp},
we see that the 1-cocycle $b_{\gamma}$ defined in (\ref{eq:cocycle})
is almost optimal in the sense that it matches up with the compression
upper bound up to the factor sequence $1/\gamma\circ\log_{2}(k_{s})$.
The following result is an analogue of \cite[Theorem 5.5 (II)]{Arzhantseva2009}
in our setting.

\begin{theorem}\label{compression-expander}

There exists absolute constants $\delta>0,C>0$ such that the following
holds. Let $\rho(x)$ be any non-decreasing function such that $\frac{x}{\rho(x)}$
is non-decreasing. Then there exists a finitely generated group $\Delta$
such that $\left(\frac{1}{C\varepsilon}\frac{\rho}{\log(1+\rho)^{(1+\varepsilon)/2}},Cp2^{p/2}\rho\right)$
is an equivariant $L_{p}$-compression gap of $\Delta$ for any $p>1$.

Further if $\lim_{x\to\infty}\rho(x)=\infty$, the group $\Delta$
constructed is elementary amenable, and $\left(\frac{1}{C\varepsilon}\frac{\rho}{\log(1+\rho)^{(1+\varepsilon)/2}},C\rho\right)$
is an equivariant $L_{1}$-compression gap of $\Delta$.

\end{theorem}

\begin{proof} We write $\rho(x)=\frac{x}{f(x)}$ with $f(x)$ between
$1$ and $x$. The sets $K=\Z_{+}\cup\{\infty\}$ and $L=\left\{ \mbox{diam }\Gamma_{m},m\geq1\right\} \cup\{\infty\}$
of diameters of Lafforgue expanders from Example \ref{examples} satisfy
the assumptions of Proposition \ref{function-approximation}. So we
can find sequences $(k_{s}),(l_{s})$ taking values in $K$ and $L$
such that the function defined by $\tilde{f}(x)=l_{s}$ on $[k_{s}l_{s},k_{s+1}l_{s}]$
and $\tilde{f}(x)=\frac{x}{k_{s+1}}$ on $[k_{s+1}l_{s},k_{s+1}l_{s+1}]$
satisfies $\tilde{f}(x)\simeq_{m_{0}C_{1}^{5}}f(x)$. Then $\tilde{\rho}(x)=\frac{x}{\tilde{f}(x)}\simeq_{m_{0}C_{1}^{5}}\rho(x)$.
Let $\Delta$ be the diagonal product associated to these sequences.

By Proposition \ref{expander-upper-Lp} and Lemma \ref{expander-Lp-infty},
any $1$-Lipschitz equivariant embedding $\Psi:\Delta\rightarrow L_{p}$
satisfies for all $x$ 
\[
\rho_{\Psi}(x)\leq C\tilde{\rho}(2x)\leq2m_{0}C_{1}^{5}C\rho(x),
\]
where by Fact \ref{kazhdan-Lp} 
\[
C\leq\max\{C(\delta,p),\frac{2}{\kappa_{L_{p}}(\Gamma)})\}\leq\max\{C'p,C''p2^{p/2}\}.
\]
This gives the upper bound of the compression gap. For $p=1$, Lemma
\ref{expander-Lp-infty} does not hold, so the upper bound is valid
only on the condition that all the diameters $l_{s}$ are finite,
which is satisfied when $\rho$ is unbounded, or equivalently $f$
is not asymptotically linear.

The lower bound is given by the $1$-cocyle $b_{\gamma}:\Delta\to \ell^{2}$
of Lemma \ref{compression-expander-lower} with $\gamma(x)=C_{\epsilon}x^{\frac{1+\epsilon}{2}}$,
where $C_{\epsilon}\sim\frac{1}{\epsilon}$ is such that $\sqrt{2C(\gamma)}=1$.
For all $x$ 
\[
\rho_{b_{\gamma}}(x)\geq C_{\epsilon}C(m_{0})\frac{\tilde{\rho}(x)}{\log_{2}(\tilde{\rho}(x))^{\frac{1+\epsilon}{2}}}\geq\frac{1}{C\epsilon}\frac{\rho(x)}{\log(1+\rho(x))^{\frac{1+\epsilon}{2}}}
\]
Since $\ell^{2}$ embeds isometrically in $L_{p}$ for all $p\ge1$,
see Lemma 2.3 in \cite{Naor2008}, it is also an $L_{p}$-compression
lower bound. \end{proof}

\section{$L_{p}$-compression of wreath product $H\wr\mathbb{Z}$ \label{HwrZ}}

In general the case with $\{\Gamma_{s}\}$ chosen be to finite groups
other than expanders is more involved. Since our main object $\Delta$
is a diagonal product of a sequence of wreath products, it is instructive
to understand compression of uniform embedding of a single wreath
product $H\wr\mathbb{Z}$.

In \cite{Naor2008}, Naor and Peres proved that if $\alpha_{2}^{\#}(H)=\frac{1}{2\beta^{*}(H)}$,
where $\beta^{\ast}(H)$ is the supremum of upper speed exponent of
symmetric random walk of bounded support on $H$, then (\cite[Corollary 1.3]{Naor2008})
\[
\alpha_{2}^{\#}(H\wr\mathbb{Z})=\frac{2\alpha_{2}^{\#}(H)}{2\alpha_{2}^{\#}(H)+1}.
\]
Further in \cite{Naor2011} which significantly extends the method
in \cite{Naor2008}, the $L_{p}$-compression exponent of $\mathbb{Z}\wr\mathbb{Z}$
was determined (\cite[Theorem 1.2]{Naor2011}), for every $p\ge1$,
\[
\alpha_{p}^{\#}\left(\mathbb{Z}\wr\mathbb{Z}\right)=\max\left\{ \frac{p}{2p-1},\frac{2}{3}\right\} .
\]
In \cite{Naor2011} Naor and Peres also proved the following result
when the base group is of polynomial volume growth at least quadratic.
Let $H$ be a nontrivial finitely generated amenable group and $\Gamma$
a group of polynomial volume growth. Suppose the volume growth rate
of $\Gamma$ is at least quadratic, then for every $p\in[1,2]$ (\cite[Theorem 3.1]{Naor2011})
\[
\alpha_{p}^{\#}(H\wr\Gamma)=\min\left\{ \frac{1}{p},\alpha_{p}^{\#}(H)\right\} .
\]

One central idea in these works is the Markov type method that connects
compression exponents of $G$ to speed exponent of certain random
walks on $G$. In this subsection we apply the spectral method as
in subsection \ref{sub:Spectral-method} to obtain a generalization
of the aforementioned results of Naor and Peres. The technique will
be useful in the study of the diagonal product $\Delta$ with dihedral
groups.

\begin{theorem}\label{HwrZformula}

Let $p\in[1,2]$, $H$ be a finitely generated infinite group. Then
the equivariant $L_{p}$-compression exponent of $H\wr\mathbb{Z}$
is 
\[
\alpha_{p}^{\#}(H\wr\mathbb{Z})=\min\left\{ \frac{\alpha_{p}^{\#}(H)}{\alpha_{p}^{\#}(H)+\left(1-\frac{1}{p}\right)},\alpha_{p}^{\#}(H)\right\} .
\]

\end{theorem}

\begin{remark}

The lower bound on $\alpha_{p}^{\#}(H\wr\mathbb{Z})$ is covered by
\cite[Theorem 3.3]{Naor2008}. We will prove the upper bound. Note
that there is an interesting dichotomy of the type of obstruction
that $H\wr\mathbb{Z}$ observes, depending on $\alpha_{p}^{\#}(H)$,
$p\in(1,2]$. 
\begin{enumerate}
\item When $\alpha_{p}^{\#}(H)>\frac{1}{p}$, then the sequence of subsets
$X_{n}$ as defined in the proof of Proposition \ref{rho-Phi} captures
distorted elements under the embedding, 
\[
\alpha_{p}^{\#}(H\wr\mathbb{Z})=\frac{\alpha_{p}^{\#}(H)}{\alpha_{p}^{\#}(H)+\left(1-\frac{1}{p}\right)}.
\]

\item When $\alpha_{p}^{\#}(H)\le\frac{1}{p}$, then one single copy of
$H$ already provide sufficient distortion, 
\[
\alpha_{p}^{\#}(H\wr\mathbb{Z})=\alpha_{p}^{\#}(H).
\]

\end{enumerate}
\end{remark}

\subsection{Upper bound of compression function of $H\wr\mathbb{Z}$\label{sub:-compression-wr-Z}}

Let $\Psi:H\wr\mathbb{Z}\to\mathfrak{X}$ be an equivariant embedding
of $G=H\wr\mathbb{Z}$ into metric space $\left(\mathfrak{X},d_{\mathfrak{X}}\right)$.
Recall that the group $H$ is naturally identified with the lamp group
over site $0$, 
\begin{align*}
i: & H\hookrightarrow G\\
h & \to\left(h\delta_{0},0\right).
\end{align*}
Then the embedding $\Psi$ induces an equivariant embedding of $H$
into $\mathfrak{X}$ we denote it by $\psi_{H}$ 
\begin{equation}
\psi_{H}(h)=\Psi\circ i(h)=\Psi\left(\left(h\delta_{0},0\right)\right).\label{eq:induce-H}
\end{equation}
Denote the compression function of $\psi_{H}:H\to\mathfrak{X}$ by
$\rho_{\psi_{H}}$. Since distortion of the inclusion map $i$ is
$1$, $\left|h\right|_{H}=\left|i(h)\right|_{G}$, it follows that
\begin{equation}
\rho_{\Psi}(t)\le\rho_{\psi_{H}}(t)\ \mbox{ for all }t\ge1.\label{eq:rho-G-H-1}
\end{equation}
We now explain how to apply the spectral method to derive a second
upper bound on $\rho_{\Psi}$ when $\Psi$ is a uniform embedding
of $G$ into $L_{p}$, $p\in(1,2]$. The novelty here is in the choice
of Markov kernels on lamplighter graphs.

\begin{proposition}\label{rho-Phi}

There exists a constant $C>0$ such that for any $p\in(1,2]$ and any $1$-Lipschitz equivariant embedding $\Psi:G\to L_{p}$  of $G=H\wr\mathbb{Z}$ into $L_{p}$,
the compression function $\rho_{\Psi}$ of $\Psi$ satisfies 
\begin{equation}
\rho_{\Psi}(t)\le C\left(\frac{p}{p-1}\right)^{\frac{1}{p}}\left(2\rho_{\psi_{H}}\circ\tau^{-1}(t)\right)^{\frac{p}{p-1}}\log^{\frac{1}{p}}\left(2\rho_{\psi_{H}}\circ\tau^{-1}(t)\right).\label{eq:rho-G-H-2}
\end{equation}
where $\psi_{H}$ is the induced embedding of the subgroup $H$ into
$L_{p}$ as in (\ref{eq:induce-H}) and the function $\tau:\mathbb{R}_{+}\to\mathbb{R}_{+}$
is defined as 
\[
\tau(x)=2^{\frac{1}{p-1}}x\left(\rho_{\psi_{H}}(x)\right)^{\frac{p}{p-1}}.
\]

\end{proposition}

\begin{proof}

By definition of the compression function $\rho_{\psi_{H}}$, for
every $n\ge1$, there exists an element $h_{n}\in H$ such that $\left|h_{n}\right|_{H}\ge n$
and 
\[
\left\Vert \psi_{H}\left(h_{n}\right)-\psi_{H}\left(e_{H}\right)\right\Vert _{p}\le2\rho_{\psi_{H}}(n).
\]
Take $X_{n}$ to be a finite subset in $G$ defined by 
\[
X_{n}=\left\{ \left(f,z\right):\ z\in\left[0,m_{n}-1\right],\ \begin{array}{c}
f(x)\in\left\{ e_{H},h_{n}\right\} \ \mbox{for }x\in\left[0,m_{n}-1\right]\\
f(x)=e_{H}\ \mbox{for }x\notin\left[0,m_{n}-1\right]
\end{array}\right\} .
\]
The length $m_{n}$ will be determined later. Note that $X_{n}$ has
the structure of a lamplighter graph, namely let $\mathcal{L}_{m_{n}}$
be lamplighter graph over segment $\left[0,m_{n}-1\right]$ defined
in Appendix \ref{stable}, then there is a bijection 
\begin{align*}
\sigma_{n}:\mathcal{L}_{m_{n}} & \to X_{n}\\
\sigma_{n}\left(f,x\right) & =\left(\tilde{f},x\right)\ \mbox{where }\tilde{f}(z)=h_{n}^{f(z)}.
\end{align*}
In Appendix \ref{stable} we defined a Markov transition kernel $\mathfrak{p}_{m_{n}}$
on $\mathcal{L}_{m_{n}}$ that moves on the base segment with a Cauchy-like
step distribution $\zeta_{m_{n}}$. On $X_{n}$, take the Markov kernel
$K_{n}$ to be $\mathfrak{p}_{m_{n}}\circ\sigma_{n}^{-1}$, that is
\[
K_{n}(u,v)=\mathfrak{p}_{m_{n}}\left(\sigma_{n}^{-1}(u),\sigma_{n}^{-1}(v)\right).
\]
Denote by $\pi_{n}$ the stationary distribution of $K_{n}$ on $X_{n}$,
$\pi_{n}=U_{m_{n}}\circ\sigma_{n}^{-1}$. Under the bijection $\sigma_{n}$,
the Poincaré inequality that $\Psi\circ\sigma_{n}:\mathcal{L}_{m_{n}}\to L_{p}$
satisfies as in Lemma \ref{lamplighter-poincare} implies 
\begin{equation}
\sum_{u,v\in X_{n}}\left\Vert \Psi(u)-\Psi(v)\right\Vert _{p}^{p}\pi_{n}(u)\pi_{n}(v)\le Cm_{n}\log m_{n}\sum_{u,v\in X_{n}}\left\Vert \Psi(u)-\Psi(v)\right\Vert _{p}^{p}K_{n}\left(u,v\right)\pi_{n}(u).\label{eq:Poincare-H-Z}
\end{equation}

Now we deduce an upper bound on $\rho_{\Psi}$ from Poincaré inequalities
(\ref{eq:Poincare-H-Z}) by applying Lemma \ref{spectral-compression}.
Because of equivariance of $\Psi$, for any $u\in G$, 
\[
\left\Vert \Psi\left(u\cdot(h\delta_{0},0)\right)-\Psi(u)\right\Vert _{p}=\left\Vert \Psi\left((h\delta_{0},0)\right)-\Psi(e_{G})\right\Vert _{p}=\left\Vert \psi_{H}\left(h\right)-\psi_{H}(e_{H})\right\Vert _{p}.
\]
Recall that $K_{n}$ moves as a \textquotedbl{}switch-walk-switch\textquotedbl{}
transition kernel, by Hölder inequality, 
\begin{align*}
 & \sum_{u,v\in X_{n}}\left\Vert \Psi(u)-\Psi\left(v)\right)\right\Vert _{p}^{p}K_{n}(u,v)\pi_{n}(u)\\
 & \le2\cdot3^{p-1}\left\Vert \psi_{H}(e_{H})-\psi_{H}(h_{n})\right\Vert _{p}^{p}\\
 & +3^{p-1}\sum_{(f,z)\in X_{n}}\sum_{y\in\left[0,m_{n}-1\right]}\left\Vert \Psi((f,z))-\Psi\left((f,z)\cdot\left(\mathbf{e}_{H},y\right)\right)\right\Vert _{p}^{p}\zeta_{m_{n}}(z,y)\pi_{n}((f,z))\\
 & \le3^{p}\left[\left(2\rho_{\psi_{H}}(n)\right)^{p}+\sum_{z,y\in[0,m_{n}-1]}\left|z-y\right|^{p}\zeta_{m_{n}}(z,y)\mathcal{C}_{m_{n}}(z)\right],
\end{align*}
where $\mathcal{C}_{m_{n}}(z)$ denotes the stationary distribution
of $\zeta_{m_{n}}$. The last step used the choice that $\left\Vert \psi_{H}\left(h_{n}\right)-\psi_{H}\left(e_{H}\right)\right\Vert _{p}\le2\rho_{\psi_{H}}(n)$
and the assumption that $\Psi$ is $1$-Lipschitz. From the explicit
formula that defines $\zeta_{m_{n}}$, we have for $p>1$, 
\[
\sum_{z,y\in[0,m_{n}-1]}\left|z-y\right|^{p}\zeta_{m_{n}}(z,y)\mathcal{C}_{m_{n}}(z)\le Cm_{n}^{p-1}.
\]
Set 
\[
m_{n}=\left\lceil \left(2\rho_{\psi_{H}}(n)\right)^{\frac{p}{p-1}}\right\rceil ,
\]
so that the two terms in the $L_{p}$-energy upper bound are comparable,
then 
\begin{equation}
\sum_{u,v\in X_{n}}\left\Vert \Psi(u)-\Psi\left(v)\right)\right\Vert _{p}^{p}K_{n}(u,v)\pi_{n}(u)\le3^{p}(1+C)m_{n}^{p-1}.\label{eq:gradient}
\end{equation}
From the explicit lamplighter structure, one checks that $\mbox{diam}_{G}(X_{n})=\left(2+\left|h_{n}\right|_{H}\right)m_{n}$,
for any $u\in X_{n}$, 
\begin{equation}
\sum_{v\in X_{n}}\mathbf{1}_{\left\{ d_{G}(u,v)\ge\frac{1}{2}\left(\left|h_{n}\right|_{H}+2\right)m_{n}\right\} }\pi_{n}(v)\ge\frac{1}{5}.\label{eq:H-Z-mass}
\end{equation}
Apply (\ref{eq:Poincare-H-Z}),(\ref{eq:gradient}),(\ref{eq:H-Z-mass})
in Lemma \ref{spectral-compression}, we have 
\[
\rho_{\Psi}\left(\frac{\left(2+\left|h_{n}\right|_{H}\right)m_{n}}{2}\right)\le\left(C'm_{n}^{p-1}\cdot m_{n}\log m_{n}\right)^{\frac{1}{p}}.
\]
Since $\left|h_{n}\right|_{H}\ge n$, it follows that there exists
$u,v\in X_{n}$, $d_{G}(u,v)\ge\frac{1}{2}\left(n+1\right)m_{n}$,
\[
\left\Vert \Psi(u)-\Psi(v)\right\Vert _{p}\le C'm_{n}\log^{\frac{1}{p}}m_{n},
\]
Note that by definition of $\tau$, $\tau(n)=\frac{1}{2}n\left(2\rho_{\psi_{H}}(n)\right)^{\frac{p}{p-1}}$,
with the choice of $m_{n}=\left\lceil \left(2\rho_{\psi_{H}}(n)\right)^{\frac{p}{p-1}}\right\rceil $,
the statement follows from rewriting the inequality 
\[
\rho_{\Psi}\left(\frac{1}{2}nm_{n}\right)\le C'm_{n}\log^{\frac{1}{p}}m_{n}.
\]

\end{proof}

The Markov type method can also be applied to this situation, it actually
yields more general results. We presented the proof for $L_{p}$-compression
of $G$ with $p\in(1,2]$ using the Poincaré inequalities because
spectral gap considerations motivate the choice of $\alpha$-stable
walk on the base with $\alpha=1$. Now we explain how to apply Markov
type method. Let $\left(\mathfrak{X},d_{\mathfrak{X}}\right)$ be
a metric space of Markov type $p$, $p>1$, $\Psi:H\wr\mathbb{Z}\to\mathfrak{X}$
be a 1-Lipschitz equivariant embedding. Make the same choice of a
sequence of finite subsets $X_{n}$ and $K_{n}$ as in the proof of
Proposition \ref{rho-Phi} with 
\[
m_{n}=\left\lceil \left(2\rho_{\psi_{H}}(n)\right)^{\frac{p}{p-1}}\right\rceil .
\]
Let $Z_{t}$ denote a Markov chain on $\mathcal{L}_{m_{n}}$ with
transition kernel $\mathfrak{p}_{m_{n}}$, then $\widetilde{Z}_{t}=\sigma_{n}\left(Z_{t}\right)$
is a Markov chain on $X_{n}$ with transition kernel $K_{n}$. Run
the Markov chain $\tilde{Z}_{t}$ up to time $t_{n}=m_{n}\log m_{n}$.
To apply Lemma \ref{Markov-type-compression}, we need a lower bound
for $\mathbf{E}_{\pi_{n}}\left[d_{G}\left(\widetilde{Z}_{t_{n}},\widetilde{Z}_{0}\right)^{p}\right]$.
The bijection $\sigma_{n}:\mathcal{L}_{m_{n}}\to X_{n}$ induces a
metric on $\mathcal{L}_{m_{n}}$ by 
\[
d_{\sigma_{n}}(u,v)=d_{G}\left(\sigma_{n}(u),\sigma_{n}(v)\right).
\]
Direct inspection shows that this metric $d_{\sigma_{n}}$ coincide
with the metric $d_{\mathbf{w}_{n}}$ with $\mathbf{w}_{n}=\left(1,\left|h_{n}\right|_{H}\right)$
introduced in Appendix \ref{stable}. By Lemma \ref{C2}, we have that for
$t_{n}=m_{n}\log m_{n}$, 
\begin{align*}
\mathbf{E}_{\pi_{n}}\left[d_{G}\left(\widetilde{Z}_{t_{n}},\widetilde{Z}_{0}\right)^{p}\right] & =\mathbf{E}_{U_{\alpha,m_{n}}}\left[d_{\mathbf{w}_{n}}\left(Z_{t_{n}},Z_{0}\right)^{p}\right]\\
 & \ge\mathbf{E}_{U_{\alpha,m_{n}}}\left[d_{\mathbf{w}_{n}}\left(Z_{t_{n}},Z_{0}\right)\right]^{p}\\
 & \ge\left(c\left(1+\left|h_{n}\right|_{H}\right)m_{n}\right)^{p}.
\end{align*}
In the proof of Proposition \ref{rho-Phi} we checked that 
\begin{align*}
\mathbf{E}_{\pi_{n}}\left[d_{\mathfrak{X}}\left(\Psi\left(\widetilde{X}_{1}\right),\Psi\left(\widetilde{X}_{0}\right)\right)^{p}\right] & =\sum_{u,v\in X_{n}}d_{\mathfrak{X}}\left(\Psi(u),\Psi(v)\right)^{p}K_{n}(u,v)\pi_{n}(v)\\
 & \le3^{p}\left[\left(2\rho_{\psi_{H}}(n)\right)^{p}+Cm_{n}^{p-1}\right].
\end{align*}
Choose $m_{n}=\left\lceil \left(2\rho_{\psi_{H}}(n)\right)^{\frac{p}{p-1}}\right\rceil $
and plug in these estimates into Lemma \ref{Markov-type-compression},
we have 
\[
\rho_{\Psi}\left(\frac{c}{2}\left(1+\left|h_{n}\right|_{H}\right)m_{n}\right)\le C'M_{p}\left(\mathfrak{X}\right)m_{n}\log^{\frac{1}{p}}m_{n}.
\]
Recall that $\left|h_{n}\right|_{H}\ge n$, therefore 
\[
\rho_{\Psi}\left(\frac{c}{2}nm_{n}\right)\le C'M_{p}\left(\mathfrak{X}\right)m_{n}\log^{\frac{1}{p}}m_{n}.
\]
The result given by the Markov type method is recorded in the following
proposition. By \cite{Naor2006}, $L_{p}$ space with $p>2$ has Markov
type $2$. Proposition \ref{rho-Phi-Markov} applies to $L_{p}$ with
$p>2$ as well, therefore is more general than Proposition \ref{rho-Phi}.

\begin{proposition}\label{rho-Phi-Markov}

There exists a constant $C>0$ such that the following holds. Let
$\left(\mathfrak{X},d_{\mathfrak{X}}\right)$ be a metric space of
Markov type $p$, $p>1$ and $\Psi:H\wr\mathbb{Z}\to\mathfrak{X}$ be
a $1$-Lipschitz equivariant embedding. The compression function $\rho_{\Psi}$
satisfies 
\begin{equation}
\rho_{\Psi}(t)\le C\left(\frac{p}{p-1}\right)^{\frac{1}{p}}M_{p}\left(\mathfrak{X}\right)\left(2\rho_{\psi_{H}}\circ\tau^{-1}(t)\right)^{\frac{p}{p-1}}\log^{\frac{1}{p}}\left(2\rho_{\psi_{H}}\circ\tau^{-1}(t)\right),\label{eq:rho-G-H-3}
\end{equation}
where $\psi_{H}$ the induced embedding of the subgroup $H$ into
$\mathfrak{X}$ as in (\ref{eq:induce-H}), $M_{p}\left(\mathfrak{X}\right)$
is the Markov-type $p$ constant of $\mathfrak{X}$, and the function
$\tau:\mathbb{R}_{+}\to\mathbb{R}_{+}$ is defined as 
\[
\tau(x)=\frac{1}{C}x\left(2\rho_{\psi_{H}}(x)\right)^{\frac{p}{p-1}}.
\]

\end{proposition}

\subsection{$L_{p}$-compression exponent of $H\wr\mathbb{Z}$\label{sub: embedding-wr}}

For the lower bound in the $L_{p}$-compression gap of $H\wr\mathbb{Z}$,
$p\in[1,2]$, we use the embedding constructed in \cite[Theorem 3.3 ]{Naor2008}.
An explicit description of an equivariant embedding is included here
as a warm-up for Section \ref{sec:dihedral-compression}. Given a
good equivariant embedding $\varphi$ of the group $H$ into $L_{p}$
, 
\[
\left\Vert \varphi(h_{1})-\varphi(h_{2})\right\Vert _{p}\ge\rho_{-}\left(d_{H}\left(h_{1},h_{2}\right)\right),
\]
we exhibit an embedding of $G=H\wr\mathbb{Z}$ into $L_{p}$ with
$\varphi$ as building blocks. Recall that for an element $(f,z)\in G$,
the word distance is given by 
\[
\left|(f,z)\right|_{G}=\left|\omega\right|+\sum_{x\in\mathbb{Z}}\left|f(x)\right|_{H},
\]
where $\omega$ is a path of shortest length that starts at $0$,
visits every point $x$ in the support of $f$, and ends at $z$.
We refer to such a path a traveling salesman path for $(f,z)$. The
embedding of $G$ into $L_{p}$ consists of two parts: part I captures
the length of $\omega$ and part II embeds the lamp configurations
$\{f(x)\}_{x\in\mathbb{Z}}$ using the embedding of $H$ into $L_{p}$.

\medskip{}

Part I: this part is essentially the same as an embedding of $\mathbb{Z}_{2}\wr\mathbb{Z}$
into $L_{p}$. Consider the following sequence of functions on $G=H\wr\mathbb{Z}$,
\begin{equation}
\phi_{n}\left(\left(f,z\right)\right)=\mathbf{1}_{\left\{ \mbox{supp}f\subseteq\left[-2^{n},2^{n}\right]\right\} }\max\left\{ 1-\frac{|z|}{2^{n}},0\right\} .\label{eq:base-test-function}
\end{equation}
Note that if $g\in G$ is in the support of $\phi_{n}$, then 
\[
\left|\omega_{g}\right|\le4\cdot2^{n},
\]
where $\omega_{g}$ denote a traveling salesman path for $g$. Also
\[
\left(\mbox{supp}\phi_{n}\right)\left(\mbox{supp}\phi_{n}\right)^{-1}\subseteq\mbox{supp}\phi_{n+1}.
\]
As in \cite{Tessera2011}, given a non-decreasing function $\gamma:\mathbb{N}\to\mathbb{R}_{+}$
with $\gamma(1)=1$ and $\sum_{t=0}^{\infty}\left(\frac{1}{\gamma(t)}\right)^{p}=C_{p}(\gamma)<\infty,$
take a cocycle $b_{\gamma}:G\to L_{p}$ defined as 
\begin{equation}
b_{\gamma}\left(g\right)=\bigoplus_{n=1}^{\infty}\frac{1}{\gamma(n)}\left(\frac{\tau_{g}\phi_{n}-\phi_{n}}{\mathcal{E}_{p}\left(\phi_{n}\right)^{1/p}}\right),\label{eq:b-0}
\end{equation}
where $\tau_{g}$ denotes the right translation of a function $\phi:G\to\mathbb{R}$,
\[
\tau_{g}\phi(u)=\phi\left(ug^{-1}\right).
\]
Because of the normalization in the definition of $b_{\gamma}$, one
readily checks that $b_{\gamma}$ is $\left(2C_{p}(\gamma)\right)^{\frac{1}{p}}$-Lipschitz.

Given $g=(f,x)\in G$, we say a path $\omega$ on $\mathbb{Z}$ is
a traveling salesman path for $g$ if it starts at $0$, visits every
$z\in\mathbb{Z}$ where $f(z)\neq e_{H}$$ $and end at $x$. Let
$\omega_{g}$ be a shortest traveling salesman path for $g$. Suppose
$\left|\omega_{g}\right|>2^{n+3}$, then $g\notin\left(\mbox{supp}\phi_{n}\right)\left(\mbox{supp}\phi_{n}\right)^{-1}$,
it follows that in this case 
\[
\frac{\left\Vert \tau_{g}\phi_{n}-\phi_{n}\right\Vert _{p}^{p}}{\mathcal{E}_{p}\left(\phi_{n}\right)}=\frac{2\left\Vert \phi_{n}\right\Vert _{p}^{p}}{\mathcal{E}_{p}\left(\phi_{n}\right)}\ge\left(2^{n}\right)^{p}.
\]
Putting the components together, we have that 
\[
\left\Vert b_{\gamma}(g)\right\Vert _{p}\ge\frac{\left|\omega_{g}\right|}{8\gamma\left(\log_{2}\left|\omega_{g}\right|\right)}.
\]

Part II: Let $\varphi:H\to L_{p}$ be a 1-Lipschitz equivariant embedding
of $H$ into $L_{p}$. Define a map $\Xi_{\varphi}:G\to L_{p}$ as
\[
\Xi_{\varphi}((f,z))=\bigoplus_{x\in\mathbb{Z}}\varphi\left(f(x)\right).
\]
Since the map $\Xi_{\varphi}$ factors through the projection $G\to\oplus_{x\in\mathbb{Z}}H$
and $\varphi$ is equivariant, it follows that $\Xi_{\varphi}$ is
a cocycle. By construction, $\Xi_{\varphi}$ is $1$-Lipschitz and
\[
\left\Vert \Xi_{\varphi}((f,z))\right\Vert _{p}^{p}=\sum_{x\in\mathbb{Z}}\left\Vert \varphi(f(x))\right\Vert _{p}^{p}.
\]

Combine the two parts, define 
\begin{align}
\Phi:G & \to L_{p}\nonumber \\
\Phi(g) & =b_{\gamma}(g)\bigoplus\Xi_{\varphi}(g).\label{eq:embed-lamplighter}
\end{align}
Then we have $\Psi$ is $(2C_{p}(\gamma))^{1/p}$-Lipschitz and 
\begin{align*}
\left\Vert \Phi(g)\right\Vert _{p}^{p} & =\left\Vert b_{\gamma}(g)\right\Vert _{p}^{p}+\sum_{x\in\mathbb{Z}}\left\Vert \varphi(f(x))\right\Vert _{p}^{p}\\
 & \ge\frac{\left|\omega_{g}\right|^{p}}{8\gamma\left(\log_{2}\left|\omega_{g}\right|\right)}+c\sum_{x\in\mathbb{Z}}\rho_{-}\left(\left|f(x)\right|_{H}\right)^{p}.
\end{align*}

Let $p\in(1,2]$. The bounds (\ref{eq:induce-H}), (\ref{eq:rho-G-H-1})
and the embedding constructed above provide rather detailed information
about the $L_{p}$-compression gap of $G=H\wr\mathbb{Z}$ in terms
of $L_{p}$-compression gap of $H$. We now derive the formula relating
the $L_{p}$-compression exponents of $G$ and $H$ stated in Theorem
\ref{HwrZformula}.

\begin{proof}[Proof of Theorem \ref{HwrZformula}]

We first treat the case $p\in(1,2]$. Let $\Psi:G\to L_{p}$ be a
1-Lipschitz equivariant embedding of $G=H\wr\mathbb{Z}$ into $L_{p}$ and
$\psi_{H}$ its induced embedding $H\hookrightarrow L_{p}$. From
(\ref{eq:rho-G-H-1}), it is always true that $\alpha_{p}^{\#}(H\wr\mathbb{Z})\le\alpha_{p}^{\#}(H).$
Now we apply Proposition \ref{rho-Phi} to prove the upper bound.
By definition of compression exponent, there exists constant $C=C\left(\psi_{H}\right)>0$
and an increasing sequence $n_{i}\in\mathbb{N}$ with $n_{i}\to\infty$
such that 
\[
\rho_{\psi_{H}}(n_{i})\le Cn_{i}^{\alpha_{p}^{\#}(H)}.
\]
Along the sequence $\left\{ n_{i}\right\} $, by Proposition \ref{rho-Phi},
\[
\rho_{\Psi}(n_{i})\le C_{1}(C,p)n_{i}^{\frac{p\alpha}{p-1+\alpha p}}\log^{\frac{1}{p}}\left(n_{i}^{\frac{p\alpha}{p-1+\alpha p}}\right),
\]
where $\alpha=\alpha_{p}^{\#}(H)$ and $C_{1}(C,p)$ is a constant
depending on $C$ and $p$. Therefore 
\[
\alpha_{p}^{\#}(H\wr\mathbb{Z})\le\frac{\alpha_{p}^{\#}(H)}{\alpha_{p}^{\#}(H)+\left(1-\frac{1}{p}\right)}.
\]
We have proved the upper bound.

Note that if $\alpha_{p}^{\#}(H)=0$, then $\alpha_{p}^{\#}(G)=0$.
Thus in the lower bound direction, we consider the case $\alpha_{p}^{\#}(H)>0$.
For any $0<\varepsilon<\alpha_{p}^{\#}(H)$, let $\varphi:H\to L_{p}$
be a 1-Lipschitz equivariant embedding such that 
\[
\rho_{\varphi}(t)\ge(ct)^{\alpha_{p}^{\#}(H)-\varepsilon},
\]
and set $\gamma(n)=\log^{\frac{1+\varepsilon}{p}}(1+n)$. Take the
embedding $\Phi:G\to L_{p}$ as defined in (\ref{eq:embed-lamplighter}),
then 
\begin{align*}
\left\Vert \Phi(g)\right\Vert _{p}^{p} & =\left\Vert b_{\gamma}(g)\right\Vert _{p}^{p}+\sum_{x\in\mathbb{Z}}\left\Vert \varphi(f(x))\right\Vert _{p}^{p}\\
 & \ge\frac{\left|\omega_{g}\right|^{p}}{8\log^{1+\varepsilon}\left(1+\left|\omega_{g}\right|\right)}+\sum_{x\in\mathbb{Z}}\left(c\left|f(x)\right|_{H}\right)^{\left(\alpha_{p}^{\#}(H)-\varepsilon\right)p}.
\end{align*}
Such an embedding $\Phi$ is analyzed in \cite{Naor2008}, from the
proof of \cite[Theorem 3.3 ]{Naor2008}, after sending $\varepsilon\to0$,
we have in the case $\alpha_{p}^{\#}(H)\le\frac{1}{p}$, $\alpha_{p}(\Phi)\ge\alpha_{p}^{\#}(H)$;
in the case $\alpha_{p}^{\#}(H)>\frac{1}{p}$, 
\[
\alpha_{p}^{\#}(\Phi)\ge\frac{\alpha_{p}^{\#}(H)}{\alpha_{p}^{\#}(H)+1-1/p}.
\]
Finally when $p=1$, from $\rho_{\Psi}(t)\le\rho_{\psi_{H}}(t)$ and
the explicit embedding $\Phi:G\to L_{1}$ given a good embedding $\varphi:H\to L_{1}$,
we deduce that 
\[
\alpha_{1}^{\#}(G)=\alpha_{1}^{\#}(H).
\]

\end{proof}

\section{Compression of $\Delta$ with dihedral groups\label{sec:dihedral-compression}}

Throughout this section, $\Delta$ denotes a diagonal
product with $\Gamma_{s}=D_{2l_{s}}$. In \cite[Section 5]{Austin2011},
Austin remarked that compression upper bounds from classical Poincaré
inequalities and Markov type inequalities can be viewed as related
to random walks, and it would be interesting to find examples of finitely
generated amenable groups for which obstructions genuinely unrelated
to inequalities concerning random walks are needed. Austin conjectured
that a group with a sequence of cubes $\left(\mathbb{Z}_{m}^{n},\ell^{\infty}\right)$
embedded would be a candidate for such type of obstructions. In some
sense our construction of diagonal product $\Delta$ with dihedral
groups realizes this idea. Because of the presence of $\ell^{\infty}$-cubes
of growing sizes in $\Delta$, we apply deep results of Mendel-Naor
\cite{Mendel2008} to estimate distortion of these finite subsets.
The main result of this section is the following.

\begin{theorem}\label{cotype2}

Let $\Delta$ be the diagonal product with $\Gamma_{s}=D_{2l_{s}}$
and parameters $( k_{s}) $, set 
\[
\theta:=\limsup_{s\to\infty}\frac{\log l_{s}}{\log k_{s}}.
\]
Assume that $( k_{s}) $ satisfies growth assumption
\ref{k_growth}. Then 
\begin{description}
\item [{(i)}] for $p\in[1,2]$, 
\[
\alpha_{p}^{\ast}\left(\Delta\right)=\max\left\{ \frac{1}{1+\theta},\frac{2}{3}\right\} .
\]

\item [{(ii)}] for $q\in[2,\infty)$, 
\begin{align*}
\alpha_{q}^{\ast}(\Delta) & =\frac{1}{1+\theta}\ \mbox{ if }0\le\theta\le\frac{1}{q}\\
\max\left\{ \frac{\theta+1-\frac{2}{q}}{\left(2-\frac{1}{q}\right)\theta+1-\frac{2}{q}},\frac{2}{3}\right\} \le\alpha_{q}^{\ast}\left(\Delta\right) & \le\frac{2\theta+1-\frac{2}{q}}{3\theta+1-\frac{2}{q}}\ \mbox{ if }\theta>\frac{1}{q}.
\end{align*}

\end{description}
\end{theorem}

When $p\in[1,2]$, the upper bound on $\alpha_{p}^{\ast}\left(\Delta\right)$
is a consequence of the Mendel-Naor metric cotype inequality cited
in Subsection \ref{sub:metric-cotype-inequalities}; in the lower
bound direction we construct an explicit embedding $\Delta\to\ell^{2}$.
The case of $p\in(2,\infty)$ is more involved. The proof is completed
in Subsection \ref{sub:Proof-cotype-2}.

Since $\Delta$ is 3-step solvable, in particular it is amenable,
by \cite[Theorem 1.6]{Naor2011}, 
\[
\alpha_{p}^{\ast}(\Delta)=\alpha_{p}^{\#}(\Delta)\mbox{ for all }p\in[1,\infty).
\]

\subsection{Upper bounds on compression}

We first explain why it's necessary to examine the distortion in a
block of side length $k_{s}$ in $\Delta_{s}$. For notational convenience,
assume that $( k_{s}) $, $( l_{s}) $ are
multiples of $4$. As in Subsection \ref{sub: embedded-metric}, consider
the subset $\Pi_{s}^{k_{s}/2}$ of $\Delta$ defined as in (\ref{eq:product-set}).
Note that $\Pi_{s}^{k_{s}/2}$ is isomorphic to the direct product
of $k_{s}/2$ copies of $D_{2l_{s}}'\simeq\mathbb{Z}_{l_{s}/2}$,
denote by $\vartheta_{s}:\mathbb{Z}_{l_{s}/2}^{k_{s}/2}\to\Pi_{s}^{k_{s}/2}$
the isomorphism. Write elements of $\Pi_{s}^{k_{s}/2}$ as vectors
$u=(u(0),\ldots,u(k_{s}-1)),$ $u(j)\in\mathbb{Z}_{l_{s}/2}$.

Now consider the induced metric on $\Pi_{s}^{k_{s}/2}$ of the word
metric $d_{\Delta}$ on $\Delta$, then by Lemma \ref{embed-product},
we have that for $u\in\Pi_{s}^{k_{s}/2}$, 
\[
\left|u\right|_{\Delta}\simeq_{72}k_{s}\max_{0\le j\le k_{s}/2-1}\left|u(j)\right|_{\mathbb{Z}/l_{s}\mathbb{Z}}.
\]
Therefore the induced metric $\left|\cdot\right|_{\Delta}$ on $\Pi_{s}^{k_{s}/2}$
can be viewed as the $\ell^{\infty}$ metric being dilated by~$k_{s}$.

Let $\left[m\right]_{\infty}^{k}$ denote the set $\left\{ 0,1,\ldots,m\right\} ^{k}$
equipped with the metric induced by $\ell^{\infty}$, 
\[
d_{\infty}(x,x')=\max_{0\le j\le k-1}\left|x_{j}-x_{j}'\right|,\ x=\left(x_{0},\ldots,x_{k-1}\right).
\]
Let $\mathfrak{X}$ be a Banach space of nontrivial type and cotype
$q$, then by \cite[Theorem 1.12]{Mendel2008}, there exists a constant
$c(\mathfrak{X},q)$ depending only on $\mathfrak{X}$ and $q$ such
that the distortion of embedding of $\left[m\right]_{\infty}^{k}$
into $\mathfrak{X}$ satisfies 
\[
c_{\mathfrak{X}}\left(\left[m\right]_{\infty}^{k}\right)\ge c(\mathfrak{X},q)\left(\min\left\{ k^{\frac{1}{q}},m\right\} \right).
\]

We now explain how to apply this distortion lower bound and the Austin
lemma to derive an upper bound on $\alpha_{\mathfrak{X}}^{\ast}(\Delta)$.
Define 
\[
m_{s}=\left\lfloor \min\left\{ k_{s}^{\frac{1}{q}},\frac{1}{4}l_{s}\right\} \right\rfloor ,
\]
and consider $\left\{ 0,1,\ldots,m_{s}\right\} $ as elements in $\mathbb{Z}_{l_{s}}$.
The grid $\left[m_{s}\right]_{\infty}^{k_{s}}$ is embedded in the
group $\Delta$ via the map $\vartheta_{s}$ . Let $\theta:=\limsup_{s\to\infty}\frac{\log l_{s}}{\log k_{s}},$
suppose $\theta\in(0,\infty)$. For any $\epsilon>0$, select a subsequence
$s_{n}$ such that $l_{s_{n}}\ge C_{\epsilon}k_{s_{n}}^{\theta-\epsilon}$
along this subsequence. To apply Lemma \ref{(Lemma-3.1-Austin)} using
the sequence of finite metric spaces $\left(\left[m_{s_{n}}\right]_{\infty}^{k_{s_{n}}},d_{\infty}\right)$,
we check that 
\begin{itemize}
\item $\mbox{diam}_{d_{\infty}}\left(\left[m_{s_{n}}\right]_{\infty}^{k_{s_{n}}}\right)=m_{s_{n}}+1$, 
\item the word distance on $\Delta$ relates to the metric $d_{\infty}$
on $\left[m_{s_{n}}\right]_{\infty}^{k_{s_{n}}}$ by 
\[
d_{\Delta}(\vartheta_{s_n}(u,o),\vartheta_{s_n}(u',0))\simeq_{72}k_{s_n}d_{\infty}(u,u').
\]
and 
\[
k_{s_n}\le C'_{\epsilon}\left(\mbox{diam}_{d_{\infty}}\left(\left[m_{s_{n}}\right]_{\infty}^{k_{s_{n}}}\right)\right)^{\min\left\{ q,\frac{1}{\theta-\epsilon}\right\} }.
\]

\item the distortion of $\left[m_{s_{n}}\right]_{\infty}^{k_{s_{n}}}$ satisfies
\[
c_{\mathfrak{X}}\left(\left[m_{s_{n}}\right]_{\infty}^{k_{s_{n}}}\right)\ge c(\mathfrak{X},q)m_{s_{n}}\ge\frac{c\left(\mathfrak{X},q\right)}{2}\mbox{diam}_{d_{\infty}}\left(\left[m_{s_{n}}\right]_{\infty}^{k_{s_{n}}}\right).
\]

\end{itemize}
Then by Lemma \ref{(Lemma-3.1-Austin)}, we have that if $\mathfrak{X}$
is a Banach space of nontrivial type and cotype $q$, then 
\begin{eqnarray}
\alpha_{\mathfrak{X}}^{\ast}(\Delta)\le1-\frac{1}{1+\min\left\{ q,\frac{1}{\theta}\right\} }=\max\left\{ \frac{1}{1+\theta},\frac{q}{1+q}\right\} .\label{Austin-upper-dihedral}
\end{eqnarray}

Note that the $\mathfrak{X}$-distortion of the grid $\left[m_{s_{n}}\right]_{\infty}^{k_{s_{n}}}$
selected is comparable to its $d_{\infty}$-diameter. Unlike the case
with $\{\Gamma_{s}\}$ taken to be expanders, the size of $m_{s}$
is constrained by $k_{s}^{\frac{1}{q}}$. We will see in Subsection
\ref{sub:compression-upper-dihderal} that when $q_{\mathfrak{X}}>2$
and $\mathfrak{X}$ is of Markov type $p<2$, this upper bound on
$\alpha_{\mathfrak{X}}^{\ast}(\Delta)$ can be improved. In the rest
of this subsection we give a more detailed description of distorted
elements and derive an upper bound for the compression functions.

\subsubsection{A first upper bound using metric cotype}

\begin{proposition}\label{distortion-one-block}

Let $\vartheta_{s}$ and $\Pi_{s}^{k_{s}/2}$ be introduced as above,
suppose $\mathfrak{X}$ is a Banach space of nontrivial type and cotype
$q$.  Then there
exists a constant $C=C(\mathfrak{X},q)$ such that for any $1$-Lipschitz
equivariant embedding $\varphi:\Delta\to\mathfrak{X}$,

\[
\rho_{\varphi}\left(\frac{1}{8}k_{s}\min\left\{ l_{s},k_{s}^{\frac{1}{q}}\right\} \right)\le C(\mathfrak{X},q)k_{s}.
\]

\end{proposition}

This proposition improves on (\ref{Austin-upper-dihedral}) as it
applies to functions. Its proof will also be useful to \textquotedbl{}locate\textquotedbl{}
the obstructions and derive afterwards a better upper bound.

\begin{proof}

Take $m\le l_{s}/4$ and consider $0,\ldots,m+1$ as elements of $\mathbb{Z}_{l_{s}/2}$.
By \cite[Lemma 6.12]{Mendel2008}, for each $\epsilon>0$, $\mathbb{\mathbb{Z}}_{2m}^{n}$
equipped with $\ell^{\infty}$ metric embeds with distortion $1+6\epsilon$
into $\left[m+1\right]_{\infty}^{(\left\lceil 1/\epsilon\right\rceil +1)n}$.
Take $\epsilon=1$ and fix a $1$-Lipschitz embedding 
\[
\psi_{s}:\mathbb{\mathbb{Z}}_{2m}^{k_{s}/2}\to\left[m+1\right]_{\infty}^{k_{s}}
\]
with distortion $c(\psi_{s})\le8$. Let $\tilde{d}_{\Delta}$ be the
induced metric by $d_{\Delta}$ on $\mathbb{\mathbb{Z}}_{2m}^{k_{s}/2}$
\[
\tilde{d}_{\Delta}(u,v)=d_{\Delta}\left(\vartheta_{s}\circ\psi_{s}(u),\vartheta_{s}\circ\psi_{s}(v)\right),
\]
and $\tilde{\varphi}$ be the induced embedding 
\[
\tilde{\varphi}=\varphi\circ\vartheta_{s}\circ\psi_{s}:\mathbb{\mathbb{Z}}_{2m}^{k_{s}/2}\to\mathfrak{X}.
\]
Let $U_{s}$ be the uniform measure on $\mathbb{\mathbb{Z}}_{2m}^{k_{s}/2}$
and $\sigma_{s}$ be the uniform measure on $\left\{ -1,0,1\right\} ^{k_{s}/2}$.
Let $\left\{ \mathbf{e}_{j}\right\} _{j=0}^{k_{s}/2-1}$ be the standard
basis of $\mathbb{R}^{k_{s}/2}$. Since $\mathfrak{X}$ is a $K$-convex
Banach space of cotype $q$, then by the metric cotype inequality
in \cite[Theorem 4.2]{Mendel2008} (reviewed in subsection \ref{sub:metric-cotype-inequalities}),
\begin{multline}
\sum_{j=0}^{k_{s}/2-1}\sum_{u\in\mathbb{\mathbb{Z}}_{2m}^{k_{s}/2}}\left\Vert \tilde{\varphi}(u)-\tilde{\varphi}\left(u+m\mathbf{e}_{j}\right)\right\Vert _{\mathfrak{X}}^{q}U_{s}(u)\\
\le\Omega^{p}\sum_{u\in\mathbb{\mathbb{Z}}_{2m}^{k_{s}/2}}\sum_{\varepsilon\in\{-1,0,1\}^{k_{s}/2}}\left\Vert \tilde{\varphi}(u)-\tilde{\varphi}\left(u+\sum_{j=0}^{k_{s}/2-1}\varepsilon_{j}\mathbf{e}_{j}\right)\right\Vert _{\mathfrak{X}}^{q}\sigma(\varepsilon)U_{s}(u),\label{eq:metric-cotype-dihedral}
\end{multline}
where 
\[
\Omega=5\max\left\{ C\left(\mathfrak{X},q\right)m,\left(\frac{k_{s}}{2}\right)^{\frac{1}{q}}\right\} .
\]
and $C\left(\mathfrak{X},q\right)$ is a constant that only depends
on the cotype constant and $K_{q}$-convexity constant of $\mathfrak{X}$.

Since $\varphi$ is $1$-Lipschitz, by Lemma \ref{embed-product}
we have 
\[
\left\Vert \tilde{\varphi}(u,0)-\tilde{\varphi}\left(u+\sum_{j=0}^{k_{s}/2-1}\varepsilon_{j}\mathbf{e}_{j},0\right)\right\Vert _{\mathfrak{X}}\le\tilde{d}_{\Delta}\left(u+\sum_{i=0}^{k_{s}/2-1}\varepsilon_{i}\mathbf{e}_{i},u\right)\le72k_{s}.
\]
Plug in (\ref{eq:metric-cotype-dihedral}), 
\[
\sum_{j=0}^{k_{s}/2-1}\sum_{u\in\left(\mathbb{Z}/2m\mathbb{Z}\right)^{k_{s}/2}}\left\Vert \tilde{\varphi}(u)-\tilde{\varphi}\left(u+m\mathbf{e}_{j}\right)\right\Vert _{\mathfrak{X}}^{q}U_{s}(u)\le\left(5\max\left\{ C(\mathfrak{X})m,k_{s}^{\frac{1}{q}}\right\} \right)^{q}\left(72k_{s}\right)^{q}.
\]
It follows there exists $u\in\left(\mathbb{Z}/2m\mathbb{Z}\right)^{k_{s}/2}$
and $j_{0}\in\left\{ 0,\ldots,k_{s}/2-1\right\} $ such that 
\[
\left\Vert \tilde{\varphi}(u)-\tilde{\varphi}\left(u+m\mathbf{e}_{j_{0}}\right)\right\Vert _{\mathfrak{X}}\le360\max\left\{ C(\mathfrak{X})mk_{s}^{1-\frac{1}{q}},k_{s}\right\} .
\]

To obtain the upper bound on $\rho_{\varphi}$, choose $m=\left\lfloor \frac{1}{4}\min\left\{ l_{s},k_{s}^{\frac{1}{q}}\right\} \right\rfloor $.
By Lemma \ref{embed-product}, 
\[
\tilde{d}_{\Delta}\left(u,u+m\mathbf{e}_{j_{0}}\right)\ge\frac{1}{2}mk_{s},
\]
it follows that 
\[
\rho_{\varphi}\left(\frac{1}{2}k_{s}m\right)\le360C(\mathfrak{X})k_{s}.
\]

\end{proof}

\begin{remark}\label{L1-cotype}

Since $L_{1}$ has trivial type, embeddings $\varphi:\Delta\hookrightarrow L_{1}$
is not covered by the lemma. However it is true that there exists
constant $C>0$ such that for $\varphi:\Delta\to L_{1}$ a 1-Lipschitz
embedding, then 
\[
\rho_{\varphi}\left(\frac{1}{C}k_{s}\min\left\{ l_{s},k_{s}^{\frac{1}{2}}\right\} \right)\le Ck_{s}.
\]
To see this, as pointed out in \cite[Remark 7.5]{Mendel2008}, since
$L_{1}$ equipped with the metric $\sqrt{\left\Vert x-y\right\Vert _{1}}$
is isomorphic to a subset of Hilbert space, \cite[Theorem 4.2]{Mendel2008}
applied to Hilbert space gives 
\begin{multline*}
\sum_{j=0}^{k_{s}/2-1}\sum_{u\in u\in\left(\mathbb{Z}/2m\mathbb{Z}\right)^{k_{s}/2}}\left\Vert \tilde{\varphi}(u)-\tilde{\varphi}\left(u+m\mathbf{e}_{j}\right)\right\Vert _{L_{1}}U_{s}(u)\\
\le C^{2}\max\left\{ m^{2},k_{s}\right\} \sum_{u\in u\in\left(\mathbb{Z}/2m\mathbb{Z}\right)^{k_{s}/2}}\sum_{\varepsilon\in\{-1,0,1\}^{k_{s}}}\left\Vert \tilde{\varphi}(u)-\tilde{\varphi}\left(u+\sum_{j=0}^{k_{s}-1}\varepsilon_{j}\mathbf{e}_{j}\right)\right\Vert _{L_{1}}\sigma(\varepsilon)U_{s}(u),
\end{multline*}
which implies the stated bound.

\end{remark}

\subsection{A more refined upper bound when $\mathfrak{X}$ has cotype $>2$
\label{sub:compression-upper-dihderal}}

In this section we develop an improvement of the compression upper
bound in Proposition \ref{distortion-one-block}. The idea is that
when $q_{\mathfrak{X}}>2$, we can further apply the Markov type method
to find obstruction in lamplighter graphs with elements in blocks
of side length $k_{s}$ considered as lamp configurations. The argument
is similar to the one for $H\wr\mathbb{Z}$ in Subsection \ref{sub:-compression-wr-Z}.

Let $\varphi:\Delta\to\mathfrak{X}$ be an equivariant 1-Lipschitz
embedding of the group $\Delta$ into $\mathfrak{X}$, assume that
$\mathfrak{X}$ is a Banach space of cotype $q$ and nontrivial type.
From the proof of Proposition \ref{distortion-one-block}, we have
that there exists an element $h_{0}^{s}=\vartheta_{s}\left(\frac{l_{s}}{2}\mathbf{e}_{j_{0}},0\right)$
satisfying $\left|h_{0}^{s}\right|_{\Delta}\ge\frac{1}{4}k_{s}l_{s}$
\begin{equation}
\left\Vert \varphi\left(h_{0}^{s}\right)-\varphi(e_{\Delta})\right\Vert _{\mathfrak{X}}\le C(\mathfrak{X},q)\max\left\{ l_{s}k_{s}^{1-\frac{1}{q}},k_{s}\right\} .\label{eq:v-j}
\end{equation}
The element $h_{0}^{s}$ is in the zero section of $\Delta_{s}$ and
it is supported at site $j_{0}$ in the interval $[0,k_{s}-1)$. Let
$h_{j}^{s}$ denote the translation of $h_{0}^{s}$ to the block $[jk_{s},(j+1)k_{s})$,
\[
h_{j}^{s}(x)=h_{0}^{s}(x-jk_{s}).
\]
Consider the following subset (not a subgroup) in $\Delta_{s}$ 
\[
L_{m}^{s}=\left\{ \left(f_{s},z\right):\ \begin{array}{c}
f_{s}\upharpoonright_{[jk_{s},(j+1)k_{s})}\in\left\{ \mathbf{0},h_{j}^{s}\upharpoonright_{[jk_{s},(j+1)k_{s})}\right\} ,\ 0\le j\le m-1\\
\mbox{supp}f_{s}\subseteq[0,mk_{s})\\
\ z\in\{0,k_{s},\ldots,(m-1)k_{s}\}.
\end{array}\right\} .
\]
Again $L_{m}^{s}$ can be naturally embedded in $\Delta$, we identify
it with its embedded image and consider $L_{m}^{s}$ as a subset of
$\Delta$. The subset $L_{m}^{s}$ has the structure of a lamplighter
graph over a segment, the lamp configuration is divided into blocks
of side length $k_{s}$, in each block it's either identically zero
or coincide with $h_{j}^{s}$. As explained in Subsection \ref{sub:-compression-wr-Z},
we can apply the Markov type method to derive a lower bound for distortion
of $L_{m}^{s}$.

\begin{proposition}\label{dihedral-compression-upper-2}

Let $\mathfrak{X}$ be a Banach space of cotype $q$ and Markov type
$p$ such that $2<q<\infty,\ p>1$. Then there exists a constant $C>0$
such that for any $1$-Lipschitz equivariant embedding $\varphi:\Delta\to\mathfrak{X}$, 
\begin{itemize}
\item if $l_{s}\le k_{s}^{\frac{1}{q}}$, 
\[
\rho_{\varphi}\left(\frac{1}{4}k_{s}l_{s}\right)\le C(\mathfrak{X},q)k_{s},
\]

\item if $l_{s}>k_{s}^{\frac{1}{q}}$, 
\end{itemize}
\[
\rho_{\varphi}\left(\frac{1}{2C}k_{s}l_{s}\left(l_{s}k_{s}^{-\frac{1}{q}}\right)^{\frac{p}{p-1}}\right)\le C\left(\frac{p}{p-1}\right)^{\frac{1}{p}}M_{p}\left(\mathfrak{X}\right)C(\mathfrak{X},q)k_{s}\left(l_{s}k_{s}^{-\frac{1}{q}}\right)^{\frac{p}{p-1}}\log^{\frac{1}{p}}\left(l_{s}k_{s}^{-\frac{1}{q}}\right).
\]

\end{proposition}

\begin{proof}

The case where $l_{s}\le k_{s}^{\frac{1}{q}}$ is covered by Proposition
\ref{distortion-one-block}.

In the case where $l_{s}>k_{s}^{\frac{1}{q}}$, we apply the Markov
type method. Let $L_{m}^{s}$ be defined as above, there is a natural
bijection $\sigma_{m}^{s}$ between the lamplighter graph $\mathcal{L}_{m}$
over the segment $\left\{ 0,\ldots,m-1\right\} $ and $L_{m}^{s}$,
explicitly, 
\begin{align*}
\sigma_{m}^{s}:\mathcal{L}_{m} & \to L_{m}^{s},\\
\sigma_{m}^{s}(u,x) & =\left(f^{u},mx\right)\ \mbox{where }f^{u}\upharpoonright_{[jk_{s},(j+1)k_{s})}=\left(h_{j}^{s}\right)^{u(j)}.
\end{align*}
Let $\mathfrak{p}_{m}$ be the lamplighter kernel on $\mathcal{L}_{m}$
defined in Appendix C. Under the bijection $\sigma_{m}^{s}$, let
$K_{m}^{s}=\mathfrak{p}_{m}\circ\left(\sigma_{m}^{s}\right)^{-1}$
be the corresponding Markov kernel on $L_{m}^{s}$. Now we run the
Markov chain with transition kernel $K_{m}^{s}$ up to time $t=m\log m$.
Lemma \ref{Markov-type-compression} implies 
\begin{equation}
\rho_{\varphi}\left(\left(\frac{1}{2}\mathbf{E}_{\pi}d_{\Delta}\left(Z_{t},Z_{0}\right){}^{p}\right)^{\frac{1}{p}}\right)\le\left(2M_{p}^{p}(\mathfrak{X})t\mbox{diam}_{\Delta}\left(L_{m}^{s}\right)^{p}\frac{\mathbf{E}_{\pi}d_{\mathfrak{X}}\left(\varphi\left(Z_{1}\right),f\left(Z_{0}\right)\right){}^{p}}{\mathbf{E}_{\pi}d_{\Delta}\left(Z_{t},Z_{0}\right){}^{p}}\right)^{\frac{1}{p}},\label{eq:dihedral-Markov}
\end{equation}
where $\pi$ is the stationary distribution of $K_{m}^{s}$ and $Z_{t}$
is a stationary Markov chain on $L_{m}^{s}$ with transition kernel
$K_{m}^{s}$.

Now we estimate the quantities that appear in the inequality. Let
$d_{\Delta}$ be the metric on $\mathcal{L}_{m}$ induced by word
metric on $\Delta$, 
\[
d_{\Delta}(u,v)=d_{\Delta}\left(\sigma_{m}^{s}(u),\sigma_{m}^{s}(v)\right),
\]
then direct inspection shows that there exists an absolute constant
$C>0$ such that 
\[
\frac{1}{C}d_{\mathbf{w}}(u,v)\le d_{\Delta}(u,v)\le Cd_{\mathbf{w}}(u,v),\ \mathbf{w}=\left(k_{s},k_{s}l_{s}\right).
\]
It follows that $\mbox{diam}_{\Delta}\left(L_{m}^{s}\right)\le2C\left(k_{s}+k_{s}l_{s}\right)m$.
By Lemma \ref{cauchy-speed}, we have that for $t=m\log m$, 
\[
\mathbf{E}_{\pi_{n}}\left[d_{\Delta}\left(Z_{t},Z_{0}\right)^{p}\right]\ge\left(c\left(k_{s}+k_{s}l_{s}\right)m\right)^{p}.
\]
For the other term, using Lemma \ref{C2}  and (\ref{eq:v-j}) when $l_{s}^{q}>k_{s}$, we have
\begin{align*}
\mathbf{E}_{\pi}\left[d_{\mathfrak{X}}\left(\varphi\left(Z_{1}\right),\Psi\left(Z_{0}\right)\right)^{p}\right] & =\sum_{u,v\in X_{n}}d_{\mathfrak{X}}\left(\varphi(u),\varphi(v)\right)^{p}K_{n}(u,v)\pi_{n}(v)\\
 & \le3^{p}\left[Ck_{s}^{p}m^{p-1}+\left(C(\mathfrak{X})l_{s}k_{s}^{1-\frac{1}{q}}\right)^{p}\right].
\end{align*}
With the choice 
\[
m=\left\lceil \left(l_{s}k_{s}^{-\frac{1}{q}}\right)^{\frac{p}{p-1}}\right\rceil ,
\]
(\ref{eq:dihedral-Markov}) implies 
\[
\rho_{\Psi}\left(\frac{1}{2C}k_{s}l_{s}m\right)\le C\left(\frac{p}{p-1}\right)^{\frac{1}{p}}M_{p}\left(\mathfrak{X}\right)C(\mathfrak{X},q)k_{s}\left(l_{s}k_{s}^{-\frac{1}{q}}\right)^{\frac{p}{p-1}}\log^{\frac{1}{p}}\left(l_{s}k_{s}^{-\frac{1}{q}}\right).
\]

\end{proof}

The upper bound on compression function immediately yields the following
upper bound on compression exponent.

\begin{corollary}\label{cotype-q}

Let $\mathfrak{X}$ be a Banach space of cotype $q$ and Markov type
$p$ with $2<q<\infty$ and $p>1$. Let $\Delta$ be the diagonal
product with $\Gamma_{s}=D_{2l_{s}}$, 
\[
\theta:=\limsup_{s\to\infty}\frac{\log l_{s}}{\log k_{s}}.
\]
Then 
\[
\alpha_{\mathfrak{X}}^{\#}\left(\Delta\right)\le\begin{cases}
\frac{1}{1+\theta} & \mbox{ if }\theta\le\frac{1}{q}\\
\frac{p\theta+p-1-\frac{p}{q}}{(2p-1)\theta+p-1-\frac{p}{q}} & \mbox{ if }\theta>\frac{1}{q}.
\end{cases}
\]

\end{corollary}

\subsection{Explicit construction of embedding into $L_{q}$, $q\ge2$\label{sub:explicit-hilbert}}

We construct an embedding of $\Delta$ into $L_{q}$ in two parts,
similar to the map in Subsection \ref{sub: embedding-wr}.

Recall the $1$-cocycle constructed in Subsection \ref{sub:Basic-test-functions}.
Fix a choice of increasing function $\gamma:\mathbb{N}\to\mathbb{R}_{+}$
such that $\gamma(1)=1$, $C(\gamma)=\sum_{n=1}^{\infty}\gamma(n)^{-2}<\infty$.
Let $b_{\gamma}:\Delta\hookrightarrow L_{2}$ be the 1-cocycle defined
by (\ref{eq:cocycle}) with $p=2$ using the basic test functions.

A standard embedding of finite dihedral groups into Euclidean space
is the following. The (unlabelled) Cayley graph of $D_{2l}$ is the
same as a cycle of length $2l$. One can simply embed it as vertices
of a regular $2l$-gon in the plane. For each element $\gamma\in D_{2l}$
, fix a word of minimal length in $a$ and $b$ such that the word
represents $\gamma$ and starts with letter $a$. Let $k(\gamma)$
be the length of such a chosen word, $k_{a}(\gamma)$ ($k_{b}(\gamma)$
resp.) be the number of occurrence of $a$ ($b$ resp.) in this word.
Take $\theta_{l}:D_{2l}\to\mathbb{R}^{2}$ as 
\[
\theta_{l}(\gamma)=\frac{1}{2\sin\left(\pi/2l\right)}\left(\cos\left(\frac{\pi k(\gamma)}{l}\right),\sin\left(\frac{\pi k(\gamma)}{l}\right)\right).
\]
It is clear that $\theta_{l}$ is $1$-Lipschitz and equivariant. We
also consider maps to vertices of $l$-gons. Let $\theta_{l}^{(a)}:D_{2l}\to\mathbb{R}^{2}$
be the map given by 
\[
\theta_{l}^{(a)}(\gamma)=\frac{1}{2\sin\left(\pi/l\right)}\left(\cos\left(\frac{2\pi k_{a}(\gamma)}{l}\right),\sin\left(\frac{2\pi k_{a}(\gamma)}{l}\right)\right).
\]
The map $\theta_{l}^{(b)}:D_{2l}\to\mathbb{R}^{2}$ is defined in
the same way with $k_{a}(\gamma)$ replaced by $k_{b}(\gamma)$. Since
$\left|k_{a}(\gamma)-k_{b}(\gamma)\right|\le1$ for any element $\gamma\in D_{2l}$,
by definition of $\theta_{l}^{(a)}$, $\theta_{l}^{(b)}$ we have
\[
\left\Vert \theta_{l}^{(a)}(\gamma)-\theta_{l}^{(b)}(\gamma)\right\Vert _{2}\le1.
\]

Recall the classical fact that $\ell^{2}$ embeds isometrically in
$L_{q}$ for all $q\ge1$, see \cite[Proposition 6.4.2]{Albiac2006}.
To construct embeddings to $L_{q}$, for each $q>2$ fix an isometric
embedding $i_{q}:\ell^{2}\to L_{q}$, and set $b_{\gamma,q}=i_{q}\circ b_{\gamma}$,
similarly, $\theta_{l,q}^{(a)}=i_{q}\circ\theta_{l}^{(a)},$ $\theta_{l,q}^{(b)}=i_{q}\circ\theta_{l}^{(b)}$.

Direct inspection of $\theta_{l}^{(a)},\theta_{l}^{(b)}$ shows the
following.

\begin{fact}\label{dihedral-l}For all $\gamma,\gamma'\in D_{2l_{s}}$,

\[
\left\Vert \theta_{l_{s},2}^{(a)}(\gamma\gamma')-\theta_{l_{s},2}^{(a)}(\gamma)\right\Vert _{2}=\left\Vert \theta_{l_{s},2}^{(a)}(\gamma')-\theta_{l_{s},2}^{(a)}\left(e_{D_{2l_{s}}}\right)\right\Vert _{2}.
\]
The same equality holds with $a$ replaced by $b$.

\end{fact}

Now we introduce a weight function. $ $Let $w_{s}:\mathbb{Z}\to[0,1]$
be the function defined as 
\[
w_{s}(y)=\begin{cases}
\frac{1}{2} & \mbox{ for } y\le-\frac{1}{2}k_{s}\mbox{ or }y\ge\frac{3}{2}k_{s},\\
\frac{\left|y\right|}{k_{s}} & \mbox{ for } -\frac{1}{2}k_{s}<y<k_{s},\\
1-\frac{y-k_{s}}{k_{s}} & \mbox{ for }  k_{s}\le y<\frac{3}{2}k_{s}.
\end{cases}
\]
It is the simplest piecewise linear function taking value $1/2$ outside
$[-k_{s}/2,3k_{s}/2]$, $0$ in $0$ and $1$ in $k_{s}$. For $x\in\mathbb{Z}$,
write $\tau_{x}w$ for the translation of $w_{s}$ by $x$, 
\[
\tau_{x}w_{s}(y)=w_{s}(y-x).
\]
Define the map $\Phi_{s,q}:\Delta_{s}\to L_{q}=\bigoplus_{y\in\mathbb{Z}}(L_{q})_{y}$
by setting for each $y\in\mathbb{Z}$, 
\begin{equation}
\left[\Phi_{s,q}\left(f_{s},z\right)\right](y)=k_{s}^{1-\frac{1}{q}}\left((\tau_{z}w_{s})(y)\theta_{l_{s},q}^{(a)}\left(f_{s}(y)\right)+(1-(\tau_{z}w_{s})(y))\theta_{l_{s},q}^{(b)}(f_{s}(y))\right),\label{eq:Phi_s-q}
\end{equation}
In words, at each site $y$, the image of $f_{s}(y)$ is a linear
combination of $\theta_{l_{s},q}^{(a)}$ and $\theta_{l_{s},q}^{(b)}$
with the weights depending on the relative position between $y$ and
the cursor.

Finally define an embedding $\Phi_{\gamma,q}:\Delta\to L_{q}$ by
\begin{equation}
\Phi_{\gamma,q}\left(\left(f_{s}\right),z\right)=\left(\bigoplus_{s=0}^{\infty}\left(\frac{1}{\gamma(s)}\Phi_{s,q}\left(f_{s},z\right)\right)\right)\bigoplus b_{\gamma,q}\left(\left(f_{s}\right),z\right),\label{eq:dihedral-q}
\end{equation}
where $\bigoplus$ is direct sum in $L_{q}$.

We now check some basic properties of the map $\Phi_{\gamma,q}$.

\begin{lemma}\label{dihedral-lipschitz}

Let $\gamma:\mathbb{N}\to\mathbb{R}_{+}$ be a function such that
$\gamma(1)=1$, $C(\gamma)=\sum_{n=1}^{\infty}\gamma(n)^{-2}<\infty$.
The map $\Phi_{\gamma,q}:\Delta\to L^{q}$ , $q\ge2$ defined in (\ref{eq:dihedral-q})
$C$-Lipschitz with $C$ only depending on $C(\gamma)$.

\end{lemma}

\begin{proof}

It suffices to check that for any $u=\left(\left(f_{s}\right),z\right)\in\Delta$
and $s\in\left\{ \tau,\alpha,\beta\right\} $ a generator, the increment
$\left\Vert \Phi_{\gamma,p}(us)-\Phi_{\gamma,p}(u)\right\Vert _{q}$
is bounded by $C$.

For the generator $\alpha$, $\left(\left(f_{s}\right),z\right)\alpha=\left(\left(f_{s}'\right),z\right)$
where $f_{s}'(y)=f_{s}'(y)$ for all $y\neq z$ and $f_{s}'(z)=f_{s}(z)a(s).$
Recall that by definitions, the weight function $w_{s}$ satisfies
$\tau_{z}w_{s}(z)=0$, and the map $\theta_{l_{s},q}^{(b)}$ satisfies
$\theta_{l_{s},q}^{(b)}(\gamma a(s))=\theta_{l_{s},q}^{(b)}(\gamma)$
for all $\gamma\in D_{2l_{s}}$. Then by (\ref{eq:Phi_s-q}), 
\[
\Phi_{s,q}\left(\left(f_{s},z\right)\alpha\right)=\Phi_{s,q}\left(\left(f_{s},z\right)\right).
\]
Therefore in the embedding (\ref{eq:dihedral-q}), 
\[
\left\Vert \Phi_{\gamma,q}\left(u\alpha\right)-\Phi_{\gamma,q}(u)\right\Vert _{q}=\left\Vert b_{\gamma}\left(u\alpha\right)-b_{\gamma}(u)\right\Vert _{2}\le\sqrt{2C(\gamma)},
\]
the last inequality uses the fact that the $1-$cocycle $b_{\gamma}$
is Lipschitz, see subsection \ref{sub:Basic-test-functions}. Similarly,
since $\tau_{z}w_{s}\left(z+k_{s}\right)=1$ and $\theta_{l_{s},q}^{(a)}(\gamma b(s))=\theta_{l_{s},q}^{(a)}(\gamma)$
for all $\gamma\in D_{2l_{s}}$, we have $\left\Vert \Phi_{\gamma,q}\left(u\beta\right)-\Phi_{\gamma,q}(u)\right\Vert _{q}=\left\Vert b_{\gamma}(u\beta)-b_{\gamma}(u)\right\Vert _{2}$
as well. For the generator $\tau$, 
\[
\Phi_{s,q}(u\tau)=\left(k_{s}^{1-\frac{1}{q}}\left(w_{s}(y-z-1)\theta_{l_{s},q}^{(a)}\left(f_{s}(y)\right)+(1-w_{s}(y-z-1))\theta_{l_{s},q}^{(b)}\left(f_{s}(y)\right)\right)\right)_{y\in\mathbb{Z}}.
\]
Then 
\[
\left\Vert \Phi_{s,q}(u\tau)-\Phi_{s,q}(u)\right\Vert _{q}^{q}=\sum_{y\in\mathbb{Z}}k_{s}^{q-1}\left|w_{s}(y-z-1)-w_{s}(y-z)\right|^{q}\left\Vert \theta_{l_{s},q}^{(a)}\left(f_{s}(y)\right)-\theta_{l_{s},q}^{(b)}\left(f_{s}(y)\right)\right\Vert _{q}^{q}.
\]
Recall that $\left\Vert \theta_{l_{s}}^{(a)}(\gamma)-\theta_{l_{s}}^{(b)}(\gamma)\right\Vert _{2}\le1$
for all $\gamma\in D_{2l_{s}}$, $w_{s}(y-z)\neq w_{s}(y-z-1)$ only
if $y-z\in\left[-\frac{k_{s}}{2},\frac{3k_{s}}{2}\right]$, and in
this interval $\left|w_{s}(y-z)-w_{s}(y-z-1)\right|=\frac{1}{k_{s}}$.
Therefore 
\[
\left\Vert \Phi_{s,q}(u\tau)-\Phi_{s,q}(u)\right\Vert _{q}^{q}\le2k_{s}k_{s}^{q-1}\left(\frac{1}{k_{s}}\right)^{q}=2.
\]
Summing up in the embedding (\ref{eq:dihedral-q}), 
\begin{align*}
\left\Vert \Phi_{\gamma,q}\left(u\tau\right)-\Phi_{\gamma,q}(u)\right\Vert _{q}^{q} & =\sum_{s}\frac{1}{\gamma(s)^{q}}\left\Vert \Phi_{s}(u\tau)-\Phi_{s}(u)\right\Vert _{q}^{q}+\left\Vert b_{\gamma}\left(u\tau\right)-b_{\gamma}(u)\right\Vert _{2}^{q}\\
 & \le2C(\gamma)+(2C(\gamma))^{\frac{q}{2}}.
\end{align*}

\end{proof}

Because of the presence of the weight function $w_{s}$, the embedding
$\Phi_{\gamma,q}:\Delta\to L_{q}$ fails to be equivariant. But the
increment $\left\Vert \Phi_{\gamma,q}\left(uv\right)-\Phi_{\gamma,q}(u)\right\Vert _{q}$
is actually comparable to $\left\Vert \Phi_{\gamma,q}\left(v\right)\right\Vert _{q}.$

\begin{lemma}\label{dihedral-increment}

There exists a constant $c>0$ depending only on $C(\gamma)$ such
that for $q\ge2$,

\[
\left\Vert \Phi_{\gamma}\left(uv\right)-\Phi_{\gamma}(u)\right\Vert _{q}\ge c\left\Vert \Phi_{\gamma}(v)\right\Vert _{q}.
\]

\end{lemma}

\begin{proof}

By the formula (\ref{eq:Phi_s-q}) that defines $\Phi_{s,q}$, $\Phi_{s,q}(uv)$
is 
\[
\left(k_{s}^{1-\frac{1}{q}}\left(w_{s}(y-z-z')\theta_{l_{s},q}^{(a)}\left(f_{s}(y)f_{s}'(y-z)\right)+(1-w_{s}(y-z-z'))\theta_{l_{s},q}^{(b)}\left(f_{s}(y)f_{s}'(y-z)\right)\right)\right)_{y\in\mathbb{Z}}.
\]
Then by the triangle inequality and Fact \ref{dihedral-l}, 
\[
\left\Vert \Phi_{s,q}\left(uv\right)-\Phi_{s,q}(u)\right\Vert _{q}\ge\left\Vert \Phi_{s,q}(v)\right\Vert _{q}-\left(k_{s}^{q-1}\sum_{y\in\mathbb{Z}}\left|w_{s}(y-z)-w_{s}(y-z-z')\right|^{q}\left\Vert \theta_{l_{s}}^{(a)}(f_{s}(y))-\theta_{l_{s}}^{(b)}(f_{s}(y))\right\Vert _{2}^{q}\right)^{\frac{1}{q}}.
\]
Since $\left\Vert \theta_{l_{s}}^{(a)}(\gamma)-\theta_{l_{s}}^{(b)}(\gamma)\right\Vert _{2}\le1$
for all $\gamma$, and 
\[
\sum_{y\in\mathbb{Z}}\left|w_{s}(y-z)-w_{s}(y-z-z')\right|^{q}\le\left(\frac{1}{k_{s}}\right)^{q}\min\left\{ \left|z'\right|,2k_{s}\right\} ,
\]
we have $\left\Vert \Phi_{s,q}\left(uv\right)-\Phi_{s,q}(u)\right\Vert _{2}\ge\left\Vert \Phi_{s,q}(v)\right\Vert _{2}-2^{\frac{1}{q}}$.
Using $(a-b)_{+}^{q}\ge\left(a/2\right)^{q}-b^{q}$ for $a,b\ge0$,
\begin{align*}
\left\Vert \Phi_{\gamma,q}\left(uv\right)-\Phi_{\gamma,q}(u)\right\Vert _{q}^{q} & \ge\sum_{s}\frac{1}{\gamma(s)^{q}}\left(\left\Vert \Phi_{s,q}(v)\right\Vert _{q}-2^{\frac{1}{q}}\right)_{+}^{q}+\left\Vert b_{\gamma,q}(v)\right\Vert _{q}^{q}\\
 & \ge\frac{1}{4C(\gamma)}\left(\sum_{s}\frac{1}{\gamma(s)^{q}}\left(\frac{1}{2^{q}}\left\Vert \Phi_{s,q}(v)\right\Vert _{q}^{q}-2\right)\right)+\left\Vert b_{\gamma,q}(v)\right\Vert _{q}^{q}\\
 & \ge\frac{1}{2^{2+q}C(\gamma)}\sum_{s}\frac{1}{\gamma(s)^{q}}\left\Vert \Phi_{s,q}(v)\right\Vert _{q}^{q}+\left\Vert b_{\gamma,q}(v)\right\Vert _{q}^{q}-\frac{1}{2}\\
 & \ge\frac{1}{2^{2+q}C(\gamma)}\left\Vert \Phi_{\gamma,q}(v)\right\Vert _{q}^{q}.
\end{align*}

\end{proof}

\subsection{Compression exponent $\alpha_{p}^{\ast}(\Delta)$ \label{sub:Proof-cotype-2}}

In this subsection we estimate the $L_{p}$-compression exponent of
$\Delta$. Recall that for $p\in[1,2]$, $L_{p}$ has Markov type
$p$ and cotype $2$; and for $p\in(2,\infty)$, $L_{p}$ has Markov
type $2$ and cotype $p$, see \cite{Naor2006} and references therein.

\begin{proof}[Proof of Theorem \ref{cotype2}]

\textbf{upper bound in (i)}

For the upper bound on $\alpha_{p}^{\ast}(\Delta)$, $p\in[1,2]$,
when $\theta=0$ the bound is trivially true. Assume now $\theta\in(0,\infty)$,
for any $\varepsilon>0$ sufficiently small, take a subsequence $\left\{ s_{i}\right\} _{i\in\mathbb{N}}$
such that 
\[
\frac{\log l_{s_{i}}}{\log k_{s_{i}}}>\theta-\varepsilon.
\]
Since the cotype of $L_{p}$ is $2$ for $p\in[1,2]$, along this
subsequence of $\left(k_{s_{i}},l_{s_{i}}\right)$ apply Proposition
\ref{distortion-one-block} and the Remark after it for $L_{1}$,
we obtain the upper bound after sending $\varepsilon\to0$. The argument
also extends to $\theta=\infty$.

\textbf{upper bound in (ii)}

The upper bound on $\alpha_{q}^{\ast}(\Delta)=\alpha_{q}^{\#}(\Delta)$
is covered by Corollary \ref{cotype-q}.

\textbf{lower bound in (ii)}

In the case $\theta\le\frac{1}{q}$, $q\ge2$, simply take the $1$-cocyle
$b_{\gamma}:\Delta\to \ell^{2}$ defined in (\ref{eq:cocycle}) using
the basic test functions (\ref{eq:test-function}), with $\gamma(n)=n^{-\frac{1}{2}-\epsilon}$.
Then Lemma \ref{compression-expander-lower} implies that the compression
exponent of $\Delta$ satisfies 
\[
\alpha_{q}^{\ast}(\Delta)\ge\alpha_{2}^{\ast}(\Delta)\ge\liminf_{s\to\infty}\frac{\log k_{s}}{\log\left(k_{s}l_{s}\right)}\ge\frac{1}{1+\theta}.
\]
Now we focus on the case $\theta>\frac{1}{q}$. Consider the explicit
embedding $\Phi_{\gamma,q}:\Delta\to L_{q}$ defined in (\ref{eq:dihedral-q})
with $\gamma$ taken to be $\gamma(n)=(1+n)^{-\frac{1+\epsilon}{2}}$.
By Lemma \ref{dihedral-lipschitz}, $\Phi_{\gamma,q}$ is Lipschitz.
By Lemma \ref{dihedral-increment}, there exists a constant $c=c\left(C(\gamma)\right)$
such that for any $u,v\in\Delta$, 
\[
\left\Vert \Phi_{\gamma,q}(uv)-\Phi_{\gamma,q}(u)\right\Vert _{q}\ge c\left\Vert \Phi_{\gamma,q}(v)\right\Vert _{q}.
\]
Let $v=\left(\left(f_{s}\right),z\right)$ be an element of $\Delta$.
At each site $y$, from definition of $\theta_{l_{s},q}^{(a)}$ ,
$\theta_{l_{s},q}^{(b)}$, 
\begin{align*}
\left\Vert (\tau_{z}w_{s})(y)\theta_{l_{s},q}^{(a)}\left(f_{s}(y)\right)+(1-(\tau_{z}w_{s})(y))\theta_{l_{s},q}^{(b)}(f_{s}(y))\right\Vert _{q} & \ge\frac{1}{\sin(\pi/l_{s})}\sin\left(\frac{\pi}{2l_{s}}\left(\left|f_{s}(y)\right|_{D_{2l_{s}}}-1\right)_{+}\right)\\
 & \ge\frac{1}{\pi}\left(\left|f_{s}(y)\right|_{D_{2l_{s}}}-1\right)_{+}.
\end{align*}
From the explicit formula (\ref{eq:Phi_s-q}) that defines $\Phi_{s,q}$,
\[
\left\Vert \Phi_{s,q}(f_{s},z)\right\Vert _{q}^{q}\ge\sum_{j\in\mathbb{Z}}\left(\frac{1}{\pi}k_{s}^{1-\frac{1}{q}}\max_{y\in I_{j}^{s}}\left(\left|f_{s}(y)\right|_{D_{2l_{s}}}-1\right)_{+}\right)^{q}.
\]
In what follows we write $R=\left|\mbox{Range}(v)\right|$, and $\tilde{f}_{s}(y)=\left(\left|f_{s}(y)\right|_{D_{2l_{s}}}-1\right)_{+}$.
We have 
\begin{align*}
\left\Vert \Phi_{\gamma,q}(v)\right\Vert _{q}^{q} & \ge\sum_{s\le s_{0}(v)}\frac{1}{\gamma(s)^{q}}\sum_{j\in\mathbb{Z}}\left(\frac{1}{\pi}k_{s}^{1-\frac{1}{q}}\max_{y\in I_{j}^{s}}\tilde{f}_{s}(y)\right)^{q}+\left\Vert b_{\gamma}\left(v\right)\right\Vert _{2}^{q}\\
 & \ge\sum_{s\le s_{0}(v)}\frac{1}{\gamma(s)^{q}}\sum_{j\in\mathbb{Z}}\left(\frac{1}{\pi}k_{s}^{1-\frac{1}{q}}\max_{y\in I_{j}^{s}}\tilde{f}_{s}(y)\right)^{q}+\left(\frac{cR}{\gamma\circ\log_{2}(R)}\right)^{q}.
\end{align*}
The last step used Lemma \ref{compression-expander-lower} and the
fact that since $k_{s+1}\ge2k_{s}$ for all $s$, $s_{0}(v)\le\log_{2}R$.
Note that in the factor $\Delta_{s}$, the number of intervals $I_{j}^{s}$
with $\max_{x\in I_{j}^{s}}\tilde{f}_{s}(y)\neq0$ is bounded from
above by $2R/k_{s}$. Therefore by Hölder inequality,$ $ 
\[
\sum_{j\in\mathbb{Z}}\max_{y\in I_{j}^{s}}\tilde{f}_{s}(y)^{q}\ge\left(\frac{2R}{k_{s}}\right)^{1-q}\left(\sum_{j\in\mathbb{Z}}\max_{y\in I_{j}^{s}}\tilde{f}_{s}(y)\right)^{q}=2Rk_{s}^{-1}\ell(s)^{q},\ \mbox{ where }\ell(s)=\frac{k_{s}\sum_{j\in\mathbb{Z}}\max_{y\in I_{j}^{s}}\tilde{f}_{s}(y)}{2R}.
\]
Consider the following three cases. 
\begin{itemize}
\item if $0\le\ell(s)\le R^{\frac{1}{q}}$, then 
\[
\sum_{j\in\mathbb{Z}}\left(\frac{1}{2}k_{s}^{1-\frac{1}{q}}\max_{y\in I_{j}^{s}}\tilde{f}_{s}(y)\right)^{q}+\left(cR\right)^{q}\ge(cR)^{q}\ge\left(c/3\right)^{q}(R+2R\ell(s))^{\frac{q^{2}}{q+1}};
\]

\item if $R^{\frac{1}{q}}\le\ell(s)\le R^{1-\frac{1}{q}}k_{s}^{\frac{2}{q}-1}$,
then it's necessary that $R\le k_{s}$. From the metric description
in Subsection \ref{sub:one-copy}, $R\le k_{s}$ implies $\sum_{j\in\mathbb{Z}}\max_{y\in I_{j}^{s}}\tilde{f}_{s}(y)\le1$.
Then in this case 
\[
\sum_{j\in\mathbb{Z}}\left(\frac{1}{2}k_{s}^{1-\frac{1}{q}}\max_{y\in I_{j}^{s}}\tilde{f}_{s}(y)\right)^{q}+\left(cR\right)^{q}\ge(cR)^{q}\ge\left(c/3\right)^{q}(R+2R\ell(s))^{q};
\]

\item if $\ell(s)>R^{1-\frac{1}{q}}k_{s}^{\frac{2}{q}-1}$, from the second
item we only need to consider $R>k_{s}$. Recall that $\ell(s)\le l_{s}\le C_{\varepsilon}k_{s}^{\theta+\varepsilon}$.
It follows that $R\le\left(C_{\varepsilon}k_{s}^{\theta+\varepsilon+1-\frac{2}{q}}\right)^{\frac{q}{q-1}}$.
Then in this case 
\begin{align*}
\sum_{j\in\mathbb{Z}}\left(\frac{1}{2}k_{s}^{1-\frac{1}{q}}\max_{y\in I_{j}^{s}}\tilde{f}_{s}(y)\right)^{q}+\left(cR\right)^{q} & \ge\frac{1}{2^{q-1}}k_{s}^{q-1}Rk_{s}^{-1}\ell(s)^{q}\ge c_{\varepsilon}'\left(R\ell(s)\right)^{q\gamma(\varepsilon)}\\
\mbox{where }\gamma(\varepsilon) & =\frac{\theta+\varepsilon+1-\frac{2}{q}}{\left(2-\frac{1}{q}\right)\left(\theta+\varepsilon\right)+1-\frac{2}{q}}.
\end{align*}
In the last inequality, we used $\ell(s)>R^{1-\frac{1}{q}}k_{s}^{\frac{2}{q}-1}$
and $R\le\left(C_{\varepsilon}k_{s}^{\theta+\varepsilon+1-\frac{2}{q}}\right)^{\frac{q}{q-1}}$. 
\end{itemize}
Note that since $\theta>\frac{1}{q}$, when $\varepsilon$ is sufficiently
small, $\gamma(\varepsilon)<\frac{q}{q+1}$, thus the worst case is
represented by the third item, we have 
\[
\left\Vert \Phi_{\gamma,q}(v)\right\Vert _{q}^{q}\ge\sum_{s\le s_{0}(v)}\frac{c_{\varepsilon}'}{\gamma(s)^{q}}\left(R+R\ell(s)\right)^{q\gamma(\varepsilon)}.
\]
Recall that by the metric estimate Proposition \ref{metric}, 
\[
\left|v\right|_{\Delta}\le500\left(R+\sum_{s\le s_{0}(v)}k_{s}\sum_{j\in\mathbb{Z}}\max_{y\in I_{j}^{s}}\tilde{f}_{s}(y)\right)=500\left(R+\sum_{s\le s_{0}(v)}2R\ell(s)\right),
\]
and $s_{0}(v)\leq\log_{2}R$ by Assumption \ref{k_growth}. Combine
with Lemma \ref{dihedral-increment}, we have for any $u\in\Delta$,
\begin{align*}
\left\Vert \Phi_{\gamma,q}(u)-\Phi_{\gamma,q}(uv)\right\Vert _{2} & \ge\frac{c'}{\log_{2}^{\frac{1}{2}+\epsilon}(R)}\max_{s\le s_{0}(v)}\left(R+R\ell(s)\right)^{\gamma(\varepsilon)}\\
 & \ge\frac{c'}{\log_{2}^{\frac{1}{2}+\epsilon}(R)}\left(\frac{|v|_{\Delta}}{1000\log_{2}(R)}\right)^{\gamma(\varepsilon)}.
\end{align*}
where $c'>0$ is a constant depending on $\theta$ and $\epsilon$.
Sending $\varepsilon\to0$, we conclude that when $\theta\in\left(\frac{1}{q},\infty\right]$,
\[
\alpha_{q}^{\ast}(\Delta)\ge\frac{\theta+1-\frac{2}{q}}{\left(2-\frac{1}{q}\right)\theta+1-\frac{2}{q}}.
\]
Note that the formula is simplified when $q=2$, namely $\alpha_{2}^{\ast}(\Delta)\ge\frac{2}{3}$.
Combine with the fact that $\alpha_{q}^{\ast}(\Delta)\ge\alpha_{2}^{\ast}(\Delta)$
we obtain the statement.

\textbf{lower bound in (i)}

Since $\ell^{2}$ embeds isometrically in all $L_{p}$, $p\ge1$, it
follows that 
\[
\alpha_{p}^{\ast}(\Delta)\ge\alpha_{2}^{\ast}(\Delta)\ge\max\left\{ \frac{1}{1+\theta},\frac{2}{3}\right\} .
\]
This completes the proof of Theorem \ref{cotype2}.

\end{proof}

\begin{example}\label{ex-dihedral-upper}[Proof of Theorem \ref{dihedral-exponent}]

Consider the construction of $\Delta$ with $\Gamma_{s}=D_{2l_{s}}$,
the parameters $\left\{ k_{s}\right\} $, $\left\{ l_{s}\right\} $
are chosen to be $k_{s}=2^{\beta s}$, $l_{s}=2^{\iota s}$ with $\beta>1,\iota\ge0$.
Then $\theta=\iota/\beta$.

For $p\in[1,2]$, Theorem \ref{cotype2} implies that 
\[
\alpha_{p}^{\ast}(\Delta)=\max\left\{ \frac{1}{1+\theta},\frac{2}{3}\right\} ,
\]
which can take any value in $\left[\frac{2}{3},1\right]$.

For $q>2$, $\theta>\frac{1}{q}$ the upper and lower bound in Theorem
\ref{cotype2} don't match up. But in some region of parameters we
can still compare it to Hilbert compression exponent. For $\theta\in\left(\frac{1}{q},1\right)$,
we have 
\[
\alpha_{q}^{\ast}(\Delta)\ge\frac{\theta+1-\frac{2}{q}}{\left(2-\frac{1}{q}\right)\theta+1-\frac{2}{q}}>\alpha_{2}^{\ast}(\Delta)=\max\left\{ \frac{1}{1+\theta},\frac{2}{3}\right\} .
\]
In particular, we can take $\theta=\frac{1}{2}$, then the corresponding
diagonal product $\Delta_{1}$ satisfies 
\[
\alpha_{q}^{\ast}(\Delta_{1})\ge\frac{3q-4}{4q-5}>\alpha_{2}^{\ast}(\Delta_{1})=\frac{2}{3}.
\]

\end{example}

\section{Discussion and some open problems}

\label{problems}

The groups of Theorem \ref{main-expanders} are diagonal products of lamplighter groups. In particular they contain many torsion elements and admit many quotients.

\begin{problem}
Find solutions to the inverse problems for speed, entropy, return probability, isoperimetric profile or compression in the class of torsion-free groups or in the class of simple groups.
\end{problem}

In Theorem \ref{possible-speed-linear}, we imposed the regularity
assumption on $\varrho$ that $\varrho(n)/\sqrt{n}$ is non-decreasing.
This is not always satisfied, it is possible to construct examples
of groups where the speed function is roughly constant over certain
long time intervals. The following question asks if this regularity
assumption can be dropped.

\begin{problem}

Let $\varrho:[1,\infty)\to[1,\infty)$ be a non-decreasing subadditive function satisfying
$\varrho(x)\ge\sqrt{x}$ for all $x$. Is there a group $G$ and a
symmetric probability measure $\mu$ of finite support on $G$ such
that 
\[
L_{\mu}(n)\simeq\varrho(n)?
\]

\end{problem}

Proposition \ref{entropy-speed-function} only partially answers the
question what joint behavior of speed and entropy can occur. Further,
the question of possible joint behavior of speed, entropy and return
probability, even restricting to group of exponential volume growth,
is wide open, see \cite[Question 6]{Amir2015}. Solving the following
problem would be a step in this direction.

\begin{problem}

Find an open set $\mathcal{O}$ in $(0,1)^{3}$ such that for any
point $(\alpha,\beta,\gamma)\in\mathcal{O}$, there exists a finitely
generated group $G$ and a symmetric probability measure of finite
support on $G$, such that $(\alpha,\beta,\gamma)$ is the exponent
of $\left(L_{\mu}(n),H_{\mu}(n),-\log\mu^{(2n)}(e)\right)$.

\end{problem}

In \cite{Gournay2014}, Gournay showed that if a simple
random walk on $G$ satisfies that for some $C>0$, $\gamma\in(0,1)$,
$\phi^{(2n)}(e)\ge\exp\left(-Cn^{\gamma}\right)$, and an off-diagonal
decay bound 
\begin{equation}
\ \phi^{(2n)}(g)\le C\phi^{(2n)}(e)\exp\left(-\frac{C|g|^{2}}{n}\right)\mbox{ for all }g\in G,\ n\in\mathbb{N},\label{eq: (OD)}
\end{equation}
then the Hilbert compression exponent $\alpha_{2}^{\#}(G)\ge1-\gamma$.
The off-diagonal decay assumption (\ref{eq: (OD)}) is difficult to
check in general. We illustrate a family of examples where it is not
valid. In Table \ref{table}, take the diagonal product $\Delta$
with parameters $\left(k_{s}=2^{2s}\right)$ and $\left\{ \Gamma_{s}\right\} $
expanders with $\mbox{diam}(\Gamma_{s})\simeq2^{2\theta}$. When $\theta>1$,
we have 
\[
-\log q^{(2n)}(e)\simeq n^{\frac{1+\theta}{3+\theta}},\ \alpha_{2}^{\ast}(\Delta)=\frac{1}{1+\theta}<1-\frac{1+\theta}{3+\theta}.
\]
Therefore we deduce from \cite[Theorem 1.4]{Gournay2014} that in this case simple
random walk on $\Delta$ fails the off-diagonal upper bound (\ref{eq: (OD)}).
On the other hand, we have the strict inequality
\[
\alpha_{2}^{\ast}(\Delta)=\frac{1}{1+\theta}>\frac{1}{2+\theta}=\frac{1-\gamma}{1+\gamma},
\]
showing the gap is far from the lower bound of \cite[Theorem 1.1]{Gournay2014}. A better understanding of the relation between return probability and
compression remains open.

\begin{problem}

Let $G$ be a finitely generated infinite group such that for some
$\gamma\in\left(0,1\right)$, simple random walk satisfies $\phi^{(2n)}(e)\ge\exp\left(-Cn^{\gamma}\right)$
for some $C>0$. Find the sharp lower bound for $\alpha_{2}^{\#}(G)$ in terms
of $\gamma$ and explicit examples where the bound is sharp.

\end{problem}

In Theorem \ref{HwrZformula} we give an explicit formula that relates
equivariant $L_{p}$-compression exponents of $H\wr\mathbb{Z}$ and
$H$ when $p\in[1,2]$. Less is known about compression exponent of
embeddings into $L_{p}$ with $p>2$. In particular, the following
problem is open.

\begin{problem}

For $p>2$, is there an explicit formula that connects equivariant
compression exponents $\alpha_{p}^{\#}(H\wr\mathbb{Z})$ and $\alpha_{p}^{\#}(H)$?

\end{problem}

The problem of determining $\alpha_{p}^{\#}(\Delta)$ for $\Delta$
constructed with dihedral groups as discussed in Section \ref{sec:dihedral-compression}
is related to the previous problem.

\begin{problem}

Determine the equivariant compression exponent $\alpha_{p}^{\#}(\Delta)$,
$p>2$, where $\Delta$ is the diagonal product constructed with dihedral
groups $\left\{ D_{2l_{s}}\right\} $.

\end{problem}

\appendix

\section{Some auxiliary facts about excursions}

\label{app: excursions}

In this section we recall some classical facts about local time and
excursions of standard simple random walk on $\mathbb{Z}$. Let $\left\{ S_{k}\right\} $
denote the standard simple random walk on $\mathbb{Z}$, starting
at $S_{0}=0$. Let $L(x,n)$ denote the local time of the random walk
at site $x$, 
\[
L(x,n)=\#\left\{ k:\ 0<k\le n,\ S_{k}=x\right\} .
\]
The distribution of $L(x,n)$ is known explicitly (\cite[Theorem 9.4 ]{Revesz2013}),
for $x=0,1,2,\ldots$ 
\[
P_{0}\left(L(x,n)=m\right)=\begin{cases}
\frac{1}{2^{n-m+1}}\left(\begin{array}{c}
n-m+1\\
(n+x)/2
\end{array}\right) & \mbox{if }n+x\mbox{ if even},\\
\frac{1}{2^{n-m}}\left(\begin{array}{c}
n-m\\
(n+x-1)/2
\end{array}\right) & \mbox{if }n+x\mbox{ if odd}.
\end{cases}.
\]
Let $\rho_{0}=0$ and $\rho_{m}=\min\left\{ j>\rho_{m-1}:\ S_{j}=0\right\} $.
Then $\rho_{1}$, $\rho_{2}-\rho_{1}$,... record the time duration
of the excursions from $0$, they form a sequence of i.i.d random
variables with distribution 
\[
P_{0}\left(\rho_{1}>2n\right)=P_{0}(S_{2n}=0)\sim\frac{1}{\sqrt{4\pi n}}.
\]
The chance that an excursion from $0$ crosses $k$ is (\cite[Theorem 9.6]{Revesz2013})
\[
P_{0}\left(\max_{0\le i\le\rho_{1}}S_{i}\ge k\right)=\frac{1}{2k},\ k=1,2,\ldots.
\]
For $x\in\mathbb{Z}$, let $\rho_{1}(x)=\min\left\{ j>0:\ S_{j}=x\right\} $
denote the first time the random walk visits $x$, $T(k,x,n)$ be
the number of excursions away from $x$ that cross $x-k$ and are
completed before time $n$. We need estimates on the moments $E_{0}\left[T(k,x,n)^{q}\right]$
, $0<q\le1$.

\begin{lemma}\label{traverse-upper}

There exists constant $C>0$ such that for all $k,n\in\mathbb{N}$,
$x\in\mathbb{Z}$, 
\[
E_{0}\left[T(k,x,n)\right]\le\frac{C\sqrt{n}}{k}\exp\left(-\frac{x^{2}}{2n}\right).
\]

\end{lemma}

\begin{proof}

Let $\rho_{m}(x)=\min\left\{ j>\rho_{m-1}(x):\ S_{j}=x\right\} $
be the $m$-th time the random walk visits $x$. Then 
\begin{align*}
E_{0}\left[T(k,x,n)\right] & \le\sum_{j\ge0}E_{0}\left[T\left(k,x,\rho_{2^{j}+1}(x)\right)\mathbf{1}_{\left\{ \rho_{2^{j}}(x)\le n<\rho_{2^{j}+1}(x)\right\} }\right]\\
 & \le\sum_{j\ge0}E_{0}\left[T\left(k,x,\rho_{2^{j}+1}(x)\right)\mathbf{1}_{\left\{ \rho_{2^{j}}(x)\le n\right\} }\right].
\end{align*}
Conditioned on the event $\left\{ \rho_{2^{j}}(x)\le n\right\} $,
the random variable $T\left(k,x,\rho_{2^{j}+1}(x)\right)$ is stochastically
dominated by a binomial random variable with parameter $2^{j+1}$
and $\frac{1}{2k}$. Therefore 
\begin{align*}
\sum_{j\ge0}E_{0}\left[T\left(k,x,\rho_{2^{j}+1}(x)\right)\mathbf{1}_{\left\{ \rho_{2^{j}}(x)\le n\right\} }\right] & \le\sum_{j\ge0}\frac{2^{j+1}}{2k}P_{0}\left(\rho_{2^{j}}(x)\le n\right)\\
 & \le\sum_{j\ge0}\frac{2^{j+1}}{2k}P_{0}\left(\rho_{1}(x)\le n\right)P_{0}\left(L(n,0)\ge2^{j}\right)\\
 & \le\frac{1}{2k}P_{0}(\rho_{1}(x)\le n)E_{0}\left(4L(n,0)\right).
\end{align*}
Plug in the estimates 
\begin{align*}
P_{0}(\rho_{1}(x) & \le n)=P_{0}\left(\max_{0\le t\le n}S_{t}\ge|x|\right)\le2\exp\left(-\frac{x^{2}}{2n}\right),\\
E_{0}(L(n,0)) & =\sum_{t=0}^{n}P_{0}(S_{t}=0)\le Cn^{\frac{1}{2}},
\end{align*}
we obtain the statement.

\end{proof}

\begin{lemma}\label{traverse-lower}

There exists a constant $c>0$ such that for all $x\in\mathbb{Z}$,
$k,n\in\mathbb{N}$ satisfying $k\le c^{2}n^{\frac{1}{2}}$, 
\[
P_{0}\left(T(k,x,n)\ge\frac{c\sqrt{n}}{4k}\right)\ge\frac{1}{2}P_{0}\left(L(x,n/2)\ge1\right).
\]

\end{lemma}

\begin{proof}

Conditioned on the event $\{\rho_{1}(x)=t\}$, $0\le t<n$, the distribution
of $T(k,x,n)$ is the same as $T(k,0,n-t)$. Therefore for any $m>0$,
\[
P_{0}\left(T(k,x,n)\ge m\right)\ge P_{0}\left(T(k,0,n/2)\ge m\right)P_{0}\left(L(x,n/2)\ge1\right).
\]
Now we show that there exists a constant $c>0$ such that for $k\le c^{2}n^{\frac{1}{2}}$,
\[
P_{0}\left(T(k,0,n/2)\ge\frac{c\sqrt{n}}{4k}\right)\ge\frac{1}{2}.
\]
Note that 
\[
\left\{ T(k,0,n/2)\ge m\right\} \supset\left\{ \rho_{\left\lfloor cn^{\frac{1}{2}}\right\rfloor }(0)\le\frac{n}{2}\right\} \cap\left\{ T\left(k,0,\rho_{\left\lfloor cn^{\frac{1}{2}}\right\rfloor }\right)\ge m\right\} .
\]
For ease of notation, in what follows write $l=\left\lfloor cn^{\frac{1}{2}}\right\rfloor $.
Since $\rho_{l}(0)$ is sum of $l$ i.i.d. random variables with distribution
\[
P_{0}\left(\rho_{1}>2t\right)=P_{0}(S_{2t}=0)\sim\frac{1}{\sqrt{4\pi t}},
\]
by classical theory of sum of i.i.d. $\alpha$-stable variables (here
$\alpha=\frac{1}{2}$), there exists constant $C=C(\alpha)$ such
that 
\[
P_{0}\left(\rho_{l}(0)\ge t\right)\le C\frac{l}{t^{\frac{1}{2}}}.
\]
Therefore 
\[
P_{0}\left(\rho_{l}(0)\ge\frac{n}{2}\right)\le\sqrt{2}Cc.
\]
For the other term, $ $$T\left(k,0,\rho_{l}\right)$ is binomial
with parameters $l$ and $1/2k$, therefore by Bernstein inequality
(see for example {[}Theorem 2.3 Revesz{]}), 
\[
P_{0}\left(T\left(k,0,\rho_{l}\right)\le\frac{cn^{\frac{1}{2}}}{4k}\right)\le2\exp\left(-\frac{l}{4^{3}k}\right).
\]
Then 
\[
P_{0}\left(T(k,0,n/2)\ge\frac{c\sqrt{n}}{4k}\right)\ge1-\sqrt{2}Cc-2\exp\left(-\frac{cn^{\frac{1}{2}}}{4^{3}k}\right).
\]
Choose $c$ sufficiently small so that $\sqrt{2}Cc\le1/4$, $\exp\left(-\frac{1}{4^{3}c}\right)\le1/8$,
we obtain the statement.

\end{proof}

\section{Approximation of functions}

\label{approximation}

For $p_{1},p_{2}\geq0$, consider the following space $\mathcal{C}_{p_{1},p_{2}}$
of continuous functions between $x^{p_{1}}$ and $x^{p_{2}}$, 
\[
\mathcal{C}_{p_{1},p_{2}}=\left\{ f:[1,\infty)\rightarrow[1,\infty):\begin{array}{l}
f\textrm{ is continuous, }f(1)=1\\
\frac{f(x)}{x^{p_{1}}}\textrm{ is non-decreasing}\\
\frac{x^{p_{2}}}{f(x)}\textrm{ is non-increasing}
\end{array}\right\} .
\]
Equivalently, $\mathcal{C}_{p_{1},p_{2}}$ is the set of functions
with $f(1)=1$ satisfying 
\[
a^{p_{1}}f(x)\le f(ax)\le a^{p_{2}}f(x)\mbox{ for all }a,x\ge1.
\]

We aim to approximate functions in $\mathcal{C}_{p_{1},p_{2}}$ up
to multiplicative constants by piecewise extremal functions.

Given two unbounded sequences $(k_{s}),(l_{s})$ of real numbers,
possibly finite with last value infinity, define the function 
\begin{eqnarray}
\tilde{f}(x)=\tilde{f}_{(k_{s}),(l_{s})}(x)=\left\{ \begin{array}{ll}
l_{s} & \textrm{ for }k_{s}l_{s}\leq x\leq k_{s+1}l_{s},\\
\frac{x}{k_{s+1}} & \textrm{ for }k_{s+1}l_{s}\leq x\leq k_{s+1}l_{s+1},
\end{array}\right.\label{tilde}
\end{eqnarray}
Similarly define 
\begin{eqnarray}
\bar{f}(x)=\bar{f}_{(k_{s}),(l_{s})}(x)=l_{s}+\frac{x}{k_{s+1}}\mbox{ for }k_{s}l_{s}\leq x\leq k_{s+1}l_{s+1}.\label{bar}
\end{eqnarray}

\begin{lemma}\label{lemma-approximation}For any $f$ in $\mathcal{C}_{0,1}$
and for any $m_{0}>1$, there exists two sequences $(k_{s}),(l_{s})$
of real numbers, possibly finite with last value infinity, such that
$k_{s+1}\geq m_{0}k_{s}$ and $l_{s+1}\geq m_{0}l_{s}$ for all $s$,
the functions defined above satisfy 
\[
\tilde{f}(x)\simeq_{m_{0}}f(x)\mbox{ and }\bar{f}(x)\simeq_{2m_{0}}f(x).
\]

Moreover if for some $\alpha>\alpha_{0}>0$ the function $\frac{f(x)}{\log^{\alpha}(x)}$
is non-decreasing, it is possible to find such functions with sequences
$(k_{s}),(l_{s})$ satisfying $\log k_{s}\leq l_{s}^{\frac{1}{\alpha_{0}}}$
for all $s$. \end{lemma}

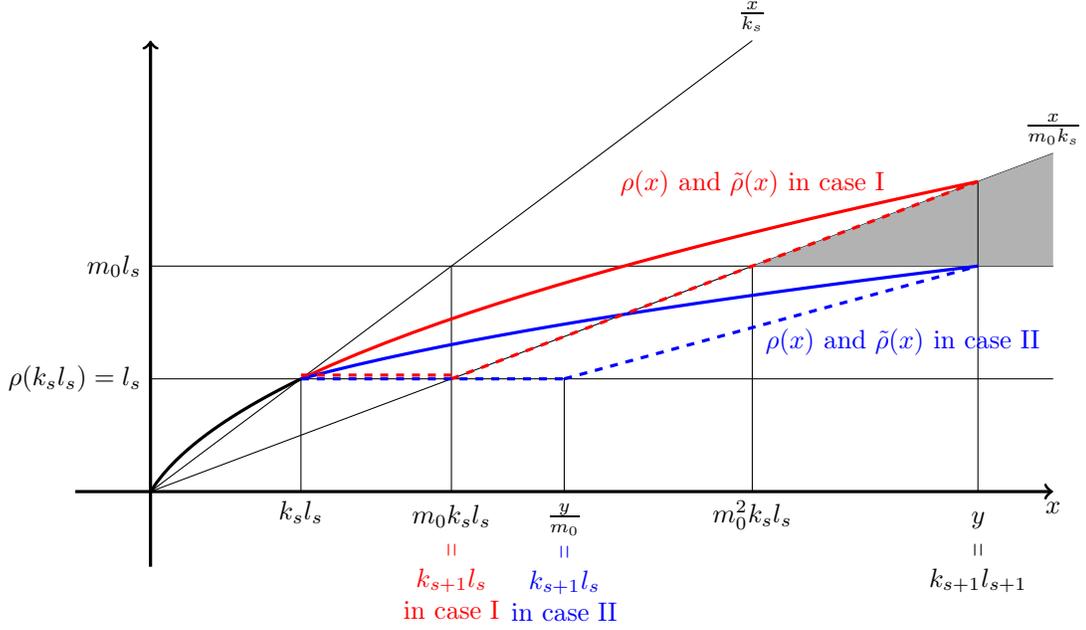
\begin{figure}
\begin{centering}
\begin{tikzpicture}

\draw[very thick,->] (-1,0) -- (12,0) node[below]{$x$};
\draw[very thick,->] (0,-1) -- (0,6);
\draw (0,1.5) node[left]{$\rho(k_sl_s)=l_s$} -- (12,1.5);
\draw (0,3) node[left]{$m_0l_s$} -- (12,3);
\draw (0,0) -- (8,6) node[above] {$\frac{x}{k_s}$};
\draw (0,0) -- (12,4.5) node[above]{$\frac{x}{m_0k_s}$};
\fill[color=gray!60] (8,3)--(12,3)--(12,4.5)-- cycle;
\draw (2,0)node[below]{$k_sl_s$} --(2,1.5);
\draw (4,0)node[below]{$\begin{array}{c} m_0k_sl_s \\ \textcolor{red}{\shortparallel} \\ \textcolor{red}{k_{s+1}l_s} \\ \textcolor{red}{\mbox{in case I}}\end{array}$}--(4,3);
\draw (5.5,0)node[below]{$\begin{array}{c} \frac{y}{m_0} \\ \textcolor{blue}{\shortparallel} \\ \textcolor{blue}{k_{s+1}l_s} \\ \textcolor{blue}{\mbox{in case II}}\end{array}$}--(5.5,1.5);
\draw (8,0)node[below]{$m_0^2k_sl_s$} --(8,3);
\draw (11,0)node[below]{$\begin{array}{c} y \\ \shortparallel \\k_{s+1}l_{s+1}\end{array}$} --(11,4.125);
\draw [domain=0:2,very thick] plot(\x,{-0.64786+1.4478*sqrt(\x+0.2)});
\draw [domain=2:11,very thick,red] plot(\x,{-0.58946+1.4087*sqrt(\x+0.2)});
\draw [domain=2:11,very thick,blue] plot(\x,{0.30602+0.80498*sqrt(\x+0.2)});
\node[color=red,thick] (T) at (8,4.1) {$\rho(x)$ and $\tilde{\rho}(x)$ in case I};
\node[color=blue,thick] (T) at (10,2) {$\rho(x)$ and $\tilde{\rho}(x)$ in case II};
\draw[very thick,red,dashed] (2,1.55)--(4,1.55);
\draw[very thick,red,dashed] (4,1.5)--(11,4.125);
\draw[very thick,blue,dashed] (2,1.5)--(5.5,1.5);
\draw[very thick,blue,dashed] (5.5,1.5)--(11,3);

\end{tikzpicture}

\par\end{centering}

\caption{\label{picture} Approximation of functions in $\mathcal{C}_{0,1}$ by functions piecewise constant and linear, as in the proof of Lemma \ref{lemma-approximation}.}
\end{figure}

\begin{proof} The proof is best understood looking at Figure \ref{picture}.
Observe that by construction, $\tilde{f}(x)$ is continuous and non-decreasing,
which is not necessarily true for $\bar{f}(x)$. By induction, assume
$k_{s},l_{s}$ already known with $\tilde{f}(k_{s}l_{s})=l_{s}=f(k_{s}l_{s})$.
The hypothesis on $f$ gives that $l_{s}\leq f(x)\leq\frac{x}{k_{s}}$
for all $x\geq k_{s}l_{s}$ .

We consider the minimal $y\geq m_{0}^{2}k_{s}l_{s}$ such that $m_{0}l_{s}\leq f(y)\leq\frac{y}{m_{0}k_{s}}$.
Assume first that such a $y$ exists. By continuity of $f$, there are two
cases.

Case I\emph{:} if $f(y)=\frac{y}{m_{0}k_{s}}$, set $k_{s+1}=m_{0}k_{s}$
and $l_{s+1}=\frac{y}{m_{0}k_{s}}\geq m_{0}l_{s}$, then for all $k_{s}l_{s}\leq x\leq k_{s+1}l_{s+1}=y$,
\[
\frac{x}{m_{0}k_{s}}\leq\tilde{f}(x)\leq\frac{x}{k_{s}},
\]
and the same inequalities hold for $f$ by sublinearity.

Case II\emph{:} if $f(y)=m_{0}l_{s}$, set $l_{s+1}=m_{0}l_{s}$ and
$k_{s+1}=\frac{y}{m_{0}l_{s}}\geq m_{0}k_{s}$, then for all $k_{s}l_{s}\leq x\leq k_{s+1}l_{s+1}=y$
\[
l_{s}\leq\tilde{f}(x)\leq m_{0}l_{s},
\]
and the same inequalities hold for $f$ (non-decreasing).

If such a $y$ does not exist, there are again two cases. Either $f(x)\geq\frac{x}{m_{0}k_{s}}$
for all $x\geq k_{s}l_{s}$, then set $k_{s+1}=m_{0}k_{s}$ and $l_{s+1}=\infty$
so $\tilde{f}(x)=\frac{x}{m_{0}k_{s}}$ for all $x\geq k_{s}l_{s}$,
generalizing case I. Or $f(x)\leq m_{0}l_{s}$ for all $x\geq k_{s}l_{s}$,
then set $k_{s+1}=\infty$ so $\tilde{f}(x)=l_{s}$ for all $x\geq k_{s}l_{s}$,
generalizing case II.

Concerning the function $\bar{f}$, a routine inspection (of the intervals
$[k_{s}l_{s},k_{s+1}l_{s}]$ and $[k_{s+1}l_{s},k_{s+1}l_{s+1}]$
in both cases) shows that $\tilde{f}(x)\leq\bar{f}(x)\leq2\tilde{f}(x)$
for any $x$.

To check the last statement, it is sufficient to check that $f(m_{0}l_{s}\exp((m_{0}l_{s})^{\frac{1}{\alpha_{0}}}))\geq m_{0}l_{s}$
so in case II we take $l_{s+1}=m_{0}l_{s}$, get $y\leq l_{s+1}\exp(l_{s+1}^{\frac{1}{\alpha_{0}}})$
and $k_{s+1}=\frac{y}{l_{s+1}}\leq\exp(l_{s+1}^{\frac{1}{\alpha_{0}}})$.
Our assumptions give 
\begin{eqnarray*}
f(l_{s+1}\exp(l_{s+1}^{\frac{1}{\alpha_{0}}})) & \geq & f(k_{s}l_{s})\left(\frac{\log(l_{s+1}\exp(l_{s+1}^{\frac{1}{\alpha_{0}}}))}{\log(k_{s}l_{s})}\right)^{\alpha}\\
 & \geq & l_{s}\left(\frac{(m_{0}l_{s})^{\frac{1}{\alpha_{0}}}}{\log k_{s}+\log l_{s}}\right)^{\alpha}\\
 & \geq & l_{s}m_{0}^{\frac{\alpha}{\alpha_{0}}}\left(\frac{l_{s}^{\frac{1}{\alpha_{0}}}}{l_{s}^{\frac{1}{\alpha_{0}}}+\log l_{s}}\right)^{\alpha},
\end{eqnarray*}
and the last parenthesis tends to $1$ as $l_{s}$ tends to infinity.

\end{proof}

\begin{proposition}\label{function-approximation} Let $C_{1}>0$
and $K,L\subset[1,\infty]$ such that for any $x$ in $[1,\infty]$,
there exists $k\in K$ and $l\in L$ with $k\simeq_{C_{1}}x$ and
$l\simeq_{C_{1}}x$. For any $f$ in $\mathcal{C}_{0,1}$ and for
any $m_{0}>1$, there exists two sequences $(k_{s}),(l_{s})$ taking
values in $K$ and $L$ respectively such that $k_{s+1}\geq m_{0}k_{s}$
and $l_{s+1}\geq m_{0}l_{s}$ for all $s$, the functions defined
in (\ref{tilde}) and (\ref{bar}) satisfy 
\[
\tilde{f}(x)\simeq_{m_{0}C_{1}^{5}}f(x)\mbox{ and }\bar{f}(x)\simeq_{2m_{0}C_{1}^{5}}f(x)
\]
Moreover if for some $\alpha>\alpha_{0}>0$ the function $\frac{f(x)}{\log^{\alpha}(x)}$
is non-decreasing, it is possible to find such functions with sequences
$(k_{s}),(l_{s})$ satisfying $\log k_{s}\leq l_{s}^{\frac{1}{\alpha_{0}}}$
for all $s$. \end{proposition}

\begin{proof} We apply Lemma \ref{lemma-approximation} with $m_{0}'>C_{1}^{2}m_{0}$,
and obtain two sequences $(k_{s}),(l_{s})$ of real numbers satisfying
$k_{s+1}\ge m_{0}C_{1}^{2}k_{s}$ and $l_{s+1}\ge m_{0}C_{1}^{2}l_{s}$.
The hypothesis on $K,L$ permits to find two sequences $(k_{s}'),(l_{s}')$
with $k_{s}'\simeq_{C_{1}}k_{s}$ and $l_{s}'\simeq_{C_{1}}l_{s}$.
The choice of $m_{0}'$ guarantees that $k_{s+1}'\ge m_{0}k_{s}'$
and $l_{s+1}'\ge m_{0}l_{s}'$.

Denote $\tilde{f}'$ and $\bar{f}'$ the functions defined by (\ref{tilde})
and (\ref{bar}) with the sequences $(k_{s}'),(l_{s}')$. It is sufficient
to check that $\tilde{f}'(x)\simeq_{C_{1}^{3}}\tilde{f}(x)$.

When $k'_{s}l'_{s}\le x\le k'_{s+1}l'_{s}$, then $\tilde{f}'(x)=l_{s}'\simeq_{C_{1}}l_{s}$.
On the other hand, $\frac{k_{s}l_{s}}{C_{1}^{2}}\leq x\leq C_{1}^{2}k_{s+1}l_{s}$
so 
\[
\frac{l_{s}}{C_{1}^{2}}=\frac{\tilde{f}(k_{s}l_{s})}{C_{1}^{2}}\le\tilde{f}(\frac{k_{s}l_{s}}{C_{1}^{2}})\le\tilde{f}(x)\le\tilde{f}(C_{1}^{2}k_{s+1}l_{s})\le C_{1}^{2}\tilde{f}(k_{s+1}l_{s})=C_{1}^{2}l_{s}.
\]
Thus $\tilde{f}'(x)\simeq_{C_{1}^{3}}\tilde{f}(x)$.

When $k'_{s+1}l_{s}'\le x\le k'_{s+1}l'_{s+1}$, set $x=\lambda k'_{s+1}l_{s}'+(1-\lambda)k'_{s+1}l'_{s+1}$.
Then 
\[
\tilde{f}'(x)=\lambda l_{s}'+(1-\lambda)l'_{s+1}\simeq_{C_{1}}\lambda l_{s}+(1-\lambda)l_{s+1}.
\]
On the other hand, $x\simeq_{C_{1}^{2}}\lambda k_{s+1}l_{s}+(1-\lambda)k_{s+1}l_{s+1}$
so $\tilde{f}(x)\simeq_{C_{1}^{2}}\tilde{f}(\lambda k_{s+1}l_{s}+(1-\lambda)k_{s+1}l_{s+1})=\lambda l_{s}+(1-\lambda)l_{s+1}$.
Thus $\tilde{f}'(x)\simeq_{C_{1}^{3}}\tilde{f}(x)$. \end{proof}

\begin{corollary}\label{approximation-speed} Let $C_{1}>0$ and
$K,L\subset[1,\infty]$ such that for any $x$ in $[1,\infty]$, there
exists $k\in K$ and $l\in L$ with $k\simeq_{C_{1}}x$ and $l\simeq_{C_{1}}x$.
For any $\varrho$ in $\mathcal{C}_{p_1,p_2}$ and for any $m_{0}>1$, there
exists two sequences $(k_{s}),(l_{s})$ taking values in $K$ and
$L$ respectively such that $k_{s+1}\geq m_{0}k_{s}$ and $l_{s+1}\geq m_{0}l_{s}$
for all $s$, the function defined by 
\[
\bar{\varrho}(x)=x^{p_{1}}l_{s}+\frac{x^{p_{2}}}{k_{s+1}}\mbox{ for }(k_{s}l_{s})^{\frac{1}{p_{2}-p_{1}}}\leq x\leq(k_{s+1}l_{s+1})^{\frac{1}{p_{2}-p_{1}}}
\]
satisfies 
\[
\bar{\varrho}(x)\simeq_{2m_{0}C_{1}^{5}}\varrho(x).
\]
Moreover if for any $\alpha>\alpha_{0}>0$ the function $\frac{\varrho(x)}{x^{p_{1}}\log^{\alpha}(x)}$
is non-decreasing, it is possible to find such functions with sequences
$(k_{s}),(l_{s})$ satisfying $\log k_{s}\leq l_{s}^{\frac{1}{\alpha_{0}}}$
for all $s$. \end{corollary}

\begin{proof} The application $T_{p_{1},p_{2}}:\mathcal{C}_{0,1}\rightarrow\mathcal{C}_{p_{1},p_{2}}$
given by $T_{p_{1},p_{2}}f(x)=x^{p_{1}}f(x^{p_{2}-p_{1}})$ is a bijection.
Take $f=T_{p_{1},p_{2}}^{-1}\varrho$, apply Proposition \ref{function-approximation}
and set $\bar{\varrho}=T_{p_{1},p_{2}}\bar{f}$. \end{proof}

\section{Stable walks on lamplighter over a segment}\label{stable}

Let $I_{m}$ be a subgraph of one dimensional lattice $\mathbb{Z}$
with vertex set $\left\{ 0,1,\ldots,m-1\right\} $, $m>1$. Consider
the lamplighter graph $\mathcal{L}_{m}$ over the segment $I_{m}$,
formally $\left\{ \left(f,x\right):\ f:I_{m}\to\{0,1\},\ x\in I_{m}\right\} $
is the vertex set of $\mathcal{L}_{m}$, and an edge connects $(f,x)$
and $(f',x')$ if $f\equiv f'$ and $x\sim x'$ in $I_{m}$, or $x=x'$,
$f\neq f'$ and $f$ differs from $f'$ only at site $x$. Random
walks on lamplighter graphs have been studied in \cite{Erschler2003,Peres2004},
see also references therein.

Random walk on $\mathbb{Z}$ driven by step distribution 
\[
\nu_{\alpha}(x)\simeq\frac{1}{\left(1+\left|x\right|\right)^{1+\alpha}},\ \alpha\in(0,2)
\]
is often referred to as a symmetric $\alpha$-stable like walks on
$\mathbb{Z}$. By abuse of terminology, we call it an $\alpha$-stable
walk. General theory regarding heat kernel estimates of $\alpha$-stable
walks on volume doubling graphs is available, see \cite{Chen2008}
and references therein. Connection between $\alpha$-stable walks
and Markov type method for bounding compression exponent was first
introduced in \cite{Naor2011}, where $p$-stable walk on $\mathbb{Z}$
is used to determine $L_{p}$-compression exponent of $\mathbb{Z}\wr\mathbb{Z}$,
$p>1$. For our purpose, we focus on the case of stable walk of index
$\alpha=1$ on the base graph.

On $I_{m}$, define transition kernel 
\[
\zeta_{m}(x,x')=\frac{c_{x,x'}}{\sum_{x'\in I_{m}}c_{x,x'}},\ x,x'\in I_{m},
\]
where 
\[
c_{x,x'}=\frac{1}{1+\left|x-x'\right|^{2}}.
\]
Then $\zeta_{m}$ is a Markov transition kernel on $I_{m}$ with stationary
distribution $\mathcal{C}_{m}(z)=\frac{\sum_{x'\in I_{m}}c_{x,x'}}{\sum_{x,x'\in I_{m}}c_{x,x'}}.$
One readily checks that $\frac{1}{5m}\le\mathcal{C}_{m}(x)\le\frac{5}{m}$
for all $x\in I_{m}$.

Now consider a random walk on the lamplighter graph $\mathcal{L}_{m}$
with transition $\zeta_{m}$ on the base. Let $\mathfrak{p}_{m}$
be the transition kernel in $\mathcal{L}_{m}$ such that for $x\neq x'$,
$\mathfrak{p}_{m}((f,x),(f',x'))=\frac{1}{4}\zeta_{m}(x,x')$ if $f(z)=f'(z)$
for all $z\notin\{x,x'\}$; for $x=x'$, $\mathfrak{p}_{m}((f,x),(f',x'))=\frac{1}{2}\zeta_{m}(x,x)$
if $f(z)=f'(z)$ for all $z\neq x$. In words, in each step the walker
first randomize the lamp configuration at the current location, then
moves according to the transition kernel $\zeta_{m}$, and randomize
the lamp at the arrival location. This Markov chain is reversible
with stationary distribution 
\[
U_{m}\left(f,x\right)=2^{-m}\mathcal{C}_{m}(x).
\]
From the upper bound on relaxation time in \cite[Theorem 1.2]{Peres2004},
we have 
\[
T_{rel}(\mathcal{L}_{m},\mathfrak{p}_{m})=\frac{1}{\lambda_{2}(\mathcal{L}_{m},\mathfrak{p}_{m})}\le\max_{x,y\in I_{m}}\mathbf{E}_{x}\tau_{y},
\]
where $\tau_{y}=\min\left\{ t:\ X_{t}=y\right\} $, $X_{t}$ denotes
the Markov chain on $\mathcal{L}_{m}$ with transition kernel $\mathfrak{p}_{m}$.
Note that although in the statement of \cite[Theorem 1.2]{Peres2004},
it is assumed that the Markov chain on the base is lazy simple random
walk on a transitive graph, the coupling argument that proves the
relaxation time upper bound is completely general, it applies to any
reversible Markov chain on the base graph. The quantity $\max_{x,y\in I_{m}}\mathbf{E}_{x}\tau_{y}$
is known as the maximal hitting time of the chain $\mathfrak{p}_{m}$.
By \cite[Lemma 4.1]{Aldous2002}, 
\[
\max_{x,y\in I_{m}}\mathbf{E}_{x}\tau_{y}=\frac{1}{2}\left(\sum_{z,z'\in I_{m}}c_{z,z'}\right)\max_{x,y}R_{x,x'},
\]
where $R_{x,x'}$ denotes the effective resistance between vertices
$x$ and $x'$ in the electric network on $I_{m}$ with edge conductances
$c_{z,z'}$ between pair of vertices $z,z'$. Estimates of effective
resistance of $\alpha$-stable walks follow from classical methods.
For the particular transition kernel $\zeta_{m}$ with $\alpha=1$,
there exists a constant $C>0$ 
\[
R_{x,x'}\le C\log\left|x-x'\right|,
\]
see for example \cite[Appendix B.2]{Caputo2009}. We conclude that
for the Markov chain with transition kernel $\mathfrak{p}_{m}$ on
$\mathcal{L}_{m}$, 
\[
\lambda_{2}(\mathcal{L}_{m},\mathfrak{p}_{m})\ge\frac{c}{m\log m}.
\]
Equivalently, we have the following Poincaré inequality: for any function
$f:\mathcal{L}_{m}\to\mathbb{R}$, 
\[
\sum_{u,v\in\mathcal{L}_{m}}\left(f(u)-f(v)\right)^{2}U_{m}(u)U_{m}(v)\le\frac{2m\log m}{c}\sum_{u,v\in\mathcal{L}_{m}}\left(f(u)-f(v)\right)^{2}\mathfrak{p}_{m}(u,v)U_{m}(u).
\]
By Matousek's extrapolation lemma for Poincaré inequalities, see \cite{Matousek}
and \cite[Lemma 4.4]{Naor2011a}, we deduce the following $l_{p}$-Poincaré
inequalities.

\begin{lemma}\label{lamplighter-poincare}

In the setting introduced above, there exists an absolute constant
$C>0$ such that for any $f:\mathcal{L}_{m}\to \ell^{p}$, 
\begin{itemize}
\item if $1\le p\le2$, 
\[
\sum_{u,v\in\mathcal{L}_{m}}\left\Vert f(u)-f(v)\right\Vert _{p}^{p}U_{m}(u)U_{m}(v)\le Cm\log m\sum_{u,v\in\mathcal{L}_{m}}\left\Vert f(u)-f(v)\right\Vert _{p}^{p}U_{m}(u)\mathfrak{p}_{m}\left(u,v\right);
\]

\item if $p>2$, 
\[
\sum_{u,v\in\mathcal{L}_{m}}\left\Vert f(u)-f(v)\right\Vert _{p}^{p}U_{m}(u)U_{m}(v)\le\left(Cm\log m\right)^{\frac{p}{2}}(2p)^{p}\sum_{u,v\in\mathcal{L}_{m}}\left\Vert f(u)-f(v)\right\Vert _{p}^{p}U_{m}(u)\mathfrak{p}_{m}\left(u,v\right).
\]

\end{itemize}
\end{lemma}

Now we introduce a distance function on $\mathcal{L}_{m}$. Let $\mathbf{w}=(w_{1},w_{2})\in\mathbb{R}_{2}^{+}$,
in the lamplighter graph $\mathcal{L}_{m}$, let $w(e)=w_{1}$ if
the edge $e$ connects $(f,x)$ and $(f,x')$ where $x\sim x'$ (edges
of first type); $w(e)=w_{2}$ if the edge $e$ connects $(f,x)$ and
$(f',x)$ where $f,f'$ only differs at site $x$ (edges of second
type). Define $d_{\mathbf{w}}$ to be distance on $\mathcal{L}_{m}$
\[
d_{\mathbf{w}}(u,v)=\min\left\{ \sum_{e\in P}w(e):\ P\mbox{ is a path in }\mathcal{L}_{m}\mbox{ connecting }u,v\right\} .
\]
From definition of $\mathfrak{p}_{m}$, it is straightforward to check
that the following.

\begin{lemma}\label{C2}

There exists constant $C>0$ such that for all $p>1$,

$\boldsymbol{}$ 
\begin{align*}
\sum_{u,v\in\mathcal{L}_{m}}d_{\mathbf{w}}(u,v)^{p}U_{m}(u)\mathfrak{p}_{m}\left(u,v\right) & \le C\left(w_{2}^{p}+\frac{1}{p-1}w_{1}^{p}m^{p-1}\right),\\
\sum_{d_{\mathbf{w}}(u,v)\ge\frac{1}{4}\left(w_{2}+w_{1}\right)m}d_{\mathbf{w}}(u,v)^{p}U_{m}(u)U_{m}(v) & \ge\frac{1}{C}\left(w_{2}+w_{1}\right)m.
\end{align*}

\end{lemma}

These ingredients allow us to carry out the Poincaré inequality method
to upper bound $L_{p}$-compression function of $H\wr\mathbb{Z}$,
it can also be used in the study of the diagonal product $\Delta$
with dihedral groups. Alternatively, we may apply the Markov type
method. To this end, the following speed lower estimate is needed.

\begin{lemma}\label{cauchy-speed}

Let $X_{t}$ be a stationary Markov chain on $\mathcal{L}_{m}$ with
transition kernel $\mathfrak{p}_{m}$ reversible with stationary measure
$U_{m}$. Then there exists $c>0$ such that 
\[
\mathbf{E}_{U_{m}}\left[d_{\mathbf{w}}\left(X_{t},X_{0}\right)\right]\ge c\left(w_{1}+w_{2}\right)\frac{t}{\log_{\ast}t}\mbox{ for all }1\le t\le m\log m.
\]

\end{lemma}

\begin{proof}

Let $S_{[0,t]}=\left\{ S_{n},0\le n\le t\right\} $ denote the sites
visited by the induced random walk $\{S_{t}\}$ on $I_{m}$. Since
in each step the chain randomizes the lamp at the current and new
locations, and any path in the graph $\mathcal{L}_{m}$ that connects
$X_{0}$ to $X_{t}$ must visit all the sites where the lamp configurations
of $X_{0}$ and $X_{t}$ differ, we have for any $u\in\mathcal{L}_{m}$,
\[
\mathbf{E}_{u}\left[d_{\mathbf{w}}\left(X_{t},X_{0}\right)\right]\ge\frac{1}{2}(w_{1}+w_{2})\mathbf{E}_{u}\left[\left|S_{[0,t]}\right|\right].
\]
Thus the question is reduced to the range of the $\zeta_{m}$-random walk
on the base $I_{m}$. Methods to estimate expected size of range of
random walk go back to Dvoretzky and Erdos \cite{Dvoretzky1951}.
Here we include a straightforward adaptation of the argument in \cite[Lemma 6.3]{Naor2008}
for completeness.

In what follows $\mathbb{E}$ means taking expectation with the law
of random walk $S_{n}$ on $I_{m}$ with step distribution $\zeta_{m}$.
For any $k\in\left\{ 1,\ldots,m\right\} $, denote by $V_{1},\ldots,V_{k}$
the first $k$ elements of $I_{m}$ that are visited by the random
walk $S_{n}$. Let 
\[
Y_{k}(t)=\left|\left\{ 0\le n\le t :S_{n}\in\left\{ V_{1},\ldots V_{k}\right\} \right\} \right|.
\]
Note that $\left\{ Y_{k}(t)<t+1\right\} =\left\{ \left|S_{[0,t]}\right|>k\right\} $.
For any starting point $z\in I_{m}$, 
\[
\mathbf{\mathbb{E}}_{z}\left[Y_{k}(t)\right]=\sum_{l=1}^{k}\mathbf{\mathbb{E}}_{z}\left[\left|\left\{ 0\le n\le t :S_{n}=V_{l}\right\}\right|\right]\le k\sum_{n=0}^{t}\max_{x\in I_{m}}\mathbf{\mathbb{P}}_{x}\left(S_{n}=x\right).
\]
Therefore 
\begin{align}
\mathbb{E}_{z}\left(\left|S_{[0,t]}\right|\right) & \ge k\mathbb{P}_{z}\left(\left|S_{[0,t]}\right|>k\right)\ge k\left(1-\frac{\mathbf{\mathbb{E}}_{z}\left[Y_{k}(t)\right]}{t+1}\right)\nonumber \\
 & \ge k\left(1-\frac{k\sum_{n=0}^{t}\max_{x\in I_{m}}\mathbf{\mathbb{P}}_{x}\left(S_{n}=x\right)}{t+1}\right).\label{eq:size-range}
\end{align}
The argument in \cite[Theorem 3.1 ]{Chen2008} implies that the chain
$\left(I_{m},\zeta_{m}\right)$ satisfies a Nash inequality that there
exists an absolute constant $C>0$, 
\[
\theta\left(\left\Vert u\right\Vert _{2}^{2}\right)\le C\mathcal{E}_{\zeta_{m}}\left(u,u\right)\ \mbox{for all }u:I_{m}\to\mathbb{R},\ \mbox{where }\theta(r)=r^{2}.
\]
This Nash inequality implies on-diagonal decay upper bound, see \cite{Diaconis1996},
\[
\mathbb{P}_{x}\left(S_{l}=x\right)\le\frac{c_{2}}{l}\mbox{ for all }l\le m,\ x\in I_{m}.
\]
For $l>m$, by monotonicity $\mathbf{P}_{x}\left(S_{l}=x\right)\le\mathbf{P}_{x}\left(S_{m}=x\right)\le c_{2}l^{-1}$.
It follows that for $k\in\{1,\ldots,m\}$, 
\[
\sum_{l=0}^{k\log k}\max_{x\in I_{m}}\mathbb{P}_{x}\left[S_{l}=x\right]\le c_{3}\log k.
\]
Together these estimates imply that 
\[
\mathbf{E}_{u}\left[d_{\mathbf{w}}\left(X_{k\log_{\ast}k},X_{0}\right)\right]\ge c\left(w_{1}+w_{2}\right)k\mbox{ for any }k\in\left\{ 1,\ldots,m\right\} ,u\in\mathcal{L}_{m}.
\]

\end{proof}

The same method can be used to estimate speed of random walks on lamplighter
graphs over other choices of the base graph.

\begin{lemma}\label{lamplighter-torus}

Let $\Gamma=\mathbb{Z}_{2}\wr D_{\infty}^{d}$, $d\ge3$ as in the
2nd item of Example \ref{example-2}, marked with generating subgroups
$A=\mathbb{\mathbb{Z}}_{2}\wr\left\langle a_{j},1\le j\le d\right\rangle $,
$B=\mathbb{\mathbb{Z}}_{2}\wr\left\langle b_{j},1\le j\le d\right\rangle $.
Fix an increasing sequence $n_{s}\in\mathbb{N}$, let $\Gamma_{s}=\mathbb{Z}_{2}\wr D_{2n_{s}}^{d}$ be a finite quotient of $\Gamma$. Let $A(s)$,
$B(s)$ denote the projection of $A$ and $B$ to $\Gamma_{s}$. There
exists a constant $\sigma_{d}>0$ only depending on $d$ such that $\left\{ \Gamma_{s}\right\} $
satisfies $\left(\sigma_{d},\left(2n_{s}\right)^{d}\right)$-linear
speed assumption.

\end{lemma}

\begin{proof}

Let $\bar{A}(s)=\left\langle a_{j},1\le j\le d\right\rangle $, $\bar{B}(s)=\left\langle b_{j},1\le j\le d\right\rangle $.
Consider a random alternating word $X_{t}^{(s)}$ in $A(s)$ and $B(s)$
of length $t$, let $\bar{X}_{t}^{(s)}$ be its projection to $D_{2n_{s}}^{d}$.
In words, if the last letter in $X_{t}^{(s)}$ is a random element
in $A(s)$, then to get to $X_{t+1}^{(s)}$, the lamp configuration
in the neighborhood $\bar{X}_{t}^{(s)}\bar{B}(s)$ is randomized,
and the walker on the base $D_{2n_{s}}^{d}$ is multiplied by a random
element in $\bar{B}(s)$. Similarly if $X_{t}^{(s)}$ ends with $B(s)$,
then the next move is uniform in $A(s)$. From this description, we
have that the lamps over the sites visited by $\bar{X}_{t}^{(s)}$
are randomized, 
\[
\left|X_{2t}^{(s)}\right|_{\Gamma_{s}}\ge\frac{1}{8}\left|\mathcal{R}_{[[0,2t]]}^{(s)}\right|,
\]
where $\mathcal{R}_{[[0,2t]]}^{(s)}=\left\{ x\in D_{2n_{s}}^{d}:\bar{X}_{2l}^{(s)}=x\mbox{ for some }0\le l\le t\right\} $.
By comparing $\left\{ \bar{X}_{2l}^{(s)}\right\} $ to standard simple
random walk on $D_{2n_{s}}^{d}$, we have that there exists a constant
$C_{d}>0$ such that 
\[
\mathbf{P}\left(\bar{X}_{2l}^{(s)}=e\right)\le C_{d}(2l)^{-\frac{d}{2}}\mbox{ for all }1\le l\le2n_{s}^{2}.
\]
It follows that 
\begin{align*}
\sum_{l=0}^{(2n_{s})^{d}}\mathbf{P}\left(\bar{X}_{2l}^{(s)}=e\right) & \le1+\sum_{l=1}^{2n_{s}^{2}}C_{d}(2l)^{-\frac{d}{2}}+\left(2n_{s}\right)^{d}(4n_{s}^{2})^{-d/2}\\
 & \le2C_{d}+2.
\end{align*}
To estimate $\mathbf{E}\left|\mathcal{R}_{[[0,2t]]}^{(s)}\right|,$
we apply the argument as in Lemma \ref{cauchy-speed}. For $t\le\left(2n_{s}\right)^{d}$,
in the inequality (\ref{eq:size-range}) choose $k=\frac{t}{4C_{d}+4}$,
we conclude that there exists constant $\sigma_{d}>0$, 
\[
\E\left|X_{2t}^{(s)}\right|_{\Gamma_{s}}\ge\sigma_{d}t.
\]
\end{proof}

 \bibliographystyle{alpha}
\bibliography{CompressionBib}

\textsc{\newline J\'er\'emie Brieussel \newline Universit\'e de Montpellier 
\newline Institut Montpelli\'erain Alexander Grothendieck  \newline 34095 Montpellier, France} \newline
\textit{E-mail address:} jeremie.brieussel@umontpellier.fr

\textsc{\newline Tianyi Zheng \newline Stanford University \newline Department of Mathematics \newline Stanford, California 94305} \newline
\textit{E-mail address:} tzheng2@stanford.edu

\end{document}